\documentclass[12pt, twoside]{article}
\usepackage{amssymb,amsmath}
\usepackage[amsmath,thmmarks,thref]{ntheorem}
\usepackage{graphics}
\usepackage{enumerate}
\usepackage{dsfont}
\usepackage{booktabs}
\usepackage{multirow}
\usepackage{gn-logic14}
\usepackage{mathrsfs}
\input{diagxy}
\xyoption{curve}
\newbox\anglebox % large pullback angle
\setbox\anglebox=\hbox{\xy \POS(75,0)\ar@{-} (0,0) \ar@{-} (75,75)\endxy}
 \def\pbangle{\copy\anglebox}
\newbox\angleboxr % reverse large pullback angle
\setbox\angleboxr=\hbox{\xy \POS(0,0)\ar@{-} (0,75) \ar@{-} (75,0)\endxy}
 \def\pbangler{\copy\angleboxr}
\newbox\sanglebox % small pullback angle
\setbox\sanglebox=\hbox{\xy \POS(50,0)\ar@{-} (0,0) \ar@{-} (50,50)\endxy}
 \def\spbangle{\copy\sanglebox}
\newbox\sangleboxr % small reverse pullback angle
\setbox\sangleboxr=\hbox{\xy \POS(0,0)\ar@{-} (0,50) \ar@{-} (50,0)\endxy}
 
\newbox\sangleboxf % small flipped pullback angle
\setbox\sangleboxf=\hbox{\xy \POS(50,50)\ar@{-} (50,0) \ar@{-} (0,50)\endxy}
 
\newbox\angleboxf % flipped pullback angle
\setbox\angleboxf=\hbox{\xy \POS(75,75)\ar@{-} (75,0) \ar@{-} (0,75)\endxy}
 \def\pbanglef{\copy\angleboxf}
\newbox\sangleboxfr % small flipped reverse pullback angle
\setbox\sangleboxfr=\hbox{\xy \POS(0,50)\ar@{-} (50,50) \ar@{-} (0,0)\endxy}
 
\newbox\angleboxfr % small flipped reverse pullback angle
\setbox\angleboxfr=\hbox{\xy \POS(0,75)\ar@{-} (75,75) \ar@{-} (0,0)\endxy}

\newdir{|>}{!/4.7pt/\dir{|}
        *:(1,-.2)\dir^{>}
        *:(1,+.2)\dir_{>}}
\def\cover{\to/-|>/}
\def\subobject{\ \to/|>->/}
\def\embedd{\to/^{ (}->/}
\newcommand{\pair}[1]{\ensuremath{\langle {#1} \rangle}}
\newcommand{\homset}[3]{\ensuremath{\operatorname{Hom}_{#1}\!\left({#2},{#3}\right)}}

\newcommand{\mb}{\ensuremath{\mathbin}}

\newcommand{\cat}[1]{\ensuremath{\mathcal{#1}}}
\newcommand{\thry}[1]{\ensuremath{\mathbb{#1}}}
\newcommand{\synt}[2]{\ensuremath{\mathcal{#1}_{\mathbb{#2}}}}
\newcommand{\alg}[1]{\ensuremath{\mathbf{#1}}}
\newcommand{\mng}[1]{\ensuremath{\mathrm{#1}}}
\newcommand{\modcat}[1]{\ensuremath{\mathrm{Mod}_{#1}}}

\newcommand{\modin}[2]{\ensuremath{\mathrm{Mod}_{#1}(#2)}}
\newcommand{\topo}[1]{\ensuremath{\mathscr{#1}}}
\newcommand{\classtop}{\ensuremath{\mathbf{Set}[\theory]}}
\newcommand{\Sets}{\ensuremath{\mathbf{Sets}}}

\newcommand{\op}[1]{\ensuremath{{#1}^\mathrm{op}}}

\newcommand{\sh}[1]{\protect\ensuremath{\operatorname{Sh}(\mathcal{#1})}}
\newcommand{\Sh}[1]{\protect\ensuremath{\operatorname{Sh}\!\left(#1\right)}}

\newcommand{\Form}{\protect\ensuremath{\operatorname{Form}}}
\newcommand{\Mod}{\protect\ensuremath{\operatorname{Mod}}}

\newcommand{\sem}[1]{\ensuremath{[\![{#1}]\!]}}
\newcommand{\csem}[2]{\ensuremath{[\![{#1}\mb|{#2}]\!]}}% withcontext

\newcommand{\fins}[1]{\exists {#1}\mathpunct .}
\newcommand{\finst}[2]{\exists {#1}\mathord :{#2}\mathpunct .}

\newcommand{\theory}{\ensuremath{\mathbb{T}}}
\newcommand{\cterm}[2]{\ensuremath{\left \{ {#1}\ \; \vrule \; \ {#2}\right \}}}
\newcommand{\syntob}[2]{\ensuremath{[{#1}\;|\;{#2}]}}
\newcommand{\funksjon}[2]{\ensuremath{\operatorname{#1} \left( #2 \right)}}

\newcommand{\Eqsheav}[2]{\protect\ensuremath{\operatorname{Sh}_{#1}(#2)}}
\newcommand{\sox}[2]{\ensuremath{\csem{#1}{#2}_{X_{\theory}}}}
\newcommand{\bopen}[1]{\ensuremath{\langle\!\! \langle {#1} \rangle\!\! \rangle}}
\title{First-Order Logical Duality}
\author{Steve Awodey\\
\small Department of Philosophy\\
\small Carnegie Mellon University\\
\small Pittsburgh, PA 15217\\
\small USA\\
\small awodey@cmu.edu
\and Henrik Forssell\thanks{Corresponding author. Present address: Dept.\ of Mathematics, Stockholm University, SE - 106 91 Stockholm, Sweden.}\\
\small Department of Philosophy\\
\small Carnegie Mellon University\\
\small Pittsburgh, PA 15217\\
\small USA\\
\small henrikforssell@gmail.com
}

\makeindex %%
\begin{document}%%%%%%%%%%%%%%%%%%%%%%%%%%%%%%%%%%%%%%%%%%%%%%%%%%%%%%%%%%%%%%%%%%%%%%%%%%%%%%%%%%%%%%%%%%%%%%%%%%%
%%%%%%%%%%%%%%%%%%%%%%%%%%%%%%%%%%%%%%%%%%%%%%%%%%%%%%%%%%%%%%%%%%%%%%%%%%%%%%%%%%%%%%%%%%%%%%%%%%%%%%%%%%%%%%%%%%%
%%%%%%%%%%%%%%%%%%%%%%%%%%%%%%%%%%%%%%%%%%%%%%%%%%%%%%%%%%%%%%%%%%%%%%%%%%%%%%%%%%%%%%%%%%%%%%%%%%%%%%%%%%%%%%%%%%%
%
%
\maketitle

\begin{abstract}
\noindent From a logical point of view, Stone duality for Boolean algebras relates theories in classical propositional logic and their collections of models.  The theories can be seen as presentations of Boolean algebras, and the collections of models can be topologized in such a way that the theory can be recovered from its space of models.  The situation can be cast as a formal duality relating two categories of syntax and semantics, mediated by homming into a common dualizing object, in this case $2$.

In the present work, we generalize the entire arrangement from propositional to first-order logic, using a representation result of Butz and Moerdijk.  Boolean algebras are replaced by Boolean categories presented by theories in first-order logic, and spaces of models are replaced by topological groupoids of models and their isomorphisms.  A duality between the resulting categories of syntax and semantics, expressed primarily in the form of a contravariant adjunction, is established by homming into a common dualizing object, now $\Sets$, regarded once as a boolean category, and once as a groupoid equipped with an intrinsic topology.

The overall framework of our investigation is provided by topos theory.  Direct proofs of the main results are given, but the specialist will recognize toposophical ideas in the background.  Indeed, the duality between syntax and semantics is really a manifestation of that between algebra and geometry in the two directions of the geometric morphisms that lurk behind our formal theory.  Along the way, we give an elementary proof of Butz and Moerdijk's result in logical terms.


\noindent \textbf{AMS classification codes:} 03G30; 18B25; 18C10; 18C50.
\end{abstract}
\tableofcontents
%
%\clearpage
%
\section*{Introduction}
\label{section: Introduction}
We present an extension of Stone duality for Boolean algebras from classical propositional logic to classical
first-order logic. In broad strokes, the leading idea is to take the traditional logical distinction between
syntax and semantics and analyze it in terms of the classical mathematical distinction between algebra and
geometry, with syntax corresponding to algebra and semantics to geometry. Insights from category theory allow
us to recognize a certain duality between the notions of algebra and geometry. We see a first glimpse of this in
Stone's duality theorem for Boolean algebras, the categorical formulation of which states that a category of
`algebraic' objects (Boolean algebras) is the categorical dual of a category of `geometrical' objects (Stone
spaces). ``Categorically dual'' means that the one category is opposite to the other, in that it can be obtained (up to equivalence)
from the other by formally reversing the morphisms. In a more far reaching manner, this form of algebra-geometry
duality is exhibited in modern algebraic geometry as reformulated in the language of schemes in the Grothendieck
school, e.g.\ in the duality between the categories of commutative rings and the category of affine schemes.

On the other hand, we are informed by the category theoretic analysis of logic that it is
closely connected with algebra, in the sense that logical theories can be regarded as categories and suitable
categories can be presented as logical theories. For instance, Boolean algebras can be seen as classical propositional
theories, categories with finite products can be seen as equational theories, Boolean coherent categories
as theories in classical first-order logic, and elementary toposes -- e.g.\ the topos of sheaves on a space -- as theories in higher-order
intuitionistic logic. Thus the study of these algebraic objects has a logical interpretation and, vice versa,
reasoning in or about logical theories has application in their corresponding algebraic objects. With the
connection between algebra and logic in hand, instances of the algebra-geometry duality can be seen to manifest
a syntax-semantics duality between an algebra of syntax and a geometry of semantics. This notion of syntax as `dual to semantics' is, expectedly, one which ignores presentation and other features which, so to speak, models can not distinguish.  In the propositional case, one passes from a propositional theory to a Boolean algebra by constructing the Lindenbaum-Tarski algebra of the theory, a construction which identifies provably equivalent formulas (and orders them by provable implication). Thus any two complete theories, for instance, are `algebraically equivalent' in the sense of having isomorphic Lindenbaum-Tarski algebras.  The situation is precisely analogous to a presentation of an algebra by generators and relations: a logical theory corresponds to such a presentation, and two theories are equivalent if they present `the same' -- i.e.\ isomorphic -- algebras.
A similar construction is used to obtain, for  a classical first-order theory, its `corresponding' Boolean coherent category, resulting in a similar notion of algebraic or categorical equivalence.

Given this connection between formal theories and categories, Stone duality manifests a syntax-semantics duality for propositional logic as follows. While a Boolean algebra can be regarded as a propositional theory modulo `algebraic' equivalence, on the other hand a Stone space can be seen as a space of corresponding two-valued models of such a theory. A model of a propositional theory is of course just a valuation of the propositional letters, or equivalently, a Boolean homomorphic valuation of all formulas. Thus we obtain the set of models of the theory corresponding to a Boolean algebra by taking morphisms in the category of Boolean algebras from the given algebra into the two-element Boolean algebra, 2,
\begin{eqnarray}\label{eq:stonerep}
\modcat{\cat{B}}\cong\homset{\alg{BA}}{\cat{B}}{2}.
\end{eqnarray}
And with a suitable topology in place---given in terms of the elements of the Boolean algebra $\cat{B}$---we can retrieve $\cat{B}$ from the space of models $\modcat{\cat{B}}$ by taking morphisms in the category of Stone spaces from it into the two-element Stone space, 2,
\begin{eqnarray*}
\cat{B} \cong \homset{\alg{Stone}}{\modcat{\cat{B}}}{2}
\end{eqnarray*}
Here, the two-element set, 2, is in a sense  living a `dual' life, and `homming into 2' forms a contravariant adjunction between the `syntactical' category of Boolean algebras and the category of topological spaces,
which, moreover, becomes an equivalence once we restrict to the `semantical' subcategory of Stone spaces.
\[\bfig
\morphism|b|/{@{>}@/_1em/}/<1250,0>[\alg{BA}`\alg{Stone};\homset{\alg{BA}}{-}{2}]
\morphism/{@{<-}@/^1em/}/<1250,0>[\alg{BA}`\alg{Stone};\homset{\alg{Stone}}{-}{2}]
\place(625,0)[\simeq]
\efig\]

Our construction for first-order logic generalizes this set-up by, on the `syntax' side, representing first-order theories by Boolean coherent categories. On the semantical side we have, for each theory, a space of models, augmented with a space consisting of the isomorphisms between those models, such that these spaces form a \emph{topological groupoid}, that is to say, such that the composition, domain and codomain, inverse arrow and identity arrow maps are all continuous. Our `semantic' side is, accordingly, a category consisting of topological groupoids and continuous homomorphisms between them. Where in Stone Duality one considers the lattice of open sets of a space in order to recover a Boolean algebra, we consider the topos (or `generalized space') of so-called \emph{equivariant sheaves} on a topological groupoid in order to recover a Boolean coherent category. In particular, we show that the topos of equivariant sheaves on the topological groupoid of models and isomorphisms of a theory is the so-called classifying topos of (the Morleyization of) the theory, from which it is known that the theory can be recovered up to a notion of equivalence. Here we build upon earlier results in \cite{butz:96} and \cite{butz:98b} to the effect that any such topos can be represented by a topological groupoid constructed from its points. Our construction differs from the one given there mainly in presentation, but also in choosing a simpler cover which is better suited for our purpose.

%a very particular kind of topos; not only can we recover the theory from it (up to a suitable sense of equivalence), but it contains a universal topos model of the theory. Thus we obtain a groupoid representation of the so-called %\emph{classifying topos} of the theory in terms of its models and isomorphisms. We hasten here, first, to add that for these statements about topos models to make proper sense one should regard the theory as Morleyized; and %second, to acknowledge our debt to previous work on  groupoid representation of toposes, in particular  \cite{joyal:84}, \cite{moerdijk:89}, and \cite{butz:98b}.

The semantic representation of this topos can also be understood from the perspective of definable sets.
%
%The question of to what extent a theory can be recovered from its models is, of course, and old one, and several `syntax-semantics' dualities have been constructed, both for fragments of and for full, classical first-order logic. %Not having the space to give these the introduction they deserve, we limit ourselves to briefly indicating a central aspect in which our construction is related to other approaches.
%
Suppose we have a theory, \theory, in first order logic or some fragment of it, and that $\phi(\alg{x})$ is some formula in the language of the theory. Then $\phi(\alg{x})$ induces a definable set functor,
\[
\sem{\phi(\alg{x})} : \modcat{\theory}\to \Sets
%
%\morphism(50,-300)/|->/[\alg{M}`\phi(\alg{M});]
%
\]
from the groupoid of \theory-models to the category of sets, which sends a model \alg{M} to the extension, $\sem{\phi(\alg{x})}^{\alg{M}}$, of $\phi(\alg{x})$ in \alg{M}. The question is, then, whether these definable set functors can somehow be characterized among all functors of the form $\modcat{\theory}\rightarrow\Sets$, so that the theory can be recovered from its models in terms of them. Notice, incidentally, that in case of a positive answer, the category of sets takes on the role of a dualizing object, in analogy with \alg{2} for Stone duality. For the models of a theory can be seen as suitable functors from the algebraic representation of the theory, \synt{C}{T}, into \Sets, so that both obtaining the models from the theory and recovering the theory from the models is done by `homming' into \Sets,
\begin{align*}
\modcat{\theory} &\simeq \homset{\!}{\synt{C}{T}}{\Sets}\\
\synt{C}{T} &\simeq\homset{\!}{\modcat{\theory}}{\Sets}
\end{align*}
Here the hom-sets must be suitably restricted from all functors to just those preserving the relevant structure, the determination of which is part of the task at hand.

Now, positive, and elegant, answers to the question  of the characterization of definable set functors exist, to begin with, for certain fragments of first-order logic. For algebraic theories---axiomatized only by equations in languages with only function symbols (and equality)---the categories of models (algebras) have all limits and colimits, and \emph{Lawvere duality} tells us that an algebraic theory \theory\ can be recovered (up to splitting of idempotents) from its category of models in the form of those functors $\modcat{\theory}\to\Sets$ which preserve limits, filtered colimits, and regular epimorphisms (see \cite{lawvere:63},\cite{adameklawvererosicky:03}).  Expanding from the algebraic case, recall, e.g.\ from \cite[D1.1.]{elephant1}, that the \emph{Horn} formulas over a first-order signature are those formulas which are constructed using only connectives $\top$ and $\wedge$. Allowing also existential quantification brings us to \emph{regular} formulas. A Horn (regular) theory is one which can be axiomatized using sequents involving only Horn (regular) formulas. In between, a Cartesian theory is a regular theory which can be axiomatized using only formulas that are Cartesian relative to the theory, in the sense, briefly, that existential quantification does not occur except under a certain condition.  Now, the category \modcat{\theory} of models and homomorphisms of a Cartesian theory \theory\ has limits and filtered colimits (but not, in general, regular epis), and \emph{Gabriel-Ulmer duality} (see e.g.\ \cite{adamekandrosicky:94})
%\marginpar{refer to G-U's lecture notes from 1971 instead? no, this suffices.}
informs us, among other things, that the definable set functors for  Cartesian formulas (relative to \theory) can be characterized as the limit and filtered colimit preserving functors $\modcat{\theory}\rightarrow\Sets$   (and that the theory can be recovered in terms of them).
%and that there is a syntax-semantics duality for Cartesian theories between (small) categories with finite limits (representing Cartesian theories) and locally finitely presentable categories (representing their categories of %models and model homomorphisms).
If we allow for unrestricted existential quantification and pass to regular logic, then categories of models need no longer have arbitrary limits. But they still have products and filtered colimits, and, as shown by M. Makkai \cite{makkai:90}, the definable set functors for regular formulas can now be characterized as those functors $\modcat{\theory}\rightarrow\Sets$ that preserve precisely that.

Adding the connectives $\bot$ and $\vee$ to regular logic gives us the fragment known as \emph{coherent logic} (see \cite[D1.1.]{elephant1}), in which a far greater range of theories can be formulated. The theory of fields, for instance, cannot be expressed as a regular theory (since the category of fields does not have arbitrary products), but it can be expressed as a coherent theory (see \cite[D1.1.7.(h)]{elephant1}). (In fact, it is a  \emph{decidable} coherent theory, where ``decidable'' means, here, that there is an \emph{inequality predicate}, in the sense of a coherent formula which is provably the complement of equality.) Moreover, any classical first-order theory can be \emph{Morleyized} to yield a coherent theory with the same category of models, see \cite[D1.5.13]{elephant1} (we take the morphisms between models of a classical first-order theory to be the elementary embeddings). Thus the categories of models of coherent theories can not, in general, be expected to have more structure than those for classical first-order theories. What they do have are ultra-products.  Although ultra-products are not an intrinsic feature of categories of models (for coherent theories), in the sense that they are not a categorical invariant,  Makkai \cite{makkai:87} shows that model categories and the category of sets can be equipped with a notion of ultra-product structure---turning them into so-called \emph{ultra-categories}---which allows for the characterization of definable set functors as those functors that preserve this additional structure. Moreover, this approach can be modified in  the case of classical first-order theories so that only the ultra-\emph{groupoids} of models and isomorphisms, equipped with ultra-product structure, need be considered, see \cite{makkai93}.

Our approach, similarly, relies on equipping the models of a theory with external structure, but in our case the structure is topological. The background for this choice of structure is the representation result of Butz and Moerdijk \cite{butz:98b} to the effect that any topos with enough points can be represented as a topos of equivariant sheaves on a topological groupoid constructed from the points of the topos and geometric transformations between them. This means that for any geometric theory (see below) with enough models---that is, such that sequents true in all (ordinary) models of the theory are provable in the theory---its so-called classifying topos can be represented as equivariant sheaves on a topological groupoid constructed from models of the theory and isomorphisms between them. Since a geometric theory can be recovered up to Morita equivalence from its classifying topos, this means that it can be so recovered from its models and isomorphisms by equipping them with topological structure.
%Moreover, this structure is, in fact, a natural generalization of that which works in the classical propositional case.

Geometric logic is obtained by adding the connective $\bigvee$ (infinite disjunction) to coherent logic (see \cite[D1.1.]{elephant1}). Thus the class of geometric theories subsumes that of (Morleyized) first-order theories, so that Butz and Moerdijk's theorem applies to first-order theories. Moreover, restricting to the latter  allows for a simplification of the topology employed. We give a thorough presentation of this (simplified) topology from a logical perspective, that is, in terms only of theories and their models. We then give a direct and elementary proof for Butz and Moerdijk's representation theorem in this setting, with very little topos theoretic machinery, and we give a direct description of the generic topos model of a theory in the process.    We carry this construction out for  decidable coherent theories, corresponding to (small) decidable coherent categories (``decidable'' meaning, in the categorical setting, that diagonals are complemented). As we remarked, the theory of fields is a notable example of such a theory, and the class of decidable coherent theories  includes all classical first-order theories in the sense that the Morleyization of a classical theory is decidable coherent. Accordingly, our construction restricts to the classical first-order case, corresponding to Boolean coherent theories.

The first part of the paper (Section \ref{Section: Reprentation Theorem}) contains the elementary proof of the (adjusted) representation theorem, involving  the characterization of definable set functors for a theory and the recovery of the theory from its groupoid of models in terms of them. The upshot is that definable sets can be characterized as being, in a sense, \emph{compact}; not by regarding each individual set as compact, but by regarding the definable set functor as being a compact object in a suitable category. Pretend, for a moment, that the models of a theory \theory\ form a set and not a proper class, and suppose, for simplicity, that the models are all disjoint. A definable set functor from the groupoid of \theory-models and isomorphisms,
%the (-) removed, but doesn't it fit well with the notation in the lines following it?
%
\[\sem{\phi(\alg{x})}^{(-)} : \modcat{\theory}\to \Sets\]
can, equivalently, be considered as
%a model-indexed set,
%
%\[\left(\sem{\phi(\alg{x})}^{(\alg{M})}\right)_{\alg{M}\vDash\theory}\]
%
a set (indexed) over the set $(\modcat{\theory})_0$ of models,
\begin{equation}\label{eq:sheafrepproj}
\coprod_{\alg{M}\vDash\theory}\sem{\phi(\alg{x})}^{\alg{M}}\to^p (\modcat{\theory})_0
\end{equation}
with $p^{-1}(\alg{M})=\sem{\phi(\alg{x})}^{\alg{M}}$, together with an action on this set by the set $(\modcat{\theory})_1$ of isomorphisms,
%
%\[\bfig
%%
%\morphism<1500,0>[\left(\modcat{\theory}\right)_1\times_{(\modcat{\theory})_0}\coprod_{\alg{M}\vDash\theory}\sem{\phi(\alg{x})}^{\alg{M}}`\coprod_{\alg{M}\vDash\theory}\sem{\phi(\alg{x})}^{\alg{M}};\alpha]
%%
%\efig\]
%
% this display seems better because it leaves the coproducts in displaystyle,
% although it would be even better if the arrow where longer...
%
\begin{equation}\label{eq:sheafrepact}
\left(\modcat{\theory}\right)_1\times_{(\modcat{\theory})_0}\coprod_{\alg{M}\vDash\theory}\sem{\phi(\alg{x})}^{\alg{M}}
\to^\alpha \coprod_{\alg{M}\vDash\theory}\sem{\phi(\alg{x})}^{\alg{M}}
\end{equation}
%
%such that for any \theory-model isomorphism, $\alg{f}:\alg{M}\rightarrow\alg{N}$, and element, $\alg{m}\in \sem{\phi(\alg{x})}^{\alg{M}}$, we have $\alpha(\alg{f},\alg{m})=\alg{f}(\alg{m})\in \sem{\phi(\alg{x})}^{\alg{N}}$.  Now, if the set of \theory-models and the set of isomorphisms are topological spaces forming a topological groupoid, then we can ask for the collection $\coprod_{\alg{M}\vDash\theory}\sem{\phi(\alg{x})}^{\alg{M}}$ of definable sets to be a space in such a way that the  function, $p:\coprod_{\alg{M}\vDash\theory}\sem{\phi(\alg{x})}^{\alg{M}}\to (\modcat{\theory})_0$, is a local homeomorphism, and such that the action, $\alpha$, is continuous. This makes definable set functors into equivariant sheaves on the groupoid, and we show that in the topos of all such sheaves they can be characterized as the compact decidable objects (up to a suitable notion of equivalence).
%
such that for any \theory-model isomorphism, $\alg{f}:\alg{M}\rightarrow\alg{N}$, and element, $\alg{m}\in \sem{\phi(\alg{x})}^{\alg{M}}$, we have $\alpha(\alg{f},\alg{m})=\alg{f}(\alg{m})\in \sem{\phi(\alg{x})}^{\alg{N}}$.  Now, if the set of \theory-models and the set of isomorphisms are topological spaces forming a topological groupoid, then we can ask for the collection $$\coprod_{\alg{M}\vDash\theory}\sem{\phi(\alg{x})}^{\alg{M}}$$ of elements of the various definable sets to be a space, in such a way that the projection function $p$ in \eqref{eq:sheafrepproj} is a local homeomorphism, and such that the action $\alpha$ in \eqref{eq:sheafrepact} is continuous. This makes definable set functors into equivariant sheaves on the groupoid, and we show that in the topos of all such sheaves they can be characterized as the compact decidable objects (up to a suitable notion of equivalence).

The second part (Section  \ref{Section: Duality}) concerns the construction, based on the representation result of the first part, of a duality between the category of decidable coherent categories (representing theories in first-order logic) and the category of topological groupoids of models. This takes the form of an adjunction between the category of decidable coherent categories and a category of `weakly coherent' topological groupoids, such that the counit component of the adjunction is an equivalence, up to pretopos completion.  Accordingly, the adjunction restricts to a duality, in a suitable 2-categorical sense, between decidable pretoposes and groupoids of models. (We otherwise do not expand of the 2-categorical aspects of the construction, preferring for a simpler presentation to fix certain choices `on the nose' instead.) We give a characterization of groupoids of models up to Morita equivalence. In line with the Butz-Moerdijk representation result, we introduce a size restriction both on theories and their models (corresponding to the pretence, above, that the collection of models of a theory forms a set). The restriction, given a theory, to a set of models large enough for our purposes can be thought of as akin to the fixing of a `monster' model for a complete theory, although in our case a much weaker saturation property is asked for, and a modest cardinal bound on the size of the models is sufficient.

In summary, we use the prior representation result of Butz and Moerdijk to give a new extension of Stone duality to first-order logic, in the form of a `syntax-semantics' adjunction between decidable coherent categories and `weakly coherent' topological groupoids, with counit components at pretoposes being equivalences. The construction differs from Makkai's \cite{makkai:87,makkai93} in  using topological structure and sheaves instead of structure based on ultra-products, and in restricting to classical Stone duality in the propositional case. Similar to Makkai's construction, it uses the category of sets as a dualizing object, in this case connecting it with the (decidable) object classifier in topos theory. The representation result of Butz and Moerdijk is given a new presentation and elementary proof from a logical perspective. In particular, the topology is simplified and presented in terms only of the signature of the theory. This results in a new description of the decidable object classifier as equivariant sheaves on the groupoid of sets equipped with a simple topology.

This paper is based on \cite{phd} where more details can be found.

% add boolean object classifier, run through replacing CohGpd with wcGpd, add all the stuff needed to characterize wcGpd?

%
%
%
%
%
\section{Topological Representation}%%%%%%%%%%%%%%%%%%%%%%%%%%%%%%%%%%%%%%%%%%%%%%%%%%%%%%%%%%%%%%%%%%%%%%%%%%%%%%%%%%%%%%%%%%%%%%%%%%%%%%%%%%%%%%%%%%%%%%%%%%%%%%%%%%%%%%%%%%%%%%%%%%%%%%%%%%%%%%%%
\label{Section: Reprentation Theorem}%%%%%%%%%%%%%%%%%%%%%%%%%%%%%%%%%%%%%%%%%%%%%%%%%%%%%%%%%%%%%%%%%%%%%%%%%%%%%%%%%%%%%%%%%%%%%%%%%%%%%%%%%%%%%%%%%%%%%%%%%%%%%%%%%%%%%%%%%%%%%%%%%%%%%%%%%%%%%%%%%%%%%%%%%%%%%%%%%%%%%%%%%%%

We show how to recover a classical, first-order theory from its groupoid of models and model-isomorphisms, bounded in size and equipped with topological structure, using the representation theorem of Butz and Moerdijk \cite{butz:98b}.  We present this from a logical perspective, that is, from the
perspective of the syntax and model theory of first-order theories, and we give a self-contained and elementary proof of Butz and Moerdijk's theorem in this setting (while also somewhat simplifying the construction). One can, of course, go back and forth between the logical perspective and the categorical perspective of (Boolean) coherent categories and
set-valued coherent functors. Section \ref{Section: Duality} briefly outlines the translation between the two, and presents a duality between the `syntactical' category of theories and a `semantical' category of model-groupoids.
We start the current section with some preliminaries from categorical logic and topos theory.

\subsection{Coherent Theories, Syntactic Categories, and Classifying Toposes}%%%%%%%%%%%%%%%%%%%%%%%%%%%%%%%%%%%%%%%%%%%%%%%%%%%%%%%%%%%%%%%%%%%%%%%%%%%%%%%%%%%%%%%%%%%%%%%%%%%%%%%%%%%%%%%%%%%%%%%%%%%%%%%%%%%%%%%%%%%%%%%%%%%%%%%
\label{Subsection: Theories and Models}

Let $\Sigma$ be a (first-order, possibly many-sorted) signature. Recall that a formula over $\Sigma$ is \emph{coherent} if it is constructed using only the connectives $\top$, $\wedge$, $\exists$, $\bot$, and $\vee$. We consider formulas in suitable contexts, \syntob{\alg{x}}{\phi}, where the context $\alg{x}$ is a list of distinct variables containing (at least) the free variables of $\phi$. A sequent, $\phi\vdash_{\alg{x}}\psi$---where $\alg{x}$ is a suitable context for both $\phi$ and $\psi$---is coherent if both $\phi$ and $\psi$ are coherent. Henceforth we shall not be concerned with axiomatizations, and so we consider a (coherent) \emph{theory} to be a deductively closed set of (coherent) sequents.

Let \theory\ be a coherent (alternatively first-order) theory over a signature, $\Sigma$. Recall that the \emph{syntactic category}, \synt{C}{T}, of
\theory\ has as objects equivalence classes of coherent (alt.\ first-order) formulas in context, e.g.\ \syntob{\alg{x}}{\phi}, which is
equivalent to a formula in context, \syntob{\alg{y}}{\psi}, if the contexts are $\alpha$-equivalent and \theory\
proves the formulas equivalent\footnote{See \cite[D1]{elephant1} for further details. Note that we, unlike
\cite{elephant1}, choose to identify \theory-provably equivalent formulas. The reason is that they define
exactly the same sets, i.e.\ the same definable set functors.}, i.e.\ \theory\ proves the following sequents.
\begin{align*}
\phi & \vdash_{\alg{x}}\, \psi[\alg{x}/\alg{y}]\\
\psi[\alg{x}/\alg{y}] & \vdash_{\alg{x}}\, \phi
\end{align*}
An arrow between two objects, say \syntob{\alg{x}}{\phi} and \syntob{\alg{y}}{\psi} (where we may assume that $\alg{x}$ and $\alg{y}$ are distinct), consists of a class of
\theory-provably equivalent formulas in context, say \syntob{\alg{x},\alg{y}}{\sigma}, such that \theory\ proves
that $\sigma$ is a functional relation between $\phi$ and $\psi$:
\begin{align*}
\sigma  &\vdash_{\alg{x},\alg{y}}\, \phi\wedge\psi\\
\phi  &\vdash_{\alg{x}}\, \fins{\alg{y}}\sigma\\
\sigma\wedge \sigma(\alg{z}/\alg{y})  &\vdash_{\alg{x},\alg{y},\alg{z}}\, \alg{y}=\alg{z}
\end{align*}
If \theory\ is a coherent theory, then \synt{C}{T} is a coherent category. If \theory, in addition, has an
inequality predicate (for each sort), that is, a formula with two free variables (of that sort), $x\neq y$, such that \theory\ proves
\begin{align*}\label{Equation: DC, Axioms of inequality}
x\neq y \wedge x=y &\vdash_{x,y}\, \bot\\
\top &\vdash_{x,y}\, x\neq y \vee x=y
\end{align*}
%
%$\psi(x\!:\!A,y\!:\!A)$ such that \theory\ proves
%%
%\begin{align*}\label{Equation: DC, Axioms of inequality}
%\psi(x,y) \wedge x=y &\vdash_{x,y} \bot\\
%\top &\vdash_{x,y} \psi(x,y) \vee x=y
%\end{align*}
%
then \synt{C}{T} is  \emph{decidable}, in the sense that for each object, $A$, the diagonal, $\Delta:A\mon A\times A$, is complemented as a subobject. We call a
coherent theory which has an inequality predicate (for each sort) a \emph{decidable} coherent theory for that reason (and with
apologies for overloading the term). Finally, if \theory\ is a first-order theory, then \synt{C}{T} is a Boolean
coherent category, i.e. a coherent category such that every subobject is complemented.

Conversely, given a coherent category, \cat{C}, one can construct the coherent \emph{theory,
$\theory_{\cat{C}}$, of} \cat{C} by having a sort for each object and a function symbol for each arrow, and
taking as axioms all sequents which are true under the canonical interpretation of this language in \cat{C} (again, see
\cite{elephant1} for details). A coherent decidable category allows for the construction of a coherent decidable
theory (including an inequality predicate for each sort), and Boolean coherent \cat{C} allows for the
construction of a first-order $\theory_{\cat{C}}$. Thus we can turn theories into categories and categories back
into theories. It is in this sense that we say that (decidable) coherent categories represent (decidable)
coherent theories, and Boolean coherent categories represent first-order theories. (Since Boolean coherent categories are, of course, coherent,  building the Boolean coherent syntactical category of a classical first-order theory and then taking its  coherent internal theory will produce a decidable coherent theory with the same models as the original classical one; thus yielding an alternative, but less economical, way of \emph{Morleyizing} a classical theory than the one presented in \cite[D1.5.13]{elephant1}.)
We show how to recover a
theory from its models in the sense that we recover its syntactic
category, up to pretopos completion. Roughly, the pretopos completion of a theory is the theory equipped with disjoint sums and quotients
of equivalence relations,  see e.g.\ \cite{makkai93}. A theory and its pretopos completion have the same models in (the pretopos) \Sets. Furthermore, since the construction works for coherent decidable theories and these subsume classical first-order theories in the sense above, we carry this out for coherent decidable theories (leaving classical first-order as a special case).
%
%\begin{remark}
%There a ref to the Morleyization given in the elephant in the introduction, but obtaining a coherent theory via the syntactic category results in a different presentation, which is why I wrote ``can, loosely, be thought of'' %instead of ``can be described as''.
%\end{remark}

The category of models and homomorphisms of a coherent theory \theory\ is equivalent to the category of coherent functors from \synt{C}{T} into the category \Sets\ of sets and functions and natural transformations between them,
\[\modcat{\theory}{}\simeq \homset{}{\synt{C}{T}}{\Sets}\]
and the same holds for models in an arbitrary coherent category, \cat{E},
\[\modcat{\theory}{(\cat{E})}\simeq \homset{}{\synt{C}{T}}{\cat{E}}\]
Indeed, this is the universal property that characterizes ${\synt{C}{T}}$. The same is true for classical first-order theories if ``homomorphism'' is replaced by ``elementary embedding'' (Note that the elementary embeddings between models of a classical first-order theory coincide with the homomorphisms between models of its Morleyization.) We pass freely between considering models traditionally as structures and algebraically as functors. In passing, we note that decidability for coherent theories can be characterized semantically:

\begin{lemma}\label{Lemma: Semantic decidability}
Let \theory\ be a coherent theory over a signature $\Sigma$, and \modcat{\theory} the category of \theory-models and homomorphisms. Then \theory\ is decidable (i.e.\ has an inequality predicate for each sort) if and only if for every \theory-model homomorphism, $\alg{f}:\alg{M}\rightarrow\alg{N}$ and every sort $A$ of $\Sigma$, the component function $f_A:\sem{A}^{\alg{M}}\rightarrow \sem{A}^{\alg{N}}$ is injective.
\begin{proof}This follows from a slight rewriting of the proof of \cite[D3.5.1]{elephant1}.
\end{proof}
\end{lemma}

Given a coherent theory \theory\ there exists a so-called \emph{classifying topos}, \classtop, defined by the universal property that the category of \theory-models in any topos \cat{E} is equivalent to the category of geometric morphisms from \cat{E} to \classtop\ and geometric transformations between them (see \cite[D3]{elephant1}),
\[\modin{\theory}{\cat{E}}\simeq\homset{\!}{\cat{E}}{\classtop }\]
The \theory-model in \classtop\ corresponding to the identity geometric morphism $\classtop\to\classtop$ is called the generic model. One can construct \classtop\ `syntactically' by equipping the category \synt{C}{T} with the Grothedieck coverage, $J$, defined by finite epimorphic families,
\[\Sh{\synt{C}{T},J}\simeq\classtop\]
in which case the generic \theory-model is given by the Yoneda embedding
\[y:\synt{C}{T}\to \Sh{\synt{C}{T},J}.\]
%We do not, here, consider \synt{C}{T} equipped with any other coverage than the coherent one, and so we simply write \Sh{\synt{C}{T}} for this topos.
%
We show, for decidable coherent theories and with an adjustment of the representation theorem of \cite{butz:98b}, how to represent \classtop\ as the topos of so-called equivariant sheaves on a topological groupoid of \theory-models, and with the generic \theory-model given by definable set functors.
\subsection{Groupoids of Models}%%%%%%%%%%%%%%%%%%%%%%%%%%%%%%%%%%%%%%%%%%%%%%%%%%%%%%%%%%%%%%%%%%%%%%%%%%%%%%%%%%%%%%%%%%%%%%%%%%%%%%%%%%%%%%%%%%%%%%%%%%%%%%%%%%%%%%%%%%%%%%%%%%%%%%%%%%%%%%%%%%%%%%%%%%%%%%%%%%%%%%%%%
\label{Subsection: Groupoids of models}%%%%%%%%%%%%%%%%%%%%%%%%%%%%%%%%%%%%%%%%%%%%%%%%%%%%%%%%%%%%%%%%%%%%%%%%%%%%%%%%%%%%%%%%%%%%%%%%%%%%%%%%%%%%%%%%%%%%%%%%%%%%%%%%%%%%%%%%%%%%%%%%%%%%%%%%%%%%%%%%%%%%%%%%%%%%%%%%%%
Fix a decidable coherent theory \theory\ over a (possibly many-sorted) signature, $\Sigma$. Fix an infinite set $\mathds{S}$ of cardinality $\kappa \geq |\Sigma|$ of `sets' or `elements'. The set $\mathds{S}$ should be thought of as the `universe' from which we construct $\Sigma$-structures, and can be taken (if one prefers) to be a (sufficiently large) initial segment of the set-theoretic universe.  Let $X_{\theory}$ be the set of all $\theory$-models with elements from $\mathds{S}$, and let $G_{\theory}$ be the set of isomorphisms between them. We think of choosing $\mathds{S}$ and $X_{\theory}$ as akin to choosing a monster model, and refer to models in $X_{\theory}$ as $\kappa$\emph{-small}.  We write finite lists or tuples of variables,  elements, and sorts in boldface, $\alg{a}=\pair{a_1,\ldots,a_n}$, with $*$ the concatenation operator, $\alg{a}*\alg{b}=\pair{a_1,\ldots,a_n, b_1,\ldots,b_m}$, and $\star$ the empty list. We also use boldface to indicate models, \alg{M},\alg{N}, and model isomorphisms, \alg{f},\alg{g}, with the component function of \alg{f} at a sort $A$ written $f_A$. We often leave the typing of variables implicit.
\begin{definition}\label{Definition> Logical topology}
The \emph{logical topology} on $X_{\theory}$ is the coarsest topology containing all sets of the following form:
\begin{enumerate}
\item For each sort, $A$, and element, $a\in \mathds{S}$, the set \[\bopen{A,a}:=\cterm{\alg{M}\in X_{\theory}}{a\in \sem{A}^{\alg{M}}}\]
\item For each relation symbol, $R:A_1,\ldots,A_n$, and $n$-tuple, $\alg{a}\in \mathds{S}$, of elements, the set \[\bopen{R,\alg{a}}:=\cterm{\alg{M}\in X_{\theory}}{\alg{a}\in \sem{R}^{\alg{M}}\subseteq \sem{A_1}^{\alg{M}}\times\ldots\times\sem{A_n}^{\alg{M}}}\]
    \textup{(}This extends to nullary relations symbols; if $R$ is a nullary relation symbol, then  \[\bopen{R,\star}=\cterm{\alg{M}\in X_{\theory}}{\alg{M}\models R}\] and $\bopen{R,a}=\emptyset$ for $a\neq\star$.\textup{)}
\item For each function symbol, $f:A_1,\ldots,A_n\rightarrow B$, and n+1-tuple of elements, $\alg{a}* b=\pair{a_1,\ldots,a_n,b}$, the set \[\bopen{f(\alg{a})=b}:=\cterm{\alg{M}\in X_{\theory}}{\sem{f}^{\alg{M}}(\alg{a})=b}\] \textup{(}This extends in the obvious way to include nullary function symbols, i.e.\ constants\textup{)}.
\end{enumerate}
Let $G_{\theory}$ be the set of isomorphisms between the models in $X_{\theory}$, with domain or \emph{surce} and codomain or \emph{target} functions $s,t: G_{\theory}\rightrightarrows X_{\theory}$. The \emph{logical topology} on $I_{\Sigma}$ is the coarsest such that both $s$ and $t$ are continuous and containing all sets of the following form:
\begin{enumerate}[(i)]
\item For each sort, $A$, and pair of elements, $a,b\in \mathds{S}$, the set
 \[\bopen{A,a\mapsto b}:=\cterm{\alg{f}\in G_{\theory}}{a\in \sem{A}^{s(\alg{f})}\ \textnormal{and}\ f_{A}(a)=b}\]
\end{enumerate}
\end{definition}
Note that the logical topology is defined in terms of the signature $\Sigma$, not the logical formulas over $\Sigma$. It is, however, convenient to also note that a \emph{basic} open set of $M_{\Sigma}$ or $X_{\theory}$ can be presented in the form
\begin{equation}\label{Equation: Basic open set}\bopen{\syntob{\alg{x}}{\phi},\alg{a}}=\cterm{\alg{M}\in X_{\theory}}{\alg{a}\in \csem{\alg{x}}{\phi}^{\alg{M}}}\end{equation}
where \syntob{\alg{x}}{\phi} is a Horn formula and $\alg{a}\in \thry{S}$. A straightforward induction on formulas shows that for any geometric formula, $\phi$, with free variables in $\alg{x}$, the set defined by (\ref{Equation: Basic open set}) is open. We write this out for reference.
\begin{Lemma}\label{Lemma: The logical topologies are all the same}
Sets of the form
$\bopen{\syntob{\alg{x}}{\phi},\alg{a}}$
form a basis for the logical topology on $X_{\theory}$, with $\phi$ ranging over all coherent formulas over $\Sigma$.
\end{Lemma}
Similarly, a basic open set of $G_{\theory}$ can be presented, by stating a \emph{source}, a \emph{preservation}, and a \emph{target condition}, in the form
%\[d^{-1}(\bopen{\syntob{\alg{x}}{\phi},\alg{a}})\cap\bopen{\alg{b} \mapsto \alg{c}}\cap c^{-1}(\bopen{\syntob{\alg{y}}{\psi},\alg{d}})\]
%
%\\
\begin{equation}\label{Equation: V array}
\left(\begin{array}{c}
\syntob{\alg{x}}{\phi},\alg{a}  \\
\alg{z}:\alg{b} \mapsto \alg{c}   \\
 \syntob{\alg{y}}{\psi},\alg{d}
\end{array}\right)
=s^{-1}(\bopen{\syntob{\alg{x}}{\phi},\alg{a}})\bigcap\bopen{\alg{z}:\alg{b} \mapsto \alg{c}}\bigcap t^{-1}(\bopen{\syntob{\alg{y}}{\psi},\alg{d}})
\end{equation}
where $\phi$ and $\psi$ are coherent formulas, and where \alg{z}, as an (implicitly) typed list of variables stands for a list of sorts.
% We think of such a presentation of a basic open set as having a \emph{domain}, a \emph{preservation}, and a \emph{codomain condition.}
%
\begin{remark}\label{Remark: The topology for classical theories}
If \theory\ is a classical first-order theory, then Lemma \ref{Lemma: The logical topologies are all the same} should be taken to be the definition of the logical topology, with ``coherent formula'' replaced by ``first-order formula''.  Definition \ref{Definition> Logical topology} would apply, but to the Morleyization of \theory, and hence to a larger signature. This is the only point where the fact that we consider classical first-order theories to be (implicitly) Morleyized makes a noteworthy difference.
\end{remark}
\begin{remark}\label{Remark: The topology compared to Butz and Makkai}
If constructed along the lines of \cite{butz:98b}, the space $X_{\theory}$ of \theory-models would consist of models the underlying sets of which are quotients of subsets of $\mathds{S}$ with infinite equivalence classes, and with the changes to Definition \ref{Definition> Logical topology} that this would entail. For instance, the open set \bopen{x=y, a,b} would be the set of models in which $a$ and $b$ occurs in the same equivalence class (hence the open set would not be empty even if $a\neq b$, whereas in our set up it would be). As such the space of models is similar to the space of `term models' of Ch.\ 6 of \cite{makkaireyes} (where it is attributed to \cite{hakim:72}).
\end{remark}

Recall that a groupoid is a category where all arrows are invertible, and that a topological groupoid is a groupoid object in the category \alg{Sp} of topological spaces and continuous maps, i.e.\ a two spaces $G$ and $X$ of `arrows' and `objects', and with continuous  `source', `target', `identity', `inverse', and `composition'  maps
\[\bfig
%
%\morphism|a|/@{>}@<5pt>/<750,0>[G\times_{X} G`G;\mng{l}]
%
\morphism<750,0>[G\times_{X} G`G;\mng{c}]
%
%\morphism|b|/@{>}@<-5pt>/<750,0>[G\times_{X} G`G;\mng{r}]
%
\morphism(750,0)|a|/@{>}@<5pt>/<750,0>[G`X;\mng{s}]
\morphism(750,0)|m|/@{<-}/<750,0>[G`X;\mathrm{Id}]
\morphism(750,0)|b|/@{>}@<-5pt>/<750,0>[G`X;\mng{t}]
\Loop(750,0)G(ur,ul)_{\mng{i}}
%\place(750,100)[\curvearrowleft^i]
%
\efig\]
satisfying the usual axioms. Recall that a topological groupoid is said to be \emph{open} if the source and target maps are open. We verify (cf \cite{butz:98b}) that $G_{\theory}$ and $X_{\theory}$ form an open topological groupoid, denoted by $\thry{G}_{\theory}$. Note that if we are presented with a basic open set
\[\bopen{\syntob{\alg{y}:\alg{B}}{\phi},\alg{b}} \subseteq X_{\theory}\ \]
we can assume without loss of generality that, for $i\neq j$, $B_i=B_j$ implies $b_i\neq b_j$. We say that
$\bopen{\syntob{\alg{y}}{\phi},\alg{b}}$ is presented in \emph{reduced form} if this condition is
satisfied. It is clear that, as long as we are careful, we can replace elements in a model by switching to an isomorphic model. We write this out as a technical lemma for reference.
\begin{lemma}\label{Lemma: Star of David} Let  a list of sorts $\alg{A}$ of
\theory\ and two tuples $\alg{a}$ and $\alg{b}$ of $\mathds{S}$ be given, of the same length as $\alg{A}$,
and with both tuples being \emph{sortwise distinct} in the sense that whenever  $i\neq j$, $A_i=A_j$ implies $a_i\neq a_j$ and $b_i\neq b_j$. Then
for any $\alg{M}\in X_{\theory}$, if $\alg{a}\in \csem{\alg{x}:\alg{A}}{\top}^{\alg{M}}$, there exists an
$\alg{N}\in X_{\theory}$ and an isomorphism $\alg{f}:\alg{M}\rightarrow \alg{N}$ in $G_{\theory}$ such that
$f_{\alg{A}}(\alg{a})= \alg{b}$.
%\begin{proof} Let $\mathfrak{S}$ be the set of sorts of \theory.
% For each $S\in \mathfrak{S}$, we can choose a set, $c_S$, in
%$\Sets_{\kappa}$ not in the transitive closure of $\{ b_1,\ldots,b_n \}\cup\sem{S}^{\alg{M}}$. Construct sets
%$\sem{S}^{\alg{N}'}$ and (the obvious) bijections $\sem{S}^{\alg{N}'} \cong \sem{S}^{\alg{M}}$ by setting
%$\sem{S}^{\alg{N}'}=\sem{S}^{\alg{M}}\times \{c_S\}$ for all $S\in\mathfrak{S}$. Next, construct sets
%$\sem{S}^{\alg{N}}$ and bijections $\sem{S}^{\alg{N}} \cong \sem{S}^{\alg{N}'}$ by setting
%$\sem{S}^{\alg{N}}=\sem{S}^{\alg{N}'}$ if $S\not\in \alg{A}$, and by setting $\sem{A_i}^{\alg{N}}$ to be the set
%$\sem{A_i}^{\alg{N}'}$ with any element $\pair{a_i,c_{A_i}}\in \sem{A_i}^{\alg{N}'}$ replaced by $b_i$.
%
%Then for all $S\in\mathfrak{S}$, we have a bijection  $f_S:\sem{S}^{\alg{M}}\rightarrow \sem{S}^{\alg{N}}$.
%These bijections induce a \theory-model $\alg{N}$ with an isomorphism $\alg{f}:\alg{M}\rightarrow \alg{N}$ such
%that $f_{\alg{A}}(\alg{a})=\alg{b}$.
%
%\end{proof}
\end{lemma}
\begin{proposition}\label{Proposition: Multisort, groupoid is open}
The groupoid $\thry{G}_{\theory}$
\[\bfig
%
%\morphism|a|/@{>}@<5pt>/<750,0>[G\times_{X} G`G;\mng{l}]
%
\morphism<750,0>[G_{\theory}\times_{X_{\theory}} G_{\theory}`G_{\theory};c]
%
%\morphism|b|/@{>}@<-5pt>/<750,0>[G\times_{X} G`G;\mng{r}]
%
\morphism(750,0)|a|/@{>}@<5pt>/<750,0>[G_{\theory}`X_{\theory};s]
\morphism(750,0)|m|/@{<-}/<750,0>[G_{\theory}`X_{\theory};Id]
\morphism(750,0)|b|/@{>}@<-5pt>/<750,0>[G_{\theory}`X_{\theory};t]
\Loop(750,0)G_{\theory}(ur,ul)_{\mng{i}}
%\place(750,100)[\curvearrowleft^i]
%
\efig\]
is an open topological groupoid.
\begin{proof}
It is straightforward to verify that the displayed maps are continuous. We show that the source map is open, from which
it follows that the target map is open as well. Let a basic open subset
\[
V=\left(\begin{array}{c}
\syntob{\alg{x}:\alg{A}}{\phi},\alg{a}  \\
  \alg{B}:\alg{b} \mapsto \alg{c}   \\
 \syntob{\alg{y}:\alg{D}}{\psi},\alg{d}
\end{array}\right)
\]
of $G_{\theory}$ be given, and suppose $\alg{f}:\alg{M}\rightarrow \alg{N}$ is in $V$. We must find an open
neighborhood around \alg{M} which is contained in $s(V)$. We claim that
\[
U=\bopen{\syntob{\alg{x}:\alg{A},\alg{y}:\alg{D},\alg{z}:\alg{B}}{\phi\wedge\psi},\alg{a}\ast
f_{\alg{D}}^{-1}(\alg{d})\ast \alg{b}}
\]
does the trick. Clearly, $\alg{M}\in U$. Suppose $\alg{K}\in U$. Consider the tuples $
f_{\alg{D}}^{-1}(\alg{d})\ast \alg{b}$ and $\alg{d}\ast \alg{c}$ together with the list of sorts $\alg{D}\ast
\alg{B}$. Since $f_{\alg{D}\ast \alg{B}}$ sends the first tuple to the second, we can assume that the conditions
of Lemma \ref{Lemma: Star of David} are satisfied (or a simple rewriting will see that they are), and so there
exists a \theory-model \alg{L} and an isomorphism $\alg{g}:\alg{K}\rightarrow \alg{L}$ such that $g\in V$. So
$U\subseteq s(V)$.
\end{proof}
\end{proposition}
We end by noting that the spaces $G_{\theory}$ and $X_{\theory}$ are sober (recall from e.g.\ \cite{johnstone:82} that a space is sober if the completely prime filters of open sets of the space are precisely the neighborhood filters of points of the space). First, a basic open of $X_{\theory}$ ($G_{\theory}$) gives a partial, finite bit of information about the models (isomorphisms) in it. A completely prime filter of open sets gives complete information about a model (isomorphism), so we have:
\begin{Proposition}\label{Proposition: Sobriety} $X_{\theory}$ and $G_{\theory}$ are sober spaces.
\begin{proof}
We provide a sketch for the $X_{\theory}$ case. $G_{\theory}$ is similar. Let a completely prime filter, $F$, of open subsets of $X_{\theory}$ be given. For each sort $A$ of $\Sigma$, set $\sem{A}\subseteq \mathds{S}$ to be the set
\[\sem{A}:=\cterm{a\in \mathds{S}}{\bopen{A, a}\in F}\]
Interpret a relation symbol, $R$, as the set
\[\sem{R}:=\cterm{\alg{a}\in \mathds{S}}{\bopen{R,\alg{a}}\in F}\]
(note that $\bopen{R,\alg{a}}\subseteq \bopen{A, a}$, where $A$ is the appropriate sort) and a function symbol, $f$, as the function
\[\alg{a}\mapsto b\ \ \Leftrightarrow\ \ \bopen{f(\alg{a})=b}\in F\]
We show that this does indeed define a function. Suppose that $f:A\rightarrow B$ is a unary (for simplicity) function symbol of the indicated type. If $a\in \sem{A}$ then $\bopen{A, a}\in F$, and
\[\bopen{A, a}\subseteq \bopen{\syntob{x}{\fins{y}f(x)=y},a}=\bigcup_{b\in \mathds{S}}\bopen{\syntob{x,y}{f(x)=y},a,b}\]
whence there exists a $b\in\mathds{S}$ such that $\bopen{\syntob{x,y}{f(x)=y},a,b}=\bopen{f(a)=b}\in F$. Moreover, $\bopen{f(a)=b}\subseteq\bopen{B,b}$ and $\bopen{f(a)=b}\cap\bopen{f(a)=c}=\emptyset$ for $b\neq c$. We have defined, therefore, a $\Sigma$-structure the elements of which are elements of $\mathds{S}$. By a straightforward induction on $\syntob{\alg{x}}{\phi}$, we have that
\[\alg{M}\vDash \phi(\alg{a})\ \ \Leftrightarrow\ \ \bopen{\syntob{\alg{x}}{\phi}, \alg{a}}\in F\]
If \theory\ proves the sequent $\phi\vdash_{\alg{x}}\psi$ then $\bopen{\syntob{\alg{x}}{\phi},\alg{a}}\subseteq\bopen{\syntob{\alg{x}}{\psi},\alg{a}}$ for all $\alg{a}\in \mathds{S}$, and so \alg{M} is a \theory-model.
\end{proof}
\end{Proposition}
\subsection{Equivariant sheaves on topological groupoids}%%%%%%%%%%%%%%%%%%%%%%%%%%%%%%%%%%%%%%%%%%%%%%%%%%%%%%%%%%%%%%%%%%%%%%%%%%%%%%%%%%%%%%%%%%%%%%%%%%%%%%%%%%%%%%%%%%%%%%%%%%%%%%%%%%%%%%%%%%%%%%%%%%%%%%%%%%%%%%%%%%%%%%
\label{Subsection: Equivariant sheaves on top gpd}%%%%%%%%%%%%%%%%%%%%%%%%%%%%%%%%%%%%%%%%%%%%%%%%%%%%%%%%%%%%%%%%%%%%%%%%%%%%%%%%%%%%%%%%%%%%%%%%%%%%%%%%%%%%%%%%%%%%%%%%%%%%%%%%%%%%%%%%%%%%%%%%%%%%%%%%%%%%%%%%%%%%%%%%%%%%%
Recall (e.g.\ from \cite{elephant1}) that if \thry{H} is an arbitrary topological groupoid, which we also write as $H_1\rightrightarrows H_0$,
the topos of \emph{equivariant sheaves} on \thry{H}, which we write interchangeably as
\Sh{\thry{H}} or \Eqsheav{H_1}{H_0}, consists of the following
%\footnote{\cite{moerdijk:88} denotes the topos of
%equivariant sheaves on \thry{H} by $B(\thry{H})$. \cite{elephant1} uses the notation \funksjon{Cont}{\thry{H}}.}
\cite[B3.4.14(b)]{elephant1}, \cite{moerdijk:88}, \cite{moerdijk:90}.
An object of \Sh{\thry{H}} is a pair
\[\pair{a:A\rightarrow H_0, \alpha}\]
where $a$ is a local
homeo\-morphism (that is, an object of \Sh{H_0}) and $\alpha:H_1\times_{H_0} A \rightarrow A$ is a continuous
function from the pullback (in \alg{Sp}) of $a$ along the source map $s:H_1\rightarrow H_0$ to $A$ such that
\[ a(\alpha(f,x))=t(f)\]
and satisfying the axioms for an action:
\begin{enumerate}[(i)]
\item $\alpha(1_h,x)=x$ for $h\in H_0$.
\item $\alpha(g,\alpha(f,x))=\alpha(g\circ f,x)$.
\end{enumerate}
For illustration, it follows that for $f\in H_1$, $\alpha(f,-)$ is a
bijective function from the fiber over $s(f)$ to the fiber over $t(f)$.
An arrow $$h:\pair{a:A\rightarrow H_0, \alpha}\to \pair{b:B\rightarrow H_0, \beta}$$ is an arrow of
\Sh{H_0},
\[\bfig \Vtriangle<300,300>[A`B`H_0;h`a`b] \efig \]
which commutes with the actions:
\[ \bfig \square<500,400>[H_1\times_{H_0}A`A` H_1\times_{H_0}B`B; \alpha`1_{H_1}\times_{H_0}h`h`\beta] \efig \]
Working in equivariant sheaf toposes, we shall mostly be concerned with taking finite limits, unions of subobjects, and image factorizations, and we note that this is quite straightforward: the terminal object is the identity on $H_0$ with the only possible action, pullbacks are by pullbacks of spaces with the expected action, a subobject of \pair{a:A\rightarrow H_0, \alpha} is an open subset of $A$ which is closed under $\alpha$, and unions (and finite intersections) of  subobjects are by unions (finite intersections) of their underlying open sets. The image of a morphism is the image of its underlying function. Finally, we note that if \thry{H} is an \emph{open} topological groupoid, then the action of an equivariant sheaf is an open map (see e.g.\ \cite{moerdijk:88}). We write this out for reference.

%
%
%
%
%
%\begin{lemma}\label{Lemma: Open projection}
%For any object in \Eqsheav{G_{\theory}}{X_{\theory}},
%
%\[\left\langle\bfig\morphism<300,0>[R`X_{\theory};r]\efig,\rho\right\rangle \]
%
%the projection $\pi_2:G_{\theory}\times_{X_{\theory}} R\to R$ is open.
%\begin{proof}
%By Proposition \ref{Proposition: Multisort, groupoid is open}, since pullback preserves open maps of spaces.
%
%
%\end{proof}
%\end{lemma}
%
%
\begin{lemma}\label{lemma: Action is open}
For any object \pair{r:R\rightarrow H_0,\rho} in \Eqsheav{H_1}{H_0},
%
%\[\left\langle\bfig\morphism<0,-200>[R`X_{\theory};r]\efig,\rho\right\rangle\]
%
the action \[\rho:H_1\times_{H_0} R\to R\] is open.
% Consequently, the stabilization of any open subset of $R$ is again open.
%\begin{proof}
%Either by Lemma \ref{Lemma: Multisort, groupoid is open}, or by Lemma \ref{Lemma: Open projection} as follows:
%Let a basic open $V\times_{X_{\theory}} U\subseteq G_{\theory}\times_{X_{\theory}} R$ be given (so that
%$U\subseteq R$ and $V\subseteq G_{\theory}$ are open). Observe that, since the inverse map $i:G_{\theory}\to
%G_{\theory}$ is a homeomorphism, $i(V)$ is open, and
%
%\[
%\begin{array}{rl}
%\rho(V\times_{X_{\theory}}U) &= \cterm{y\in R}{\fins{\pair{f,x}\in V\times_{X_{\theory}}U}\rho(f,x)=y}\\
%                 &= \cterm{y\in R}{\fins{f^{-1}\in i(V)}s(f^{-1}) = r(y) \wedge \rho(f^{-1},y)\in U}\\
%                 &= \pi_2(\rho^{-1}(U)\cap(i(V)\times_{X_{\theory}} R))
%\end{array}\]
%
%is open by Lemma \ref{Lemma: Open projection}. Finally, for any open $U\subseteq R$, the stabilization of $U$ is
%$\rho(G_{\theory}\times_{X_{\theory}} U)$.
%
%\end{proof}
\end{lemma}
\subsection{The Representation Theorem}%%%%%%%%%%%%%%%%%%%%%%%%%%%%%%%%%%%%%%%%%%%%%%%%%%%%%%%%%%%%%%%%%%%%%%%%%%%%%%%%%%%%%%%%%%%%%%%%%%%%%%%%%%%%%%%%%%%%%%%%%%%%%%%%%%%%%%%%%%%%%%%%%%%%%%%%%%%%%%%%%%%%%%%%%%%%%%%%%%%%%%
\label{Subsection:The representation theorem}%%%%%%%%%%%%%%%%%%%%%%%%%%%%%%%%%%%%%%%%%%%%%%%%%%%%%%%%%%%%%%%%%%%%%%%%%%%%%%%%%%%%%%%%%%%%%%%%%%%%%%%%%%%%%%%%%%%%%%%%%%%%%%%%%%%%%%%%%%%%%%%%%%%%%%%%%%%%%%%%%%%%%%%%%%%%
We give a proof that the topos of equivariant sheaves on $\thry{G}_{\theory}$ is the classifying topos of \theory, with the functor $\cat{M}:\synt{C}{T}\to\Eqsheav{G_{\theory}}{X_{\theory}}$ which sends a formula to its corresponding definable set functor being the generic topos model of \theory. We devote some space, first, to briefly explain the setup and point out the formal similarities to the representation of a Boolean algebra in terms of its Stone space. Returning to definable set functors, notice that if we consider such a functor restricted to the set of models $X_{\theory}$
%Next, for $\syntob{\alg{x}}{\phi}\in\synt{C}{T}$, the definable set functor given by $\phi$ restricts to a functor
%
\begin{align*}
\csem{\alg{x}}{\phi}^{(-)} :\ & X_{\theory}\to \Sets\\
& \alg{M} \longmapsto \csem{\alg{x}}{\phi}^\alg{M}
\end{align*}
then, following the equivalence $\Sets^{X_{\theory}}\simeq\Sets/X_{\theory}$, we can write it as a set over
$X_{\theory}$, e.g.:
\[
\csem{\alg{x}}{\phi}_{X_{\theory}}:=\cterm{\pair{\alg{M},\alg{b}}}{ \alg{M}\in X_{\theory}, \alg{b}\in
\csem{\alg{x}}{\phi}^{\alg{M}}}\to^{\pi} X_{\theory}
\]
where $\pi$ projects out the model $\alg{M}$. Note the notation ``\sox{\alg{x}}{\phi}'' for the set on the left, we will use this notation in what follows. Now, the mapping $\syntob{\alg{x}}{\phi}\mapsto
(\pi:\sox{\alg{x}}{\phi}\rightarrow X_{\theory})$ gives us the object part of a functor,
\[
\cat{M}_d:\synt{C}{T}\to \Sets/X_{\theory}
\]
(which sends an arrow of \synt{C}{T} to the obvious function over $X_{\theory}$). It is easy to see, since the relevant structure is computed `pointwise' in $\Sets/X_{\theory}$, that $\cat{M}_d$ is a coherent functor. It is equally easy to see, since \theory\ is complete with respect to models in $X_{\theory}$, that $\cat{M}_d$ is conservative, i.e.\ faithful and isomorphism-reflecting. Bringing model-isomorphisms back into consideration, we can extend $\cat{M}_d$ to a functor
\[\cat{M}'_d:\synt{C}{T}\to \Sets_{G_{\theory}}(X_{\theory})\simeq \Sets^{\thry{G}_{\theory}}\]
into the topos of equivariant sets (defined as in Section \ref{Subsection: Equivariant sheaves on top gpd} for the discrete topologies) on $\thry{G}_{\theory}$, by sending an object $\syntob{\alg{x}}{\phi}\in \synt{C}{T}$ to the pair \pair{\sox{\alg{x}}{\phi}, \theta_{\syntob{\alg{x}}{\phi}}}, where $\theta_{\syntob{\alg{x}}{\phi}}$ is the expected `application' action defined by
\[\theta_{\syntob{\alg{x}}{\phi}}(\pair{\alg{M},\alg{a}},\alg{f}:\alg{M}\rightarrow \alg{N})=\pair{\alg{N},\alg{f}(\alg{a})}\]
(when clear from context we shall drop the subscript on $\theta$). Again, this defines a coherent and conservative functor.
%

%\begin{remark}\label{Remark: Enough models}
The functor $\cat{M}_d$ (and hence also the functor $\cat{M}'_d$) is, in fact, not only conservative, but also \emph{cover-reflecting} with respect to the canonical coverage on the topos $\Sets/X_{\theory}$ and the coherent coverage on \synt{C}{T}. That is to say, for a family of arrows into a common codomain in \synt{C}{T}, if the image of those arrows  in $\Sets/X_{\theory}$ are jointly epimorphic  then there exist a finite selection of arrows that are already jointly epimorphic in \synt{C}{T}. The reason is that we have made sure that $X_{\theory}$ contains `enough models' for \theory\, in the sense that for any geometric sequent over $\Sigma$, if $\alg{M}\vDash\sigma$ for all $\alg{M}\in X_{\theory}$, then $\sigma$ is provable from \theory. (If \theory\ is a first-order theory, this property corresponds to the property that every \theory-type is realizable in some model in $X_{\theory}$.)
%\end{remark}
%

The embedding of \synt{C}{T} into sets over its set of models (or into equivariant sets over its groupoid of models) could be called `Stone representation for coherent theories'. From the representation result of Butz and Moerdijk \cite{butz:98b}, it is clear that, similar to the case for Boolean algebras and Stone spaces, one can equip the groupoid of models with topological structure such that the image of the embedding can be characterized in the resulting topos of equivariant sheaves. We give a elementary and self-contained proof of this where, although we include the perhaps tedious details, the conceptual idea is quite straightforward (and with some analogy to the Boolean algebra case): We first factor the embedding through the topos of equivariant sheaves on $\thry{G}_{\theory}$ by equipping the definable set functors with a suitable topology, and show that the resulting embedding is full. Then, we show that the objects in the image of the embedding form a `basis' or \emph{generating set} for the topos, whence it is the classifying topos for \theory. It follows that \synt{C}{T} can be recovered from \Eqsheav{G_{\theory}}{X_{\theory}} up to pretopos completion as the compact decidable objects, where a compact object is an object with the property that any covering family of subobjects must contain a covering finite subfamily.
\begin{definition}\label{Definition: Logical topology on sheaves}
For an object $\syntob{\alg{x}}{\phi}$ of \synt{C}{T}, the \emph{logical topology} on the set
\[ \sox{\alg{x}}{\phi}=\cterm{\pair{\alg{M},\alg{a}}}{\alg{M}\in X_{\theory},
\alg{a}\in \csem{\alg{x}}{\phi}^{\alg{M}}} \] is the coarsest such that the map $\pi: \csem{\alg{x}}{\phi}_{X_{\theory}}\rightarrow X_{\theory}$ is continuous and such that for every list of elements $\alg{a}\in\mathds{S}$ of the same length as \alg{x}, the image of the map $\bopen{\syntob{\alg{x}}{\phi},\alg{a}}\rightarrow \sox{\alg{x}}{\phi}$ defined by $\alg{M}\mapsto \pair{\alg{M},\alg{a}}$ is open.
\end{definition}
\begin{lemma}\label{lemma: Logical topology on sheaves}
For an object $\syntob{\alg{x}}{\phi}$ of \synt{C}{T}, a basis for the \emph{logical topology} on the set $\sox{\alg{x}}{\phi}$ is given by sets of the form
\[
\bopen{\syntob{\alg{x},\alg{y}}{\psi}, \alg{b}}:= \cterm{\pair{\alg{M},\alg{a}}}{ \alg{a}\ast\alg{b}\in
\csem{\alg{x},\alg{y}}{\phi\wedge\psi}^{\alg{M}}}
\]
(where $\alg{b}$ is of the same length as $\alg{y}$)
\begin{proof}Note that the (open) image of the map $\bopen{\syntob{\alg{x}}{\phi},\alg{a}}\rightarrow \sox{\alg{x}}{\phi}$ defined by $\alg{M}\mapsto \pair{\alg{M},\alg{a}}$ can be written as \bopen{\syntob{\alg{x},\alg{y}}{ \alg{x}=\alg{y}}, \alg{a}}. In general,  \bopen{\syntob{\alg{x},\alg{y}}{\psi}, \alg{b}} is an open set: for if $\pair{\alg{M},\alg{a}}\in\bopen{\syntob{\alg{x},\alg{y}}{\psi}, \alg{b}}$, then $\pair{\alg{M},\alg{a}}\in \bopen{\syntob{\alg{x},\alg{y}}{ \alg{x}=\alg{y}}, \alg{a}}\cap \pi^{-1}(\bopen{\syntob{\alg{x},\alg{y}}{\psi}, \alg{a}\ast\alg{b}})\subseteq \bopen{\syntob{\alg{x},\alg{y}}{\psi}, \alg{b}}$. It is clear that such sets form a basis.
\end{proof}
\end{lemma}
We now have the following:
\begin{proposition}\label{Proposition: DC, Md factors as M}
The mapping $\syntob{\alg{x}}{\phi}\mapsto \pair{\sox{\alg{x}}{\phi}, \theta}$ defines the object part of a coherent functor
\[
\cat{M}:\synt{C}{T}\to\Eqsheav{G_{\theory}}{X_{\theory}}
\]
which is cover-reflecting with respect to the coherent coverage on \synt{C}{T} and the canonical coverage on \Eqsheav{G_{\theory}}{X_{\theory}} (in particular, \cat{M} is conservative).
\begin{proof}
It is clear from Definition \ref{Definition: Logical topology on sheaves} and Lemma \ref{lemma: Logical topology on sheaves} that the projection $\pi:\sox{\alg{x}}{\phi}\rightarrow X_{\theory}$ is a local homeomorphism.
Also, given an arrow
\[\syntob{\alg{x},\alg{y}}{\sigma}:\syntob{\alg{x}}{\phi}\to
\syntob{\alg{y}}{\psi}\]
in \synt{C}{T}, the function $f_{\sigma}=\cat{M}(\sigma): \sox{\alg{x}}{\phi}\rightarrow \sox{\alg{y}}{\psi}$ is
continuous. For given a basic open
$\bopen{\syntob{\alg{y},\alg{z}}{\xi}, \alg{c}} \subseteq \sox{\alg{y}}{\psi}$, then
\[
f_{\sigma}^{-1}\left(\bopen{\syntob{\alg{y},\alg{z}}{\xi}, \alg{c}} \right) =
\bopen{\syntob{\alg{x},\alg{z}}{\fins{\alg{y}} \sigma \wedge \xi}, \alg{c}}
\]
Next, the action $\theta_{\syntob{\alg{x}}{\phi}}$ is continuous: Let a basic open
\[U= \bopen{\syntob{\alg{x},\alg{y}}{\psi}, \alg{b}}\subseteq \sox{\alg{x}}{\phi}\]
be given, and suppose $\theta(\alg{f},\pair{\alg{M},\alg{a}})= \pair{\alg{N},\alg{f}(\alg{a})}\in U$ for
$\alg{M},\alg{N}\in X_{\theory}$ and $f:\alg{M}\rightarrow \alg{N}$ in $G_{\theory}$. Then we can specify an
open neighborhood around \pair{\alg{f},\pair{\alg{M},\alg{a}}} which $\theta$ maps into $U$ as:
\[
\pair{\alg{f},\pair{\alg{M},\alg{a}}}\in \left(\begin{array}{c}
-  \\
  \alg{y}:\alg{f}^{-1}(\alg{b}) \mapsto \alg{b}   \\
 -
\end{array}\right)
 \times_{X_{\theory}}
\bopen{\syntob{\alg{x},\alg{y}}{\psi}, \alg{f}^{-1}(\alg{b})}
\]
Finally, it is a straightforward computation (either directly in \Eqsheav{G_{\theory}}{X_{\theory}} or in $\Sets/X_{\theory}$ using that the forgetful functor is a conservative geometric functor) to show that \cat{M} is coherent and cover-reflecting.
\end{proof}
\end{proposition}
We refer to the objects (and morphisms) in the image of \cat{M} as the \emph{definable} objects.
Next, we show that \cat{M} is full on compact subobjects, and conclude that \cat{M} is full. For a subobject (represented by an inclusion) $\syntob{\alg{x}}{\xi}\embedd \syntob{\alg{x}}{\phi}$ in
\synt{C}{T}, the open subset $\sox{\alg{x}}{\xi}\subseteq \sox{\alg{x}}{\phi}$ is closed under the
action $\theta$ in the usual sense that $\theta(a)\in \sox{\alg{x}}{\xi}$ for any point $a\in \sox{\alg{x}}{\xi}$. For an object, \pair{A\rightarrow X_{\theory}, \alpha}, of \Eqsheav{G_{\theory}}{X_{\theory}},
we call a subset, $S\subseteq A$, that is closed under the action of $G_{\theory}$ \emph{stable}, so as to
reserve ``closed'' to mean topologically closed. We claim that the only stable opens of $\sox{\alg{x}}{\phi}$ come
from subobjects of $\syntob{\alg{x}}{\phi}$ as joins.  Specifically:
\begin{lemma}\label{Lemma: Stabilization of basic open U}
Let $\syntob{\alg{x}:\alg{A}}{\phi}$ in \synt{C}{T} and $U$ a  basic open subset of
$\sox{\alg{x}:\alg{A}}{\phi}$ of the form
\[
U=\bopen{\syntob{\alg{x}:\alg{A},\alg{y}:\alg{B}}{\psi}, \alg{b}}
\]
be given. Then the stabilization (closure) of $U$ under the action $\theta$ of $G_{\theory}$ on
$\bopen{\syntob{\alg{x}:\alg{A}}{\phi}}$ is a subset of the form $\sox{\alg{x}:\alg{A}}{\xi}\subseteq
\sox{\alg{x}:\alg{A}}{\phi}$.

\begin{proof} We can assume without loss that $U$ is in reduced form. Let $\varphi$ be the formula expressing the
conjunction of inequalities $y_i\neq y_j$ for all pairs of indices $i\neq j$ such that $B_i=B_j$ in $\alg{B}$.
We claim that the stabilization of $U$ is $\sox{\alg{x}:\alg{A}}{\xi}$ where $\xi$ is the formula
$\finst{\alg{y}}{\alg{B}}\phi \wedge \psi\wedge \varphi$. First, $\sox{\alg{x}:\alg{A}}{\xi}$ is a stable set
containing $U$. Next, suppose $\pair{\alg{M},\alg{a}}\in \sox{\alg{x}:\alg{A}}{\xi}$. Then there exists
$\alg{c}$ such that $\alg{a}\ast \alg{c}\in \csem{\alg{x}:\alg{A},\alg{y}:\alg{B}}{\phi \wedge \psi\wedge
\varphi}^{\alg{M}}$. Then $\alg{b}$ and $\alg{c}$ (with respect to $\alg{B}$) satisfy the conditions of Lemma
\ref{Lemma: Star of David}, so there exists a \theory-model $\alg{N}$ with isomorphism
$\alg{f}:\alg{M}\rightarrow \alg{N}$ such that $f_{\alg{B}}(\alg{c})=\alg{b}$. Then $\theta
(\alg{f},\pair{\alg{M},\alg{a}})\in U$, and hence \pair{\alg{M},\alg{a}} is in the stabilization of $U$.
\end{proof}
\end{lemma}
%
%
%\begin{definition}
%
%We call a subset of the form $\sox{\alg{x}}{\xi}\subseteq \sox{\alg{x}}{\phi}$, for a subobject
%
%\[\syntob{\alg{x}}{\xi}\embedd \syntob{\alg{x}}{\phi}\]
%
%in \synt{C}{T}, a \emph{definable} subset of $\sox{\alg{x}}{\phi}$.
%
%\end{definition}
%
\begin{corollary}\label{Corollary: Definables the only stable opens}
\cat{M} is full and full on compact subobjects.
\begin{proof}
By Lemma \ref{Lemma: Stabilization of basic open U} a subobject of a definable object is a join of definable subobjects, so a compact subobject is a finite join of definable subobjects, and therefore definable. Moreover, since \cat{M} is cover-reflecting, definable objects are compact. Thus the graph of an arrow between definable objects is a compact subobject of a definable object, and so definable, and since \cat{M} is coherent and conservative, it is the image of a graph in \synt{C}{T}.
\end{proof}
\end{corollary}
It remains to show that the definable objects form a generating set for \Eqsheav{G_{\theory}}{X_{\theory}}.
\begin{lemma}\label{Lemma: Basic opens of X generate representables}
Let $\bopen{\syntob{\alg{x}:\alg{A}}{\phi},\alg{a}}$ be a basic open of $X_{\theory}$ in reduced form. Then
there exists a sheaf $\cat{M}({\syntob{\alg{x}:\alg{A}}{\xi}})$ and a (continuous) section
\[s:\bopen{\syntob{\alg{x}:\alg{A}}{\phi},\alg{a}} \to \sox{\alg{x}:\alg{A}}{\xi}\]
such that $\sox{\alg{x}:\alg{A}}{\xi}$ is the stabilization of the open set
$s(\bopen{\syntob{\alg{x}:\alg{A}}{\phi}})\subseteq \sox{\alg{x}:\alg{A}}{\xi}$.
\begin{proof}
Let $\varphi$ be the formula expressing the inequalities $x_i\neq x_j$ for all pairs of indices $i\neq j$ such
that $A_i=A_j$ in $\alg{A}$. Let $\xi:=\phi \wedge \varphi$ and consider the function
$s:\bopen{\syntob{\alg{x}:\alg{A}}{\phi},\alg{a}} \to \sox{\alg{x}:\alg{A}}{\xi}$ defined by $\alg{M}\mapsto
\pair{\alg{M},\alg{a}}$. The image of $s$ is open, so $s$ is a (continuous) section.
And by the proof of Lemma \ref{Lemma: Stabilization of basic open U}, the stabilization of
$\bopen{\syntob{\alg{x}:\alg{A},\alg{y}:\alg{A}}{\alg{x}=\alg{y}}, \alg{a}}$ is exactly
$\sox{\alg{x}:\alg{A}}{\xi}$.
\end{proof}
\end{lemma}
\begin{lemma}\label{Lemma: definables are generators}
The definable objects generate the topos \Eqsheav{G_{\theory}}{X_{\theory}}.
\begin{proof}
Let a sheaf $\pair{\bfig\morphism<300,0>[R`X_{\theory};r]\efig,\rho}$ and an element $x\in R$ be given. We  show that there exists a morphism with definable domain with $x$ in its image. First, we show that  there exists a basic open $\bopen{\syntob{\alg{x}:\alg{A}}{\phi},\alg{a}}\subseteq
X_{\theory}$ and a section $v: \bopen{\syntob{\alg{x}:\alg{A}}{\phi},\alg{a}}\rightarrow R$ containing $x$ such
that for any $\alg{f}:\alg{M}\rightarrow \alg{N}$ in $G_{\theory}$ such that $\alg{M}\in
\bopen{\syntob{\alg{x}:\alg{A}}{\phi},\alg{a}}$ and $f_{\alg{A}}(\alg{a})=\alg{a}$ (so that $\alg{N}$ is also in
$\bopen{\syntob{\alg{x}:\alg{A}}{\phi},\alg{a}}$), we have $\rho(\alg{f},v(\alg{M}))= v(\alg{N})$:
Given $x\in R$,
choose a section $s:\bopen{\syntob{\alg{y}:\alg{B}}{\psi},\alg{b}}\rightarrow R$ such that $x\in
s(\bopen{\syntob{\alg{y}:\alg{B}}{\psi},\alg{b}})$. Pull the open set
$s(\bopen{\syntob{\alg{y}:\alg{B}}{\psi},\alg{b}})$ back along the continuous action $\rho$,
\[
\bfig \square|allb|/->` >->`
>->`->/<600,300>[V`s(\bopen{\syntob{\alg{y}:\alg{B}}{\psi},\alg{b}})`G_{\theory}\times_{X_{\theory}} R`R;`\subseteq`
\subseteq`\rho ] \place(75,250)[\spbangle]  \efig
\]
to obtain an open set $V$ containing \pair{1_{r(x)},x}. Since $V$ is open, we can find a box of basic opens
around \pair{1_{r(x)},x} contained in $V$:
\[
\pair{1_{r(x)},x}\in W:= \left(\begin{array}{c}
\syntob{\alg{z}:\alg{C}}{\xi},\alg{c}  \\
  \alg{K}:\alg{k} \mapsto \alg{k}   \\
 \syntob{\alg{z'}:\alg{C'}}{\eta},\alg{c'}
\end{array}\right)
\times_{X_{\theory}} v'(U_{\syntob{\alg{y'}:\alg{D}}{\theta},\alg{d}})\subseteq V
\]
where $v'$ is a section $v':\bopen{\syntob{\alg{y'}:\alg{D}}{\theta},\alg{d}}\rightarrow R$ with $x$ in its
image. Notice that the preservation condition of $W$  (i.e.\ $\alg{K}:\alg{k} \mapsto \alg{k}  $) must have the
same sets on both the source and the target side, since it is satisfied by $1_{r(x)}$. Now, restrict $v'$ to the
subset
\[
U:=\bopen{\syntob{\alg{z}:\alg{C}, \alg{z''}:\alg{K}, \alg{z'}:\alg{C'}, \alg{y'}:\alg{D}}{\xi \wedge \eta
\wedge \theta},\alg{c}\ast \alg{k}\ast \alg{c'}\ast \alg{d}}
\]
to obtain a section $v=v'\upharpoonright_{U}:U\rightarrow R$. Notice that $x\in v(U)$. Furthermore,
$v(U)\subseteq s(\bopen{\syntob{\alg{y}:\alg{B}}{\psi},\alg{b}})$, for if $v(\alg{M})\in v(U)$, then
$\pair{1_{\alg{M}},v(\alg{M})}\in W$, and so $\rho(\pair{1_{\alg{M}},v(\alg{M})})=v(\alg{M})\in
s(\bopen{\syntob{\alg{y}:\alg{B}}{\psi},\alg{b}})$. Finally, if $\alg{M}\in U$ and $\alg{f}:\alg{M}\rightarrow
\alg{N}$ is an isomorphism in $G_{\theory}$ such that
\[ f_{\alg{C}\ast \alg{K}\ast \alg{C'}\ast
\alg{D}}(\alg{c}\ast \alg{k}\ast \alg{c'}\ast \alg{d}) =\alg{c}\ast \alg{k}\ast \alg{c'}\ast \alg{d}\]
then $\pair{\alg{f},v(\alg{M})} \in W$, and so $\rho(\alg{f},v(\alg{M}))\in
s(\bopen{\syntob{\alg{y}:\alg{B}}{\psi},\alg{b}})$. But we also have $v(\alg{N})\in v(U) \subseteq
s(\bopen{\syntob{\alg{y}:\alg{B}}{\psi},\alg{b}})$, and $r(\rho(\alg{f},v(\alg{M})))=r(v(\alg{N})$,  so
$\rho(\alg{f},v(\alg{M}))=v(\alg{N})$. This concludes the first part. Write \bopen{\syntob{\alg{x}:\alg{A}}{\phi},\alg{a}} for $U$ to shorten the notation.
Next, we lift the section $v:\bopen{\syntob{\alg{x}:\alg{A}}{\phi},\alg{a}}\rightarrow R$ to a morphism with definable domain.
We can
assume that $\bopen{\syntob{\alg{x}:\alg{A}}{\phi},\alg{a}}$ is on reduced form. Then, by Lemma \ref{Lemma:
Basic opens of X generate representables} there exists an object $\syntob{\alg{x}:\alg{A}}{\xi}$ in \synt{C}{T}
and a section $s:\bopen{\syntob{\alg{x}:\alg{A}}{\phi},\alg{a}}\rightarrow \sox{\alg{x}:\alg{A}}{\xi}$ such that
$\sox{\alg{x}:\alg{A}}{\xi}$ is the stabilization of $s(\bopen{\syntob{\alg{x}:\alg{A}}{\phi},\alg{a}})$. Define
a mapping $\hat{v}:\sox{\alg{x}:\alg{A}}{\xi}\rightarrow R$ as follows: for an element
$\pair{{\alg{N}},\alg{c}}\in \sox{\alg{x}:\alg{A}}{\xi}$, there exists $\pair{{\alg{M}}, \alg{a}}\in
s(\bopen{\syntob{\alg{x}:\alg{A}}{\phi},\alg{a}})\subseteq \sox{\alg{x}:\alg{A}}{\xi}$ and
$\alg{f}:{\alg{M}}\rightarrow {\alg{N}}$ in $G_{\theory}$ such that $f_{\alg{A}}(\alg{a})=\alg{c}$. Set
$\hat{v}(\pair{{\alg{N}},\alg{c}})=\rho(\alg{f}, v({\alg{M}}))$. We verify that $\hat{v}$ is well defined:
suppose $\pair{{\alg{M}}',\alg{a}}\in s(\bopen{\syntob{\alg{x}:\alg{A}}{\phi},\alg{a}}) \subseteq
\sox{\alg{x}:\alg{A}}{\xi}$ and $\alg{g}:{\alg{M}}'\rightarrow {\alg{N}}$ in $G_{\theory}$ is such that
$g_{\alg{A}}(\alg{a})=\alg{c}$. Then $\alg{g}^{-1}\circ \alg{f}:{\alg{M}}\rightarrow {\alg{M}}'$ sends
$\alg{a}\in \csem{\alg{x}:\alg{A}}{\phi}^{\alg{M}}$ to $\alg{a}\in \csem{\alg{x}:\alg{A}}{\phi}^{{\alg{M}}'}$,
and so by the choice of section $v:\bopen{\syntob{\alg{x}:\alg{A}}{\phi},\alg{a}}\rightarrow R$, we have that
$\rho(\alg{g}^{-1}\circ \alg{f},v({\alg{M}}))=v({\alg{M}}')$. But then
\[
\rho(\alg{g},v({\alg{M}}'))= \rho(\alg{g},\rho(\alg{g}^{-1}\circ \alg{f},v({\alg{M}})))
=\rho(\alg{f},v({\alg{M}}))\] so the value of $\hat{v}$ at \pair{{\alg{N}},\alg{c}} is indeed independent of the
choice of \pair{{\alg{M}},\alg{a}} and $\alg{f}$. Moreover, the following triangle commutes,
\begin{equation}\label{eq: commuting triangle in generators proof}
\bfig \Vtriangle/>`<-< `<-<
/<300,300>[\sox{\alg{x}:\alg{A}}{\xi}`R`\bopen{\syntob{\alg{x}:\alg{A}}{\phi},\alg{a}};\hat{v}`s`v] \efig
\end{equation}
and so $x$ is in the image of $\hat{v}$. The definition of $\hat{v}$ makes it straightforward to see that $\hat{v}$ commutes with the actions $\theta$
and $\rho$ of $\sox{\alg{x}:\alg{A}}{\xi}$ and $R$, respectively. Remains to show that $\hat{v}$ is continuous.
%Recall the commuting triangle from the proof of Lemma \ref{Lemma: Existence of function}:
%
%\[
%\bfig \Vtriangle/>`<-< `<-<
%/<300,300>[E_{\syntob{\alg{x}:\alg{A}}{\xi}}`R`U_{\syntob{\alg{x}:\alg{A}}{\phi},\alg{a}}; \hat{v}`s`v] \efig
%\]
%
Consider the triangle (\ref{eq: commuting triangle in generators proof}). Let $y\in
\hat{v}(\sox{\alg{x}:\alg{A}}{\xi})$ be given, and suppose $U$ is a open neighborhood of $y$. By Lemma
\ref{lemma: Action is open}, we can assume that $U\subseteq \hat{v}(\sox{\alg{x}:\alg{A}}{\xi})$. Suppose
$y=\hat{v}(\pair{{\alg{N}},\alg{c}})=\rho(\alg{f},v({\alg{M}}))$ for a $\alg{f}:{\alg{M}}\rightarrow {\alg{N}}$
such that $\theta(\alg{f},s({\alg{M}}))=\pair{{\alg{N}},\alg{c}}$. We must find an open neighborhood $W$ around
$\pair{{\alg{N}},\alg{c}}$ such that $\hat{v}(W)\subseteq U$. First, define the open neighborhood $T\subseteq
G_{\theory}\times_{X_{\theory}}R$ around \pair{\alg{f},v({\alg{M}})} by
\[
%\pair{f,v({\alg{M}})}\in
T:=\rho^{-1}(U)\cap \left( G_{\theory}\times_{X_{\theory}}v(\bopen{\syntob{\alg{x}:\alg{A}}{\phi},\alg{a}})
\right)
\]
From the homeomorphism $v(\bopen{\syntob{\alg{x}:\alg{A}}{\phi},\alg{a}}) \cong
s(\bopen{\syntob{\alg{x}:\alg{A}}{\phi},\alg{a}})$ we obtain a homeomorphism
$G_{\theory}\times_{X_{\theory}}v(\bopen{\syntob{\alg{x}:\alg{A}}{\phi},\alg{a}})\cong
G_{\theory}\times_{X_{\theory}} s(\bopen{\syntob{\alg{x}:\alg{A}}{\phi},\alg{a}})$. Set $T'\subseteq
G_{\theory}\times_{X_{\theory}}s(\bopen{\syntob{\alg{x}:\alg{A}}{\phi},\alg{a}})$ to be the open subset
corresponding to $T$ under this homeomorphism,
\begin{align*}
\pair{\alg{f}, v({\alg{M}})} &\in  T \subseteq G_{\theory}\times_{X_{\theory}}v(\bopen{\syntob{\alg{x}:\alg{A}}{\phi},\alg{a}})\\
                  &  \ \ \ \ \ \ \cong\\
\pair{\alg{f},s({\alg{M}})} &\in T'   \subseteq
G_{\theory}\times_{X_{\theory}}s(\bopen{\syntob{\alg{x}:\alg{A}}{\phi},\alg{a}})
\end{align*}
Then $\pair{{\alg{N}},\alg{c}}=\theta(\alg{f}, s({\alg{M}}))\in \theta(T')$, and by Corollary \ref{lemma:
Action is open}, $\theta(T')$ is open. We claim that $\hat{v}(\theta(T'))\subseteq U$: for suppose
$\pair{\alg{g},s(\alg{P})}\in T'$. Then $\pair{\alg{g},v(\alg{P})}\in T\subseteq \rho^{-1}(U)$, and so
$\hat{v}(\theta(\alg{g},s(\alg{P})))=\rho(\pair{\alg{g},v(\alg{P})})\in U$. Thus $\theta(T')$ is the required
$W$.
\end{proof}
\end{lemma}
%
%
%Through Lemmas \ref{Lemma: Existence of function}--\ref{Lemma: Function is morphism of eq.sheaves} we have
%established the following:
%
%
We conclude:
\begin{theorem}\label{Theorem: Representation theorem}
For a decidable coherent theory \theory\ we have an equivalence of toposes,
\[
\Eqsheav{G_{\theory}}{X_{\theory}}\simeq \classtop .
\]
where $\thry{G}_{\theory}$ is the topological groupoid of \theory-models constructed over the set $\mathds{S}$.
\begin{proof}Since, by Lemma \ref{Lemma: definables are generators}
the definable objects form a generating set, the full subcategory of definable objects is a site for
\Eqsheav{G_{\theory}}{X_{\theory}}\ when equipped with the canonical coverage inherited from
\Eqsheav{G_{\theory}}{X_{\theory}} (see e.g.\ \cite[C2.2.16]{elephant1}). Since, by Proposition \ref{Proposition: DC, Md factors as M} and Corollary \ref{Corollary: Definables the only stable opens} the functor
$\cat{M}:\synt{C}{T}\to\Eqsheav{G_{\theory}}{X_{\theory}}$ is full and faithful and cover-reflecting with respect to the coherent coverage, this means that  \Eqsheav{G_{\theory}}{X_{\theory}} is equivalent to the topos \Sh{\synt{C}{T},J} of sheaves for the coherent coverage on \synt{C}{T}, and the latter is the classifying topos of \theory\ by \cite[D3.1.9]{elephant1}.
\end{proof}
\end{theorem}
A consequence of Theorem \ref{Theorem: Representation theorem} is that a theory can be recovered from its topological groupoid of models. We explore this further, and with further explanation, in Section \ref{Section: Duality}, but state it here. We say that two theories are the same up to pretopos completion if their syntactic categories have equivalent pretopos completions.
\begin{corollary}\label{Corollary: Recover theory}
A coherent decidable theory \theory\ can be recovered, up to pretopos completion,  from its topological groupoid $\thry{G}_{\theory}$ of models as the full subcategory of compact decidable objects in the topos \Eqsheav{G_{\theory}}{X_{\theory}}.
\begin{proof}By \cite[D3.3]{elephant1}
\end{proof}
\end{corollary}
%
%
%\begin{remark}
%An alternate proof of Theorem \ref{Proposition: Multisort classtop eqsh equivalence}, following the lines of \cite{butz:98b}, is given in \cite[Chapter 3]{phd}.  It proceeds by showing that the spatial covering $$m : %\Sh{X_\theory} \to \classtop$$ of Section \ref{Subsection: Stone representation} is an open surjection and thus, by results of \cite{joyal:84}, an effective descent morphism.  The groupoid representation %$\Eqsheav{G_{\theory}}{X_{\theory}}\simeq \classtop$ then follows from descent theory.
%\end{remark}
%
%
%
%
%
\section{Duality}%%%%%%%%%%%%%%%%%%%%%%%%%%%%%%%%%%%%%%%%%%%%%%%%%%%%%%%%%%%%%%%%%%%%%%%%%%%%%%%%%%%%%%%%%%%%%%%%%%%%%%%%%%%%%%%%%%%%%%%%%%%%%%%%%%%%%%%%%%%%%%%%%%%%%%%%%%%%%%%%%%%%%%%%%%%%%%%%%%%%%%%%%%%%%%%%%%%%%%%%%%%%%%
\label{Section: Duality}%%%%%%%%%%%%%%%%%%%%%%%%%%%%%%%%%%%%%%%%%%%%%%%%%%%%%%%%%%%%%%%%%%%%%%%%%%%%%%%%%%%%%%%%%%%%%%%%%%%%%%%%%%%%%%%%%%%%%%%%%%%%%%%%%%%%%%%%%%%%%%%%%%%%%%%%%%%%%%%%%%%%%%%%%%%%%%%%%%%%%%%%%%%%%%%%%%%%%%%
%%%%%%%%%%%%%%%%%%%%%%%%%%%%%%%%%%%%%%%%%%%%%%%%%%%%%%%%%%%%%%%%%%%%%%%%%%%%%%%%%%%%%%%%%%%%%%%%%%%%%%%%%%%%%%%%%%%%%%%%%%%%%%%%%%%%%%%%%%%%%%%%%%%%%%%%%%%%%%%%%%%%%%%%%%%%%%%%%%%%%%%%%%%%%%%%%%%%%%%%%%%%%%%%%%%%%%%%%%%%%%%
%%%%%%%%%%%%%%%%%%%%%%%%%%%%%%%%%%%%%%%%%%%%%%%%%%%%%%%%%%%%%%%%%%%%%%%%%%%%%%%%%%%%%%%%%%%%%%%%%%%%%%%%%%%%%%%%%%%%%%%%%%%%%%%%%%%%%%%%%%%%%%%%%%%%%%%%%%%%%%%%%%%%%%%%%%%%%%%%%%%%%%%%%%%%%%%%%%%%%%%%%%%%%%%%%%%%%%%%%%%%%%%
%
%It is clear from the representation theorem of Butz and Moerdijk \cite{butz:98b} that one can pass back and forth between theories and their corresponding groupoids of models and isomorphism
%
Based on our version of Butz and Moerdijk's groupoid representation of toposes with enough points in the form of Theorem \ref{Theorem: Representation theorem}, we construct a `duality' between decidable coherent categories, representing decidable coherent theories, and a category of topological groupoids. This  takes the form of a contravariant adjunction between decidable coherent categories and a category of groupoids which is a `duality' in the sense that the counit components at pretoposes are equivalences. There are several possibilities for choosing suitable categories of groupoids with respect to which this adjunction can be constructed. We chose one which seems natural for our purposes in that it is easy to specify, makes it straightforward to extract decidable coherent categories from the groupoids in it, and is quite inclusive so as to leave scope for further restrictions. Specifically, we take the category of groupoids to consist of those topological groupoids which are `weakly coherent' in the sense that the induced equivariant sheaf toposes have a generating set of compact objects and the property that compact objects are closed under finite products, with morphisms between such groupoids being those with induced inverse image functors that preserve compact objects. We give an intrinsic characterization of such `weakly coherent' groupoids as well as of a more restricted class of `decidable coherent' groupoids to which the adjunction can be restricted. We leave the intrinsic characterization of the morphisms, or of further restrictions with respect to morphisms, to future work. Section  \ref{Subsection: Weakly coherent groupoids} introduces weakly coherent groupoids, Section \ref{Subsection: Representation theorem for decidable coherent categories} translates Theorem \ref{Theorem: Representation theorem} to decidable coherent categories, and Sections \ref{Subsection: DC, The Semantical Functor}--\ref{Subsection: DC, The Syntax-Semantics Adjunction} construct the adjunction.
\subsection{Groupoids and sheaves}%%%%%%%%%%%%%%%%%%%%%%%%%%%%%%%%%%%%%%%%%%%%%%%%%%%%%%%%%%%%%%%%%%%%%%%%%%%%%%%%%%%%%%%%%%%%%%%%%%%%%%%%%%%%%%%%%%%%%%%%%%%%%%%%%%%%%%%%%%%%%%%%%%%%%%%%%%%%%%%%%%%%%%%%%%%%%%%%%%%%%%%%%%
\label{Subsection: Weakly coherent groupoids}%%%%%%%%%%%%%%%%%%%%%%%%%%%%%%%%%%%%%%%%%%%%%%%%%%%%%%%%%%%%%%%%%%%%%%%%%%%%%%%%%%%%%%%%%%%%%%%%%%%%%%%%%%%%%%%%%%%%%%%%%%%%%%%%%%%%%%%%%%%%%%%%%%%%%%%%%%%%%%%%%%%%%%%%%%%%%%%%%%%
%
%
%We intrinsically characterize, using Moerdijk's site description for open topological groupoids, a category of `weakly coherent groupoids' from which we construct an adjoint to \Mod.
%
%
\subsubsection{Moerdijk's Site Construction for Topological Groupoids}%%%%%%%%%%%%%%%%%%%%%%%%%%%%%%%%%%%%%%%%%%%%%%%%%%%%%%%%%%%%%%%%%%%%%%%%%%%%%%%%%%%%%%%%%%%%%%%%%%%%%%%%%%%%%%%%%%%%%%%%%%%%%%%%%%%%%%%%%%%%%%%%%%%%%%%%%%%%
\label{Subsubsection: Moerdijks site construction}

We recall the essentials of the site description for toposes of equivariant sheaves on open topological groupoids given in \cite{moerdijk:88}. Let $\thry{G}=(G_1\rightrightarrows G_0)$ be an  open topological groupoid; let $N\subseteq G_1$ be an open subset of arrows that is closed under inverses and compositions; and let $U=s(N)=t(N)\subseteq G_0$. We refer to the pair $(U,N)$ as an \emph{open subgroupoid} of \thry{G}. Then
\[s^{-1}(U)/_{\sim_N}\to^t G_0\]
is an equivariant sheaf over $G_0$, denoted \pair{\thry{G},U,N}, where $f\sim_N g$ iff $t(f)=t(g)$ and $g^{-1}\circ f \in N$. The action is defined by composition,
\[\pair{g:y\rightarrow z,[f:x\rightarrow y]}\mapsto [g\circ f].\]
The set of objects of this form is a generating set for \Eqsheav{G_1}{G_0}, that is, a set of objects such that for all equivariant sheaves there exists a covering (epimorphic) family of arrows with domains in the set. Briefly, this is because if \pair{\rho, r:R\rightarrow G_0} is an equivariant sheaf and $u:U\rightarrow R$ is a continuous section,
then $N=\cterm{f\in s^{-1}(U)\cap t^{-1}(U)}{\rho(f,u(s(f))=u(t(f))}$ is an open set of arrows closed under inverses and compositions. The map $e:U\rightarrow s^{-1}(U)/_{\sim_N}$ defined by $x\mapsto [1_x]$ is a continuous section, which we will refer to as the \emph{canonical section},  and $u$ lifts to a morphism $\hat{u}:\pair{\thry{G},U,N}\rightarrow \pair{\rho, r:R\rightarrow G_0}$ such that $u=\hat{u}\circ e$.
\begin{equation}\label{Eq: GUN generating}\bfig
\ptriangle/>`<-`<-/<500,350>[s^{-1}(U)/_{\sim_N}`R`U;\hat{u}`e`]
\dtriangle/<-`>`/<500,350>[R`U`G_0;u`r`]
\place(250,0)[\subseteq]
\efig\end{equation}
Specifically, $\hat{u}([f:v\rightarrow x])=\rho(f, u(v))$. Since
\begin{align*}
&\hat{u}([f:v\rightarrow x])=\hat{u}([g:v'\rightarrow x])\\
&\Rightarrow \rho(f, u(v))=\rho(g, u(v'))\\
&\Rightarrow \rho(g^{-1}\circ f, u(v))= u(v')\\
&\Rightarrow f\sim_N g
\end{align*}
$\hat{u}$ is 1-1, so that, in fact, every equivariant sheaf is covered by its subobjects of the form \pair{\thry{G},U,N}.
Refer to the full subcategory of  objects of the form \pair{\thry{G},U,N}  as the \emph{Moerdijk site} for \Eqsheav{G_1}{G_0} (the implicit coverage is the canonical one inherited from \Eqsheav{G_1}{G_0}), and denote it by $\cat{S}_{\thry{G}}$. The following properties of Moerdijk sites will be of use and we state them in a single lemma here for reference (cf.\ \cite{moerdijk:88}, in particular Lemma 6.2).
\begin{Lemma}\label{Lemma: GUN subobjects}Let \thry{G} be an  open topological groupoid.

(1) The Moerdijk site of \Eqsheav{G_1}{G_0} is closed under subobjects. In particular, let \pair{\thry{G},U,N} be an object of \Eqsheav{G_1}{G_0} Then
\[V \mapsto s^{-1}(V)/_{\sim_{N\upharpoonright_{V}}}\]
defines an isomorphism between the frame of open subsets of $U$ that are closed under $N$ and the frame of subobjects of  \pair{\thry{G},U,N}.

(2) If $f:\thry{H}\rightarrow \thry{G}$ is a morphism of open topological groupoids and $(U,N)$ is an open subgroupoid of \thry{G}, then $(f_0^{-1}(U), f_1^{-1}(N))$ is an open subgroupoid of \thry{H}. The canonical `comparison' morphism
 $\hat{k}: \pair{\thry{H},f_0^{-1}(U), f_1^{-1}(N)}\rightarrow f^*(\thry{G},U,N)$ defined by $[g]\mapsto [f_1(g)]$ is an isomorphism if and only if for each $h:u\rightarrow f_0(x)$ in $s^{-1}(U)$ there exists $g:v\rightarrow x$ in $s^{-1}(f_0^{-1}(U))$ such that $h^{-1}\circ f_1(g)\in N$.

(3) In particular, if $f:\thry{H}\rightarrow \thry{G}$ is a morphism of open topological groupoids such that for all $(h:x\rightarrow f_0(y))\in G_1$ there exists $g\in H_1$ such that $t(g)=y$ and $f_1(g)=h$, then $f^*:\Eqsheav{G_1}{G_0}\rightarrow \Eqsheav{H_1}{H_0}$ restricts to a morphism of Moerdijk-sites such that
\[f^*(\pair{\thry{G},U,N})\cong\pair{\thry{H},f_0^{-1}(U),f_1^{-1}(N)}\]
for all open subgroupoids $(U,N)$ of \thry{G}.
\begin{proof}

(1) The inverse is given by
\[\bfig
\square/ >->`->`->` >->/<700,350>[V`S`U`s^{-1}(U)/_{\sim_{N}};`\subseteq`\subseteq`e]
\place(100,250)[\pbangle]
\efig\]

(2) Consider the diagram
\[\bfig
\square|alrm|/>`<-`>`>/<1200,500>[s^{-1}(f_0^{-1}(U))/_{\sim_{N_k}}`H_0\times_{G_0}s^{-1}(U)/_{\sim_N}`V=f_0^{-1}(U)`H_0;\hat{k}`e``\subseteq]
\place(1300,400,)[\pbangle]
\morphism<1200,500>[V=f_0^{-1}(U)`H_0\times_{G_0}s^{-1}(U)/_{\sim_N};k]
\square(1200,0)<1200,500>[H_0\times_{G_0}s^{-1}(U)/_{\sim_N}`s^{-1}(U)/_{\sim_N}`H_0`G_0;``t`f_0]
\efig\]
where $k$ is the section obtained by pulling back the canonical section $e:U\rightarrow s^{-1}(U)/_{\sim_N}$---so that $k(v)=\pair{v,[1_{f_0(v)}]_{\sim_N}}$---and $N_t\subseteq H_1$ and $\hat{k}$ are the induced open subgroupoid and morphism. Now, we have
\begin{align*}
N_k&=\cterm{g\in s^{-1}(V)\cap t^{-1}(V)}{f_1(g)\circ [1_{f_0(s(g))}]_{\sim_N}=[1_{f_0(t(g))}]_{\sim_N}}\\
   &=\cterm{g\in s^{-1}(V)\cap t^{-1}(V)}{f_1(g)\in N}\\
   &=f_1^{-1}(N)
\end{align*}
and so $\pair{\thry{H},f_0^{-1}(U),f_1^{-1}(N)}=\pair{\thry{H},f_0^{-1}(U),N_k}$, and as noted above $\hat{k}$ is injective. Since
\[\hat{k}([g:v\rightarrow x]_{\sim_{N_k}})=\pair{x,f_1(g)\circ[1_{f_0(v)}]_{\sim_N}}=\pair{x,[f_1(g)]_{\sim_N}}\]
it is clear that $\hat{k}$ is surjective if and only if for each $h:u\rightarrow f_0(x)$ in $s^{-1}(U)$ there exists $g:v\rightarrow x$ in $s^{-1}(f_0^{-1}(U))$ such that $h^{-1}\circ f_1(g)\in N$. (3) is clearly implied.
%
%(3) The condition entails that the condition of (2) is satisfied for all open subgroupoids $(U,N)$ of \thry{G}.
%Remains to show that it is also surjective. Let \pair{x,[g:u\rightarrow f_0(x)]_{\sim_N}} be given. Since $f:\thry{H}\rightarrow \thry{G}$ is strictly full, there exist $(h:y\rightarrow x)\in H_1$ such that $f_1(h)=g$, and since, %accordingly, $f_0(y)=u$ we have $h\in s^{-1}(f_0^{-1}(U))$. But then
%\[\hat{t}([h]_{\sim_{N_t}})=\pair{x,f_1(h)\circ[1_{f_0(y)}]_{\sim_N}}=\pair{x,[f_1(h)]_{\sim_N}}=\pair{x,[g]_{\sim_N}}.\]
%
\end{proof}
\end{Lemma}
For an open subgroupoid $(U,N)$ of \thry{G} say that a morphism $f:\thry{H}\to\thry{G}$ of open topological groupoids is $N$-\emph{fibration} if it satisfies the condition of Lemma \ref{Lemma: GUN subobjects} (2). Say that $f$ is a \emph{fibration} if it satisfies he condition of  Lemma \ref{Lemma: GUN subobjects} (3).

Morphisms in the Moerdijk site can also be described in terms of open sets (cf.\ \cite[6.3]{moerdijk:88}):
\begin{lemma}\label{Lemma: Characterizing morphisms}
Given two objects, $\pair{\thry{G},U,N}$ and $\pair{\thry{G},V,M}$, in $\Eqsheav{G_1}{G_0}$, morphisms between them,
\[\bfig
\Vtriangle<400,350>[s^{-1}(U)/_{\sim_N}`s^{-1}(V)/_{\sim_M}`G_0;\hat{k}`t`t]
\efig\]
in  are in one-to-one correspondence with open subsets
\[K\subseteq s^{-1}(V)\]
satisfying the following properties:
\begin{enumerate}[i)]
\item $c(K\times_{G_0}M)\subseteq K$, i.e., $K$ is closed under $\sim_{M}$;
\item $t(K)=U$;
\item $c(K^{-1}\times_{G_0}K)\subseteq M$, i.e., if two arrows in $K$ share a codomain then they are $\sim_{M}$-equivalent;
\item $c(N\times_{G_0}K)\subseteq K$, i.e., if $f:x\rightarrow y$ is in $K$ and $g:y\rightarrow z$ is in $N$ then $g\circ f\in K$.
\end{enumerate}
Moreover, $\hat{k}$ can be thought of as `precomposing with $K$', in the sense that $\hat{k}([f]_{\sim_N})=[f\circ g]_{\sim_M}$ for some (any) $g\in K$ such that $t(g)=s(f)$.
\begin{proof}
Let a morphism $\hat{k}:s^{-1}(U)/_{\sim_N}\to s^{-1}(V)/_{\sim_M}$ be given, and let $k:U\rightarrow s^{-1}(V)/_{\sim_M}$ be the composition of $\hat{k}$ with the canonical section $e:U\rightarrow s^{-1}(U)/_{\sim_N}$, so that $\hat{k}([g: u\rightarrow x])=g\circ t(u)$. Pull $k$ back along the quotient map to obtain an open set, $K\subseteq s^{-1}(V)$,
\[\bfig
\square/>` >->` >->`->>/<600,300>[K`U`s^{-1}(V)`s^{-1}(V)/_{\sim_M};`\subseteq`k`q]
\place(100,200)[\pbangle]
\efig\]
Properties (i)--(iii) then easily follow. Property (iv) follows since  $[g\circ f]=g\circ[f]=g\circ k(t(f))= g\circ \hat{k}([1_{t(f)}])=\hat{k}([g])=\hat{k}([1_{t(g)}])=k(t(g))$.

Conversely, let $K\subseteq s^{-1}(V)$ be given and assume $K$ satisfies properties (i)--(iv). Map an object $x\in U$ to the set $t(x):=\cterm{f\in K}{t(f)=x}$. This yields a well-defined function $t:U\to s^{-1}(V)/_{\sim_M}$ by properties (i), (ii) and (iii). And since $t(U)=q(K)$ and $q:s^{-1}(V)\to s^{-1}(V)/_{\sim_M}$ is open, $t(U)$ is open, so $t$ is a continuous section. By property (iv), $t:U\to s^{-1}(V)/_{\sim_M}$ can easily be seen to determine a morphism $\hat{t}:s^{-1}(U)/_{\sim_N}\to s^{-1}(V)/_{\sim_M}$. It is clear that these constructions are inverse to each other. The final statement of the lemma is then clear from the fact that $\hat{t}$ commutes with the actions.
\end{proof}
\end{lemma}
This allows us to translate the notion of a generating set to subgroupoids:
\begin{lemma}\label{Lemma: Generating set of GUNS}
Let $S=\cterm{(U_i,N_i)}{i\in I}$ be a set of open subgroupoids of \thry{G}. The induced objects \pair{\thry{G},U_i,N_i} form a generating set for \Eqsheav{G_1}{G_0} if and only if for all open subgroupoids $(V,M)$ and all $v\in V$ there exists $N_i\in S$ and open subset $K\subseteq s^{-1}(V)\cap t^{-1}(U_i)$  satisfying the conditions of Lemma \ref{Lemma: Characterizing morphisms} such that $v\in s(K)$.
\begin{proof}Straightforward by Lemma \ref{Lemma: Characterizing morphisms}.
\end{proof}
\end{lemma}
Say, accordingly, that a set of open subgroupoids satisfying the conditions of Lemma \ref{Lemma: Generating set of GUNS} is \emph{generating}. Finally, we note that groupoids of $\kappa$-small models  $\thry{G}_{\theory}$ (for theories with enough such models) have generating sets of subgroupoids of the following form.
\begin{lemma}\label{Lemma: Definables are GUNS}
Let \theory\ be a (coherent decidable) theory with enough $\kappa$-small models, and $\thry{G}_{\theory}$ its groupoid of $\kappa$-small models, as in Section \ref{Section: Reprentation Theorem}. Let \syntob{\alg{x}}{\phi} be a formula which implies that the variables in \alg{x} are sortwise distinct (for instance a conjunction of an arbitrary formula with a formula stating that they are sortwise distinct). Let \alg{a} be a list of distinct elements from $\mathds{S}$ of the same length as \alg{x}.  Then the sheaf $\pair{\thry{G}_{\theory},U,N}$ with $U=\bopen{\syntob{\alg{x}}{\phi},\alg{a}}$ and
\[N=\left(\begin{array}{c}
\syntob{\alg{x}}{\phi},\alg{a}  \\
\alg{x}:\alg{a} \mapsto \alg{a}   \\
 \syntob{\alg{x}}{\phi},\alg{a}
\end{array}\right)\]
is isomorphic to the definable sheaf \pair{\sox{\alg{x}}{\phi}\rightarrow X_{\theory},\theta}.
\begin{proof}Consider the continuous section $k:U\rightarrow \sox{\alg{x}}{\phi}$ defined by $\alg{M}\mapsto \pair{\alg{M},\alg{a}}$. Then the open subgroupoid induced by $k$ is precisely $(U,N)$, and $k$ lifts to a 1-1 morphism of equivariant sheaves $\hat{k}:s^{-1}(U)/_{\sim_N} \to \sox{\alg{x}}{\phi}$. This morphism is a surjection by Lemma \ref{Lemma: Stabilization of basic open U}.
\end{proof}
\end{lemma}
\begin{corollary}\label{Corollary: Definable GUNS gnerate}
Open subgroupoids of the form described in Lemma \ref{Lemma: Definables are GUNS} form a generating set of subgroupoids for $\thry{G}_{\theory}$.
\begin{proof}By (the proof of) Lemma \ref{Lemma: definables are generators}.
\end{proof}
\end{corollary}
\subsubsection{Weakly Coherent Groupoids}%%%%%%%%%%%%%%%%%%%%%%%%%%%%%%%%%%%%%%%%%%%%%%%%%%%%%%%%%%%%%%%%%%%%%%%%%%%%%%%%%%%%%%%%%%%%%%%%%%%%%%%%%%%%%%%%%%%%%%%%%%%%%%%%%%%%%%%%%%%%%%%%%%%%%%%%%%%%%%%%%%%%%%%%%%%%%%%%%%%%%%%%%%
\label{Subsubsection: Weakly coherent groupoids}%%%%%%%%%%%%%%%%%%%%%%%%%%%%%%%%%%%%%%%%%%%%%%%%%%%%%%%%%%%%%%%%%%%%%%%%%%%%%%%%%%%%%%%%%%%%%%%%%%%%%%%%%%%%%%%%%%%%%%%%%%%%%%%%%%%%%%%%%%%%%%%%%%%%%%%%%%%%%%%%%%%%%%%%%%%%%%%%%%%
Recall (e.g.\ from \cite{elephant1}) the following: (1) An object $A$ in a topos is \emph{compact} if every covering of it (in terms of morphisms or subobjects) has a finite subcovering. (2) An object $C$ in a topos \cat{E} is \emph{coherent} if (a) it is compact; and (b) it is \emph{stable}, in the sense that for any morphism $f:B\to A$ with $B$ compact, the domain $K$ of the kernel pair of $f$,
\[ K\two^{k_1}_{k_2} B\to^f A\]
is again compact. (3) A topos is \emph{compact} if the terminal object is compact. (4) A topos is \emph{coherent} if it has a generating set of compact objects the full subcategory of which is Cartesian, equivalently that the topos is of the form \Sh{\cat{C},P} for \cat{C} a coherent category and $P$ the coherent coverage. (5) In a coherent topos, \Sh{\cat{C},P} say, the full subcategory, $\cat{D}\embedd \Sh{\cat{C},P}$, of coherent objects is a pretopos. Furthermore,  \cat{D} forms a coherent site for \Sh{\cat{C},P}; includes \cat{C} (through the Yoneda embedding); and is a pretopos completion of \cat{C}. Thus one can recover \cat{C} from \Sh{\cat{C},P} up to pretopos completion as the coherent objects. (6) Any compact decidable object in a coherent topos is coherent. The full subcategory of decidable objects in a coherent category is again a coherent category. Accordingly, the full subcategory of compact decidable objects in a coherent topos is a decidable coherent category. We shall say that a topos is \emph{decidable coherent} if it is on the form \Sh{\cat{C},P} for \cat{C} a decidable coherent category and $P$ the coherent coverage.
\begin{definition}\label{Definition: Weakly coherent topos}
Say that a topos is \emph{weakly coherent} if there exists a generating set of compact objects and a finite product of compact objects is compact (so that, in particular, the terminal object is compact).
\end{definition}
\begin{remark}\label{Remark: Coherent and weakly coherent}
Thus a topos is weakly coherent if there exists a generating full subcategory of compact objects closed under finite products, and coherent if there exists generating full subcategory of compact objects closed under finite products and equalizers. An example of a topos which is weakly coherent but not coherent can be constructed by taking presheaves on a small category which has finite products but not fc-equalizers (see \cite{bekekarazerisrosicky:05} for relevant results and the definition of fc-equalizers).
\end{remark}
\begin{lemma}\label{Lemma: Compact decidables in weakly compact topos}
A compact decidable object in a weakly coherent topos is a coherent object. The full subcategory of compact decidable objects in a weakly coherent topos is a  decidable coherent category.
\begin{proof}Straightforward.
\end{proof}
\end{lemma}
In addition to the notion of a compact object in a topos and the usual notion of a compact space, we introduce the notion of a \emph{s-compact} open subgroupoid (`s' for `sheaf'). Say that an open subgroupoid $(U,N)$ of an open topological groupoid \thry{G} is \emph{s-compact} if $U$ is compact in the lattice of open subsets of $U$ that are closed under $N$. Thus, by Lemma \ref{Lemma: GUN subobjects}, an open subgroupoid $N$ of \thry{G} is s-compact precisely when the induced equivariant sheaf \pair{\thry{G},U,N} is compact. Accordingly, we make the following definition:
\begin{definition}\label{Definition: Compact groupoid}
Say that an open topological groupoid \thry{G} is \emph{s-compact} if $G_0$ is compact with respect to open subsets that are closed under $G_1$. Say that \thry{G} is \emph{locally s-compact} if for every open subgroupoid $(U,N)$ and every $x\in U$ there is an open neighborhood $x\in V\subseteq U$ such that $V$ is closed under $N$ and $V$ is compact with respect to open subsets closed under $N$.
\end{definition}
Note that \pair{\thry{G},G_0,G_1} is the terminal object in \Eqsheav{G_1}{G_0}.  Therefore,   Lemma \ref{Lemma: GUN subobjects} immediately gives us the following.
\begin{lemma}\label{Lemma: Compact and locally compact}
\thry{G} is s-compact if and only if (the terminal object in) \Eqsheav{G_1}{G_0} is compact. Furthermore, the following are equivalent:
\begin{enumerate}[(i)]
\item \thry{G} is locally s-compact;
\item every object \pair{\thry{G},U,N} in $\cat{S}_{\thry{G}}\hookrightarrow\Eqsheav{G_1}{G_0}$ is a join of compact subobjects;
\item the compact objects in $\cat{S}_{\thry{G}}$ form a generating set for \Eqsheav{G_1}{G_0};
\item \Eqsheav{G_1}{G_0} has a generating set of compact objects.
\end{enumerate}
\begin{proof}The equivalence of (i) and (ii) is direct from Definition \ref{Definition: Compact groupoid} and Lemma \ref{Lemma: GUN subobjects} (1). (ii)$\Rightarrow$(iii)$\Rightarrow$(iv) is immediate. (iv) implies that any \pair{\thry{G},U,N} can be covered by compact objects, and since images of compact objects are compact and $\cat{S}_{\thry{G}}$ is closed under subobjects, (ii) follows.
\end{proof}
\end{lemma}
An equivariant sheaf topos $\Eqsheav{G_1}{G_0}$ being weakly coherent now translates into the following property of  open subgroupoids of \thry{G}.
Consider a pair of open subgroupoids $(U,N)$ and $(V,M)$. Starting out with the (sub)space $t^{-1}(U)\cap s^{-1}(V)$, form the quotient space
\begin{equation}\label{Equation: Index space}\mathrm{DC}(M,N)=t^{-1}(U)\cap s^{-1}(V)/_{_{N}\sim_{M}}\end{equation}
by the equivalence relation ${_{N}\sim_{M}}$ defined by $(f:v_1\rightarrow u_1){_{N}\sim_{M}}(g:v_2\rightarrow u_2)$ if there exists arrows $n\in N,\ m\in M$ forming a commutative square:
\[\bfig
\square/<-`>`>`<-/<400,300>[u_1`v_1`u_2`v_2;f`n`m`g]
\efig\]
Call $\mathrm{DC}(M,N)$ the \emph{double-coset space} of the open subgroupoids $M$ and $N$.
\begin{lemma}\label{Lemma: Compact bi-index}
Let $\pair{\thry{G},U,N}$ and $\pair{\thry{G},V,M}$ be two objects in \Eqsheav{G_1}{G_0}. Then the product $\pair{\thry{G},U,N}\times\pair{\thry{G},V,M}$ is compact (in \Eqsheav{G_1}{G_0}) if and only if the double-coset space $\mathrm{DC}(M,N)$ is compact (as a topological space).
\begin{proof}
Consider the square
\begin{equation}\label{Equation: Compact bi-index square}\bfig
\square/->>`->>`->>`->>/<1500,500>[s^{-1}(U)\times_{G_0}s^{-1}(V)`t^{-1}(U)\cap s^{-1}(V)`s^{-1}(U)/_{\sim_N}\times_{G_0}s^{-1}(V)/_{\sim_M}`t^{-1}(U)\cap s^{-1}(V)/_{_N\sim_M};c\circ\pair{i,1}`q\times q`k`p]
\efig\end{equation}
where $k$ is the quotient map, and $p$, as the top horizontal map, inverts the left arrow and composes:
\[ p\pair{[f]_{\sim_N},[g]_{\sim_M}}=[f^{-1}\circ g]_{_N\sim_M}\]
Then one easily sees that: i) $p$ is well-defined; ii) the square commutes; iii) all maps of the diagram (\ref{Equation: Compact bi-index square}) are, as indicated, surjective;  iv) the left horizontal map $q\times q$ and the top vertical map $c\circ\pair{i,1}$ are open maps (so, in particular, $p$ is continuous); moreover, v) for all (open) sets $W\subseteq s^{-1}(U)/_{\sim_N}\times_{G_0}s^{-1}(V)/_{\sim_M}$ we have $(c\circ\pair{i,1})( (q\times q)^{-1}(W))= k^{-1}(p(W))$; therefore, vi) the bottom horizontal map $p$ is also an open surjection; and, finally, vii) for a pair of arrows $f:u\rightarrow x \leftarrow v:g$ with $u\in U,\ v\in V$ and an arrow $h:x\rightarrow y$, we have $p\pair{[h\circ f]_{\sim_N},[h\circ g]_{\sim_M}}=p\pair{[f]_{\sim_N},[g]_{\sim_M}}$. From this, it is readily verified that $p^{-1}$ is a frame isomorphism between open subsets of $\mathrm{DC}(M,N)$ and open sets of $s^{-1}(U)/_{\sim_N}\times_{G_0}s^{-1}(V)/_{\sim_M}$ which are closed under composing with arrows from $G_1$, with image along $p$ being the inverse. As such, it yields an isomorphism between open subsets of $\mathrm{DC}(M,N)$ and subobjects of $\pair{\thry{G},U,N}\times\pair{\thry{G},V,M}$, and so the latter is compact if and only if the space $\mathrm{DC}(M,N)$ is.
\end{proof}
\end{lemma}
Putting this together with Lemma \ref{Lemma: Compact and locally compact}, we have:
\begin{proposition}\label{Proposition: T-groupoids}
\Eqsheav{G_1}{G_0} is weakly coherent if and only if \thry{G} is s-compact and locally s-compact and for any s-compact open subgroupoids $(V,M)$ and $U,N$ the double-coset space $\mathrm{DC}(M,N)$ is a compact space.
\begin{proof} By Lemma \ref{Lemma: Compact and locally compact}, Lemma \ref{Lemma: Compact bi-index} and the fact that the existence in a topos \topo{E}  of a generating set $S$ of compact objects such that $A\times B$ is compact for all $A,B\in S$ implies that a binary product of compact objects in \topo{E} is compact.
\end{proof}
\end{proposition}
\begin{remark}
As a special case, we obtain the characterization of coherent groups from  \cite[D3.4]{elephant1}. For a topological group $G$ and open subgroups $M,N\subseteq G$, $\mathrm{DC}(M,N)$ is the discrete space of double cosets, \cterm{NgM}{g\in G}. Since $G$ is automatically s-compact and locally s-compact in the sense of Definition \ref{Definition: Compact groupoid}, the topos of continuous G-sets $\mathbf{Cont}(G)\simeq \Eqsheav{G}{\{\star\}}$  is weakly coherent if and only if these sets are finite for all open subgroups, i.e.\ if $G$ has \emph{finite bi-index}, in the sense of \emph{loc.cit}. And since $\mathbf{Cont}(G)$ is Boolean, it is coherent if and only if it is weakly coherent.
\end{remark}
\begin{definition}\label{Definition: Weakly compact groupoids}

(1) Say that a topological groupoid is \emph{weakly coherent} if it is open and satisfies the conditions of Proposition \ref{Proposition: T-groupoids}.

(2) Say that morphism of open topological groupoids $f:\thry{H}\to\thry{G}$ is
%\emph{weakly coherent} if there exists a generating family $(U_i,N_i)_I$ of compact open subgroupoids of \thry{G} such that $f$ is $N_i$-full, and $(f_0^{-1}(U_i),f_1^{-1}(N_i))$ is compact, for all $i\in I$.
%
\emph{compact} if the induced inverse image functor preserves compact objects.
%
%\emph{weakly compact} if it is a fibration (as in Lemma \ref{Lemma: GUN subobjects} (2)) and has the property that whenever $(U,N)$ is a compact open subgroupoid of \thry{G}, then the open subgroupoid $(f_0^{-1}(U), %f_1^{-1}(N))$ is also compact.

(3) Let \alg{wcGpd} be the category of weakly coherent groupoids and compact morphisms.
\end{definition}
%
%
%Together with Lemma \ref{Lemma: Compact decidables in weakly compact topos}, the reason for the above definition lies in the following, which we use in the next section.
%In the sections to follow, we shall construct a adjunction between decidable coherent categories and weakly coherent groupoids. There is some flexibility with respect to which category of groupoids to consider. For instance, we %could restrict to the subcategory of decidable coherent groupoids stated at the end of this section. However, we shall go with the simpler and more general property of being weakly coherent. Also, if the intrinsic %characterization of weakly coherent morphisms seems overly complicated, the adjunction could equally well be constructed using the category of weakly coherent groupoids and morphisms between them with the property that the %induced inverse image functors preserve compact objects. However, we shall go with the notion that we have an intrinsic characterization of. We note that weakly coherent morphisms have the required property:
%
Note that the inverse image functor $f^*:\Eqsheav{G_1}{G_0}\to\Eqsheav{H_1}{H_0}$ induced by a compact morphism of weakly coherent groupoids $f:\thry{H}\to\thry{G}$ restricts to a (coherent) functor between the (coherent) subcategories of compact decidable objects. We end this section by briefly considering decidable objects.
%
%\begin{lemma}\label{Lemma: Weakly compact morphisms preserve compact decidable objects}
%Let $f:\thry{H}\to\thry{G}$ be a morphism in \alg{wcGpd}. Then the induced inverse image functor $f^*:\Eqsheav{G_1}{G_0}\to\Eqsheav{H_1}{H_0}$ restricts to a (coherent) functor between the full subcategories of compact decidable %objects.
%\begin{proof} Since (by Lemma \ref{Lemma: GUN subobjects} and \ref{Lemma: Generating set of GUNS}) the compact \pair{\thry{G},U_i,N_i}s generate \Eqsheav{G_1}{G_0} and $f^*$ sends them to compact objects, $f^*$ takes compact %objects to compact objects. And all inverse image functors preserve decidable objects.
%\end{proof}
%\end{lemma}
%
%
%Before proceeding, we note the following.
%
\begin{lemma}\label{Lemma: Decidable GUNS}
An object of the form \pair{\thry{G},U,N} is decidable if and only if $N\subseteq s^{-1}(U)\cap t^{-1}(U)$ is clopen (that is, if $N$ is a closed subset of $s^{-1}(U)\cap t^{-1}(U)$).
\begin{proof}The bottom horizontal maps in the following diagram
\[\bfig
\square/<<-`>`>`<<-/<1400,400>[\Delta`\sim_{N}`s^{-1}(U)/_{\sim_N}\times_{G_0}s^{-1}(U)/_{\sim_N}`s^{-1}(U)\times_{G_0}s^{-1}(U);`\subseteq`\subseteq`q\times q]
\square(1400,0)/->>`>`>`->>/<1000,400>[\sim_{N}`N`s^{-1}(U)\times_{G_0}s^{-1}(U)`s^{-1}(U)\cap t^{-1}(U);``\subseteq`c\circ\pair{i,1_{G_1}}]
\place(1300,300)[\pbangler]
\place(1500,300)[\pbangle]
\efig\]
are both open surjections.
\end{proof}
\end{lemma}
Say, accordingly,  that an open subgroupoid $(U,N)$ is \emph{decidable} if $N$ is a closed subset of  $s^{-1}(U)\cap t^{-1}(U)$. Say that an open topological groupoid is \emph{coherent decidable} if it satisfies the conditions of the following proposition.
\begin{proposition}\label{Proposition: Coherent decidable groupoid}
\Eqsheav{G_1}{G_0} is coherent decidable if and only if \thry{G} is s-compact and there exist a generating set \cterm{N_i\subseteq G_1}{i\in I} of s-compact decidable subgroupoids such that $\mathrm{DC}(N_i,N_j)$ is a compact space for all $i,j\in I$.
\begin{proof}The only if direction follows by Lemma \ref{Lemma: Compact and locally compact}, the fact that a subobject of a decidable object is decidable, and that compact objects are closed under finite products in a coherent topos. The if direction follows since the generating set of s-compact open subgroupoids induces a generating set $S$ of compact decidable objects in \Eqsheav{G_1}{G_0} such that $A\times B$ is compact for $A,B\in S$. In particular, therefore, \Eqsheav{G_1}{G_0} is weakly coherent.   Since a finite product of decidable objects is decidable and a complemented subobject of a compact object is compact, any finite limit of objects from $S$ is again compact and decidable.
\end{proof}
\end{proposition}
\subsection{Representation Theorem for Decidable Coherent Categories}%%%%%%%%%%%%%%%%%%%%%%%%%%%%%%%%%%%%%%%%%%%%%%%%%%%%%%%%%%%%%%%%%%
\label{Subsection: Representation theorem for decidable coherent categories}
Since one can pass back and forth between coherent theories and categories by taking the theories of categories
and the syntactic categories of theories, Theorem \ref{Theorem: Representation theorem} translates to a representation result for decidable coherent categories. For this and the following sections, we fix a choice of set from which to construct models as follows. Chose a regular cardinal $\kappa$ and let $\mathds{S}$ be the set of all hereditarily smaller than $\kappa$ sets. Let  $\Sets_\kappa$ be the small, decidable coherent category of sets with elements from $\mathds{S}$. Thus, in the notation of e.g.\ \cite{kunen:80}, $\mathds{S}=H(\kappa)$ and $\Sets_\kappa$ is the full subcategory of \Sets\ the objects of which are subsets of $H(\kappa)$.  We refer to such sets as $\kappa$-\emph{small} (although unless $\kappa$ is countable, such sets need not themselves have cardinality $\leq \kappa$). We translate Theorem \ref{Theorem: Representation theorem} to the setting of decidable coherent categories and groupoids of $\Sets_{\kappa}$-valued coherent functors and invertible natural transformations
between them, and use this form of the theorem in the construction of a duality theorem for decidable coherent categories.

Let \cat{D} be a (small) decidable coherent category, that is, a category with finite limits, images, stable
covers, stable finite unions of subobjects, and complemented diagonals (\cite[A1.4]{elephant1}). We say that \cat{D} \emph{has enough $\kappa$-small models} if the
coherent functors from \cat{D} to $\Sets_{\kappa}$,
\[\cat{D}\to\Sets_{\kappa}\]
jointly reflect covers, in the sense that
%
%\begin{definition}
%For an uncountable cardinal $\kappa$, we say that \cat{D} \emph{has a saturated set of $<\kappa$-models} if the
%set of coherent functors $\cat{D}\to H(\kappa):=\Sets_{\kappa}$ (the hereditarily $<\kappa$ sets, see e.g.\
%\cite[IV,6]{kunen:80}) jointly reflect covers, where \cat{D} is considered to have the coherent coverage and
%$\Sets_{\kappa}$ has its canonical coverage.
%
%\end{definition}
%
%Explicitly, this means that
%
for any family of arrows $f_i:C_i\rightarrow C$ in \cat{D}, if for all $M:\cat{D}\to\Sets_{\kappa}$ in
$X_{\cat{D}}$
\[ \bigcup_{i\in I}\funksjon{Im}{M(f_i)}= M(C)\]
then there exists $f_{i_1}, \ldots, f_{i_n}$ such that $\funksjon{Im}{f_{i_1}}\vee\ldots\vee
\funksjon{Im}{f_{i_n}}=C$.
%
%It is sufficient that (the set of objects and arrows of) a coherent category \cat{D} is of cardinality $<\kappa$
%for it to have a saturated set of $<\kappa$ models, see Lemma \ref{Lemma: DC, saturated set of kappa models,
%sufficient consition} below.
%
\begin{definition}\label{Definition: DC, DCkappa}
Let \alg{dCoh} be the category of small decidable coherent categories with coherent functors between them. Let
\alg{dCoh_{\kappa}} be the full subcategory of those categories which have enough $\kappa$-small models, i.e.\ such that the coherent functors to
$\Sets_{\kappa}$ jointly reflect covers.
\end{definition}
Note that any coherent category which is of cardinality $\leq\kappa$ is in $\alg{dCoh}_{\kappa}$, as are all
distributive lattices.
\begin{definition}\label{Definition: DC, Coherent topology}
For \cat{D} in $\alg{dCoh}_{\kappa}$:
\begin{enumerate}
\item Let $X_{\cat{D}}$ be the set of coherent functors from \cat{D} to $\Sets_{\kappa}$,
\[X_{\cat{D}}=\homset{\alg{dCoh}}{\cat{D}}{\Sets_{\kappa}}.\]
\item Let $G_{\cat{D}}$ be the set of invertible natural transformations between functors in $X_{\cat{D}}$, with
$s$ and $t$ the source and target, or domain and codomain, maps,
\[s,t:G_{\cat{D}}\rightrightarrows X_{\cat{D}}\]
Denote the resulting groupoid by $\thry{G}_{\cat{D}}$.
%thus forming the (discrete) groupoid:
%
%\[ \thry{G}_{\cat{D}}=\underline{\mng{Hom}}^*_{\alg{dCoh}_{\kappa}}(\cat{D},\Sets_{\kappa}).  \]
%
\item The \emph{coherent topology} on $X_{\cat{D}}$ is given by taking as a subbasis the collection of sets of the
form,
\begin{align*}
&\bopen{\alg{f},\alg{a}} =\bopen{\pair{f_1:A\rightarrow B_1,\ldots,f_n:A\rightarrow B_n},\pair{a_1,\ldots,a_n}}\\
  &=\cterm{M\in X_{\cat{D}}}{\fins{x\in M(A)}M(f_1)(x)=a_1\wedge \ldots \wedge M(f_n)(x)=a_n}
\end{align*}
for a finite span of arrows
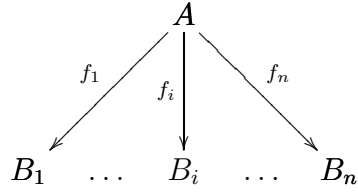
\[\bfig
\Atriangle/>`>`/<500,500>[A`B_1`B_n;f_1`f_n`]\place(250,0)[\ldots]\place(750,0)[\ldots]
\morphism(500,500)<0,-500>[A`B_i;f_i]
\efig\]
in \cat{D} and $a_1,\ldots,a_n\in \mathds{S}$.  Let the \emph{coherent topology} on $G_{\cat{D}}$ be the
coarsest topology such that $s,t:G_{\cat{D}}\rightrightarrows X_{\cat{D}}$ are both continuous and all sets of
the form
\[\bopen{A,a\mapsto b}=\cterm{\alg{f}:M\rightarrow N}{a\in M(A) \wedge f_A(a)=b}\]
are open, for $A$ an object of \cat{D} and $a,b\in \mathds{S}$.
\end{enumerate}
\end{definition}
\begin{remark}\label{Remark: X, propositional case}
Note that if \cat{D} is a Boolean algebra and we require coherent functors into \Sets\ to send the terminal object to the distinguished terminal object $\{\star\}$ in \Sets, then $X_{\cat{D}}$ is the Stone space of \cat{D}.
\end{remark}
For \cat{D} in \alg{dCoh_{\kappa}}, we have the decidable coherent theory $\theory_{\cat{D}}$ of \cat{D}, and
its syntactic category, \synt{C}{\theory_{\cat{D}}} (as described in Section \ref{Subsection: Theories and Models}). Sending an object, $D$, in \cat{D} to the object \syntob{x:D}{\top} in \synt{C}{\theory_{\cat{D}}},
and an arrow $f:C\rightarrow D$ to \syntob{x:C,y:D}{f(x)=y}, defines a functor
\[ \zeta_{\cat{D}}:\cat{D}\to\synt{C}{\theory_{\cat{D}}} \]
which is one half of an equivalence, the other half being the (or a choice of) canonical
$\theory_{\cat{D}}$-model in \cat{D}.

Now, any $\theory_{\cat{D}}$-model \alg{M} with elements from $\mathds{S}$ can be seen as a coherent functor,
$\alg{M}:\synt{C}{\theory_{\cat{D}}}\to\Sets_{\kappa}$. Composition with $\zeta_{\cat{D}}$
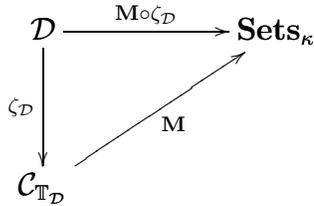
\[\bfig
\ptriangle/>`>`<-/<750,500>[\cat{D}`\Sets_{\kappa}`\synt{C}{\theory_{\cat{D}}};\alg{M}\circ
\zeta_{\cat{D}}`\zeta_{\cat{D}}`\alg{M}]
\efig\]
induces restriction functions
\[\bfig
\square|arrb|/>`@{>}@<5pt>`@{>}@<5pt>`>/<1000,500>[G_{\theory_{\cat{D}}}`G_{\cat{D}}`X_{\theory_{\cat{D}}}`
X_{\cat{D}}; \phi_1`t`t`\phi_0]
\square|allb|/>`@{>}@<-5pt>`@{>}@<-5pt>`>/<1000,500>[G_{\theory_{\cat{D}}}`G_{\cat{D}}`X_{\theory_{\cat{D}}}`
X_{\cat{D}}; \phi_1`s`s`\phi_0]
\efig\]
commuting with source and target (as well as composition and insertion of identities) maps.
\begin{lemma}
The maps $\phi_0$ and $\phi_1$ are homeomorphisms of spaces.
\begin{proof}
Any coherent functor $M:\cat{D}\to\Sets_{\kappa}$ lifts to a unique $\theory_{\cat{D}}$-model
$\alg{M}:\synt{C}{\theory_{\cat{D}}}\to \Sets_{\kappa}$, to yield an inverse $\psi_0:X_{\cat{D}}\rightarrow
X_{\theory_{\cat{D}}}$ to $\phi_0$. Similarly, an invertible natural transformation of functors $f:M\rightarrow
N$ lifts to a unique $\theory_{\cat{D}}$-isomorphism $\alg{f}:\alg{M}\rightarrow \alg{N}$ to yield an inverse
$\psi_1:G_{\cat{D}}\rightarrow G_{\theory_{\cat{D}}}$ to $\phi_1$. We verify that these four maps are all
continuous. For a subbasic open
\[
U=\bopen{\pair{f_1:A\rightarrow B_1,\ldots,f_n:A\rightarrow B_n},\pair{a_1,\ldots,a_n}} \subseteq X_{\cat{D}}\]
we have
\[
\phi_0^{-1}(U)= \bopen{\syntob{y_1:B_1,\ldots,y_n:B_n}{\fins{x:A}\bigwedge_{1\leq i\leq n}f_i(x)=y_i},\vec{a}}
\]
so $\phi_0$ is continuous. To verify that $\psi_0$ is continuous, there are two cases to consider, namely
non-empty and empty context. For basic open
\[\bopen{\syntob{x:A_1,\ldots, x_n:A_n}{\phi}, \pair{a_1,\ldots,a_n}}
\subseteq X_{\theory_{\cat{D}}}\]
the canonical interpretation of $\theory_{\cat{D}}$ in \cat{D} yields a subobject of a product in \cat{D},
\[\csem{x:A_1,\ldots, x_n:A_n}{\phi}\subobject A_1\times\ldots\times A_n\to^{\pi_i}A_i.\]
Choose a monomorphism $r:R\mon A_1\times\ldots\times A_n$ representing that subobject. Then
\begin{align*}
&\psi_0^{-1}(\bopen{\syntob{x:A_1,\ldots, x_n:A_n}{\phi}, \pair{a_1,\ldots,a_n}})\\
&= \bopen{\pair{\pi_1\circ r:R\rightarrow A_1,\ldots,\pi_n\circ r: R\rightarrow A_n}, \pair{a_1,\ldots,a_n}}
\end{align*}
and it is clear that this is independent of the choice of product diagram and of representing monomorphism. For
the empty context case, consider a basic open $U=\bopen{\syntob{}{\varphi}, \star}$, where $\varphi$ is a
sentence of $\theory_{\cat{D}}$ and $\star$ is the element of the distinguished terminal object of \Sets\
(traditionally $\star=\emptyset$, notice that  any $\bopen{\syntob{}{\varphi},a}$ with $a\neq\star$ is
automatically empty). The canonical interpretation of $\varphi$ in \cat{D} yields a subobject of a terminal
object, $\sem{\varphi}\subobject 1$. Choose a representative monomorphism $r:R\mon 1$. Then, independently of
the choices made,
\[ \psi_0^{-1}(U)=\bigcup_{a\in \mathds{S}}\bopen{r:R\rightarrow 1,a}.\]
So $\psi_0$ is continuous. With $\phi_0$ continuous, it is sufficient to check $\phi_1$ on subbasic opens of the
form $U=\bopen{A, a\mapsto b}\subseteq G_{\cat{D}}$. But
\[
\phi_1^{-1}(U)=\left(\begin{array}{c}
-  \\
  \syntob{x:A}{\top}:a \mapsto b   \\
 -
\end{array}\right)
\]
so $\phi_1$ is continuous. Similarly, it is sufficient to check $\psi_1$ on subbasic opens of the form
\[
U=\left(\begin{array}{c}
-  \\
  \syntob{x:A}{\top}:a \mapsto b   \\
 -
\end{array}\right)
\]
but $\psi_1^{-1}(U)=\bopen{A,a \mapsto b}$, so $\psi_1$ is continuous.
\end{proof}
\end{lemma}
\begin{corollary}
Definition \ref{Definition: DC, Coherent topology} yields, for a decidable coherent category \cat{D}, a
topological groupoid $\thry{G}_{\cat{D}}$ such that
\[\thry{G}_{\cat{D}} \cong \thry{G}_{\theory_{\cat{D}}}\]
in the category \alg{Gpd}.
\end{corollary}
We can now state the main representation result of this section (which forms, then, another variation of the representation result of \cite{butz:98b} restricted to decidable coherent categories). We use the notations \Eqsheav{G_{\cat{D}}}{X_{\cat{D}}} and (the shorter) \Sh{\thry{G}_{\cat{D}}} interchangeably.
\begin{Theorem}\label{Theorem: DC, Representation result}
For a decidable coherent category with enough $\kappa$-small models, the topos of coherent sheaves on
\cat{D} is equivalent to the topos of equivariant sheaves on the topological groupoid $\thry{G}_{\cat{D}}$ of
models and isomorphisms equipped with the coherent topology,
\[ \sh{D}\simeq \Sh{\thry{G}_{\cat{D}}}.\]
\begin{proof}
The equivalence $\zeta_{\cat{D}}:\cat{D}\to \synt{C}{\theory_{\cat{D}}}$ yields an equivalence
$\sh{D}\simeq\Sh{\synt{C}{\theory_{\cat{D}}}}$, whence
\[ \sh{D}\simeq \Sh{\synt{C}{\theory_{\cat{D}}}}\simeq \Sh{\thry{G}_{\theory_{\cat{D}}}}\cong \Sh{\thry{G}_{\cat{D}}}\]
by Theorem \ref{Theorem: Representation theorem}.
\end{proof}
\end{Theorem}
\subsection{The Semantical Functor $\Mod$}%%%%%%%%%%%%%%%%%%%%%%%%%%%%%%%%%%%%%%%%%%%%%%%%%%%%%%%%%%%%%%%%%%%%%%%%%%%%%%%%%%%%%%%%%%%%%%%%%%%%%%%%%%%%%%%%%%%%%%%%%%%%%%%%%%%%%%%%%%%%%%%%%%%%%%%%%%%%%%%%%%%%%%%%%%%%%%%%%
\label{Subsection: DC, The Semantical Functor}
%
%
%We proceed to construct a `syntax-semantics' adjunction between the category \alg{dCoh_{\kappa}} (syntax) and a
%subcategory of topological groupoids (semantics).
The mapping of a decidable coherent category (with enough $\kappa$-models) \cat{D}  to its $\kappa$-small models
\[
\Mod(\cat{D}) = \homset{\alg{dCoh}}{\cat{D}}{\Sets_{\kappa}},
\]
regarded as a groupoid of natural isomorphisms and equipped with the coherent topology, as in Definition \ref{Definition: DC, DCkappa}, is the object part of a contravariant functor into groupoids. Given a coherent functor
\[
F:\cat{A}\to\cat{D}
\]
between two objects of \alg{dCoh_{\kappa}}, precomposition with $F$,
\[\bfig
\morphism<350,0>[\cat{A}`\cat{D};F] \morphism(350,0)/@{>}@<7pt>/[\cat{D}`\Sets_{\kappa};M]
\morphism(350,0)|b|/@{>}@<-7pt>/[\cat{D}`\Sets_{\kappa};N] \place(550,0)[\Downarrow]
\efig\]
yields a `restriction' morphism of (discrete) groupoids
\begin{equation}\label{Equation: BC, Composition with F gives maps of groupoids}
\bfig
\square|arrb|/>`@{>}@<5pt>`@{>}@<5pt>`>/<500,250>[G_{\cat{D}}`G_{\cat{A}}`X_{\cat{D}}`X_{\cat{A}}; f_1`t`t`f_0]
\square|allb|/>`@{>}@<-5pt>`@{>}@<-5pt>`>/<500,250>[G_{\cat{D}}`G_{\cat{A}}`X_{\cat{D}}`X_{\cat{A}};
f_1`s`s`f_0]
\efig\end{equation}
We verify that $f_0$ and $f_1$ are both continuous. We have
\begin{align*}
&U=\bopen{\pair{g_1:A\rightarrow B_1,\ldots,g_n:A\rightarrow B_n},\pair{a_1,\ldots,a_n}}\subseteq X_{\cat{A}}\\
&\Rightarrow  f_0^{-1}(U)=\bopen{\pair{F(g_1):FA\rightarrow FB_1,\ldots,F(g_n):FA\rightarrow FB_n},\pair{a_1,\ldots,a_n}}
\end{align*}
and for basic open $U=\bopen{C,a\mapsto b}\subseteq G_{\cat{A}}$, we see that
$f_1^{-1}(U)=\bopen{F(C),a\mapsto b}$.
Thus composition with $F$ yields a morphism of topological groupoids,
$f:\thry{G}_{\cat{D}}\to\thry{G}_{\cat{A}},$ and thereby we get a contravariant functor, %
\[
\Mod: \alg{dCoh}^{\mng{op}}_{\kappa}\to \alg{Gpd}.
\]
which we shall refer to as the \emph{semantic} functor. We verify that the semantic functor factors through \alg{wcGpd}. Define the embedding of \cat{D} into \Sh{\thry{G}_{\cat{D}}} as the composition
\[\cat{Y}_{\cat{D}}:\cat{D}\to^{\zeta_{\cat{D}}}\synt{C}{T_{\cat{D}}}\to^{\cat{M}}\sh{\thry{G}_{\synt{C}{T_{\cat{D}}}}}\cong\Sh{\thry{G}_{\cat{D}}}\]
\begin{lemma}\label{Lemma: Restriction morphisms}Let $F:\cat{C}\to\cat{D}$ be a morphism in \alg{dCoh_{\kappa}}, and $\Mod(F)=f:\thry{G}_{\cat{D}}\to\thry{G}_{\cat{C}}$. Then the following commutes
\[\bfig
\square<700,350>[\cat{C}`\cat{D}`\Sh{\thry{G}_{\cat{C}}}`\Sh{\thry{G}_{\cat{D}}};F`\cat{Y}_{\cat{C}}`\cat{Y}_{\cat{D}}`f^*]
\efig\]
where $f^*$ is the induced inverse image functor.
\begin{proof} By (a straightforward, if a bit tedious, check of the) construction.
\end{proof}
\end{lemma}

\begin{lemma}\label{Lemma: Mod factors through weakly coherent groupoids}
The functor $\Mod: \op{\alg{dCoh}_{\kappa}}\to \alg{Gpd}$ factors through the category \alg{wcGpd}.
\begin{proof} The category $\Sh{\cat{D}}\simeq \Sh{\thry{G}_{\cat{D}}}$ is coherent, so weakly coherent. For a coherent functor $F:\cat{C}\to\cat{D}$, Lemma \ref{Lemma: Restriction morphisms} implies that the inverse image functor $f^*:\Sh{\thry{G}_{\cat{C}}}\to\Sh{\thry{G}_{\cat{D}}}$ induced by $\Mod(F)=f: \thry{G}_{\cat{D}}\to\thry{G}_{\cat{C}}$ takes an object $\cat{Y}_{\cat{C}}(C)$ to $\cat{Y}_{\cat{D}}(F(C))$. By Lemma \ref{Lemma: Definables are GUNS}, $\cat{Y}_{\cat{C}}(C)$ is isomorphic to the equivariant sheaf induced by the subobject $(U,N)$ with $U=\bopen{1_C, a}$ for an object $a\in \mathds{S}$, and
\[N=\left(\begin{array}{c}
1_C, a  \\
  C:a \mapsto a   \\
 1_C, a
\end{array}\right)\]
It is now straightforward to see that $F(C)$ is isomorphic the sheaf induced by $(f_0^{-1}(U),f_1^{-1}(N))$ and that the conditions of Definition \ref{Definition: Weakly compact groupoids} (2) are fulfilled.
%
%
%so it remains only to verify that $f$ is a fibration [IT ISN'T]. Consider the following diagram of coherent functors and an invertible natural transformation;
%\[\bfig
%\Atriangle/<-`>`>/<1000,300>[\cat{D}`\cat{C}`\Sets_{\kappa};F`\alg{M}`\alg{K}]
%\place(1000,150)[\Uparrow]
%\place(1075,125)[\Phi]
%\efig\]
%We need to construct a coherent functor $\alg{N}:\cat{D}\to \Sets_{\kappa}$ and an invertible natural transformation $\Psi:\alg{N}\Rightarrow \alg{M}$ such that restricting along $F$ sends \alg{M} to \alg{K} and $\Psi$ to $\Phi$. %For each object $D$ in the image of $F$ fix a choice of object $C$ such that $F(C)=D$, and similarly for arrows.  Define \alg{N} on objects by setting $\alg{N}(F(C))= K(C)$ (for a choice of $C$ if necessary) and %$\alg{N}(D)=\alg{M}(D)$ if $D$ is not in the image of $F$. Similarly, $\alg{N}(F(g))=K(g)$ and $\alg{N}(g)=\alg{M}(g)$ if neither the domain nor codomain of $g$ is in the image of $F$. If, say, the codomain of $g: D\rightarrow %F(C)$ then set $\alg{N}(g)= \Phi_C^{-1}\circ \alg{M}(g)$, and similarly if the domain of $g$ is in the image of $F$ but not the codomain. This clearly defines a functor. Define the morphism $\Psi_{F(C)}:\alg{N}(F(C))\rightarrow %\alg{M}(F(C))$ to be $\Phi_C$ and $\Psi_{D}:\alg{N}(D)\rightarrow \alg{M}(D)$ to be $1_D$ if $D$ is not in the image of $F$. This clearly defines an invertible natural transformation such that restricting along $F$ sends $\Psi$ %to $\Phi$.
\end{proof}
\end{lemma}
%
%$\Mod(F)=f:\thry{G}_{\cat{D}}\to\thry{G}_{\cat{C}}$. Then the following commutes
%\[\bfig
%\square<700,350>[\cat{C}`\cat{D}`\Sh{\thry{G}_{\cat{C}}}`\Sh{\thry{G}_{\cat{D}}};F`\cat{Y}_{\cat{C}}`\cat{Y}_{\cat{D}}`f^*]
%
\subsection{The Syntactical Functor $\Form$}%%%%%%%%%%%%%%%%%%%%%%%%%%%%%%%%%%%%%%%%%%%%%%%%%%%%%%%%%%%%%%%%%%%%%%%%%%%%%%%%%%%%%%%%%%%%%%%%%%%%%%%%%%%%%%%%%%%%%%%%%%%%%%%%%%%%%%%%%%%%%%%%%%%%%%%%%%%%%%%%%%%%%%%%%%%%%%%%%%%%%%%%
\label{Subsection: DC, The Syntactical Functor}
We construct an adjoint to the semantical functor \Mod\ from the category \alg{wcGpd} of weakly coherent groupoids.
By Theorem \ref{Theorem: DC, Representation result}, $\Mod(\cat{D})$ is a (weakly) coherent groupoid, for any \cat{D} in $\alg{dCoh}_{\kappa}$, and we can recover \cat{D} from $\Mod(\cat{D})$, up to pretopos completion, by taking the compact decidable objects in \Sh{\Mod(\cat{D})}. For arbitrary coherent groupoids, however, this procedure will yield an decidable coherent category, but not necessarily one in $\alg{dCoh}_{\kappa}$, i.e.\ not necessarily with enough $\kappa$-small models. However, one can use the groupoid, $\Sets^*_{\kappa}$ of smaller than $\kappa$ sets and bijections to classify a suitable collection of objects, as we now proceed to describe.
\subsubsection{The Decidable Object Classifier}%%%%%%%%%%%%%%%%%%%%%%%%%%%%%%%%%%%%%%%%%%%%%%%%%%%%%%%%%%%%%%%%%%%%%%%%%%%%%%%%%%%%%%%%%%%%%%%%%%%%%%%%%%%%%%%%%%%%%%%%%%%%%%%%%%%%%%%%%%%%%%%%%%%%%%%%%%%%%%%%%%%%%%%%%%%%%%%%%%%%%
\label{Subsubsection: DC, The Decidable Object Classifier}
\begin{definition}\label{Definition: The groupoid of sets}
The topological groupoid \thry{S} consists of $\kappa$-small sets with bijections between
them, equipped with topology as follows. The topology on the set of objects, $S_0$, is generated by the empty
set and basic opens of the form
\[
\bopen{a_1,\ldots,a_n}:=\cterm{A\in \Sets_{\kappa}}{a_1,\ldots,a_n\in A}
\]
for $a_1,\ldots,a_n\in\mathds{S}$, while the topology on the set, $S_1$ of bijections between $\kappa$-small sets is the coarsest
topology such that the source and target maps $s,t:S_1\rightrightarrows S_0$ are both continuous, and containing
all sets of the form
\[
\bopen{a\mapsto b}:=\cterm{f:A\to^{\cong} B\ \textnormal{in}\ \Sets_{\kappa}}{a\in A\wedge f(a)=b}
\]
for $a,b\in\mathds{S}$.
\end{definition}
We recognize $\thry{S}$ as the groupoid of models and isomorphisms for the decidable coherent theory, $\theory_{\neq}$, of equality and inequality (with the obvious signature and axioms). We state this for reference.
\begin{lemma}
There is an isomorphism $\thry{S}\cong \thry{G}_{\theory_{\neq}}$ in \alg{Gpd}.
\begin{proof} Compare Definitions \ref{Definition> Logical topology} and \ref{Definition: The groupoid of sets}.
\end{proof}
\end{lemma}
The topos \Sh{\thry{S}} of equivariant sheaves on \thry{S}, therefore, classifies decidable objects, as
$\Sh{\thry{S}}\simeq \Sh{\thry{G}_{\theory_{\neq}}}\simeq \Sh{\synt{C}{T_{\neq}}}$, where the last equivalence is by Theorem \ref{Theorem: Representation theorem}.

\begin{corollary}\label{Corollary: Sh(S) coh}
The groupoid $\thry{S}$ of small sets is weakly coherent.
\end{corollary}

\begin{corollary}\label{Corollary: DC, decidable object class equivalences}
There is an equivalence of toposes,
\[\Sets^{\mng{Fin}_i}\simeq \Sh{\synt{C}{T_{\neq}}}\simeq \Sh{\thry{S}}\]
where $\mng{Fin}_i$ is the category of finite sets and injections.
\begin{proof}
$\Sets^{\mng{Fin}_i}\simeq \Sh{\synt{C}{T_{\neq}}}$ by \cite[VIII, Exc.7--9]{maclane:92}.
\end{proof}
\end{corollary}
%
%
%\begin{remark}\label{Remark: Boolean object classifier} Add description of Boolean object classifier?
%\end{remark}
%
%
\begin{definition}\label{Stipulation: U} We fix the generic decidable object, $\cat{U}$, in \Sh{\thry{S}} to be the set $U\rightarrow S_0$ over $S_0$ such that the fiber over a set $A\in S_0$ is the set $A$ (i.e.\ $U = \coprod_{A\in S_0}A$), and the action by the set $S_1$ of isomorphisms is just applying those isomorphisms to the fibers.  Thus, forgetting the topology, \cat{U} is simply the inclusion $\thry{S}\embedd\Sets$. The topology on $U$ is the coarsest such that the projection $U\rightarrow S_0$ is continuous and such that for each $a\in \mathds{S}$ the image of the section $s_a:\bopen{a}\rightarrow U$ defined by $s_a(A)=a$ is an open set.
%
% the definable sheaf
%$\pair{\csem{x}{\top}_{X_{\theory_{\neq}}}\rightarrow X_{\theory_{\neq}}\cong S_0,\theta_{\syntob{x}{\top}}}$, which we also abbreviate as $\cat{U}=\pair{U\rightarrow S_0,\theta_U}$ (see the following remark).
\end{definition}
\begin{remark}\label{Remark: U Moerdijk descripiton}
Comparing Definitions \ref{Stipulation: U} and \ref{Definition: Logical topology on sheaves}, we see that \cat{U} can also be described as the definable sheaf $\pair{\csem{x}{\top}_{X_{\theory_{\neq}}}\rightarrow X_{\theory_{\neq}}\cong S_0,\theta_{\syntob{x}{\top}}}$.

\end{remark}
\subsubsection{Formal Sheaves}%%%%%%%%%%%%%%%%%%%%%%%%%%%%%%%%%%%%%%%%%%%%%%%%%%%%%%%%%%%%%%%%%%%%%%%%%%%%%%%%
\label{Subsubsection: Morphisms into S}
We use the groupoid $\thry{S}$ of (small) sets to recover an object in $\alg{dCoh}_\kappa$ from a coherent groupoid by considering the set $\homset{\alg{wcGpd}}{\thry{G}}{\thry{S}}$ of morphisms into \thry{S}. (Consider the analogy to the propositional case, where the algebra of clopen sets of a Stone space is recovered by homming into the discrete space $2$.)  First, however, a note on notation and bookkeeping: because we shall be concerned with functors into $\Sets_{\kappa}$---a subcategory of \Sets\ which is not closed under isomorphisms---we fix certain choices \emph{on the nose}, instead of working up to isomorphism or assuming a canonical choice as arbitrarily given. Without going into the (tedious) details of the underlying book-keeping, the upshot is that we allow ourselves to treat (the underlying set over $G_0$ and action of) an equivariant sheaf over a groupoid, $\thry{G}$ as a functor $\thry{G}\to\Sets$ in an intuitive way. In particular, we refer to the definable set $\csem{\vec{x}}{\phi}^{\alg{M}}$ as the fiber of $\sox{\vec{x}}{\phi}\rightarrow X_{\theory}$ over $\alg{M}$, although that is not strictly speaking the fiber (strictly speaking the fiber is, according to our definition, the set  $\{\alg{M}\}\times\csem{\vec{x}}{\phi}^{\alg{M}}$). Moreover, we chose the induced inverse image functor $f^*:\Sh{\thry{H}}\to\Sh{\thry{G}}$ induced by a morphism, $f:\thry{G}\to\thry{H}$, of topological groupoids so that, for $A\in \Sh{\thry{H}}$ the fiber over $x\in G_0$ of $f^*(A)$ is the same set as the fiber of $A$ over $f_0(x)\in H_0$. For example, and in particular, any morphism of topological groupoids $f:\thry{G}\to\thry{S}$ induces a geometric morphism $f:\Sh{\thry{G}}\to\Sh{\thry{S}}$ the inverse image part of which takes the generic decidable object \cat{U} of \ref{Stipulation: U} to an (equivariant) sheaf over \thry{G},
\[
\bfig
\square(0,0)|allb|<600,400>[ A`U=\csem{x}{\top}_{X_{\theory_{\neq}}}`G_{0}`S_{0};```f_0]
%
%\square(0,0)|arrb|/>`@{>}@<3pt>`@{>}@<3pt>`>/<600,400>[ G_{1}`H_{1}`G_{0}`H_{0};f_1`t`t`f_0]
%
\place(100,300,)[\pbangle]
\efig
\]
such that the fiber $A_x$ over $x\in G_0$ is the same set as the fiber of $U$ over $f_0(x)$, which is the set $f_0(x)\in S_0=\Sets_{\kappa}$. We hope that this is sufficiently intuitive so that we may hide the underlying book-keeping needed to make sense of it. With this in mind, then, we make the following stipulation.
\begin{definition}\label{Definition: Form(G)}
For a weakly coherent groupoid \thry{G}, let $\Form(\thry{G})\hookrightarrow\Sh{\thry{G}}$ be the full subcategory consisting of objects of the form $f^*(\cat{U})$ for all $f:\thry{G}\to \thry{S}$ in \alg{wcGpd}. Such objects will be called \emph{formal sheaves}.
\end{definition}
%
%\begin{remark}
%The term ``Stone sheaf" is perhaps not ideal here, since these objects actually correspond to the clopen sets of the propositional case, and thus the collection $\Form(-)$ of all of them returns the ``algebra," rather than the ``(Stone) space".  Nonetheless, it seems reasonable to use it for this central concept of the theory.
%\end{remark}
%\marginpar{Henrik: what do you think about this terminology?}
%
Observe that:
\begin{lemma}\label{Lemma: enough to check coh on U}
For a weakly coherent groupoid \thry{G}, a morphism $f:\thry{G}\to \thry{S}$  of topological groupoids is in \alg{wcGpd} if (and only if) the classified object $f^*(\cat{U})\in \Sh{\thry{G}}$ is compact.
\end{lemma}
\begin{proof}The objects in the image of the embedding $\synt{C}{T_{\neq}}\embedd \Sh{\thry{S}}$ (which sends \syntob{x}{\top} to $\cat{U}$) is a generating set $Y$ of compact decidable objects.  If $f^*(\cat{U})$ is compact, and therefore compact decidable, then, by Lemma \ref{Lemma: Compact decidables in weakly compact topos}, so is $f^*(C)$ for any $C\in Y$. Since $Y$ is a generating set, $f^*$ preserves compact objects.

%By Corollary \ref{Corollary: DC, decidable object class equivalences} we have that $\synt{C}{T_{\neq}}$ is a site for \Sh{\thry{S}}. Write \cat{S} for the image of the full and faithfull inclusion of $\synt{C}{T_{\neq}}$ into %\Sh{\thry{S}}. We have that \cat{U} is in \cat{S}, as the image of the object \syntob{x}{\top} in $\synt{C}{T_{\neq}}$. Since $\cat{U}$ is decidable, so is $f^*(\cat{U})$.  Therefore, since $f^*(\cat{U})$ is compact, it is %coherent. Now, as an inverse image functor, $f^*$ is coherent and since \Sh{\thry{G}} is a coherent topos, this means that for any $A$ in \cat{S} we have that $f^*(A)$ is coherent, and therefore, in particular, compact. Finally, %for any compact object $E$ in $\Sh{\thry{S}}$ there is a cover $e:A_1+\ldots +A_n\epi E$, where $A_1,\ldots, A_n$ are objects of \cat{S}, and $f^*$ takes this to a cover $e:f^*(A_1)+\ldots +f^*(A_n)\epi f^*(E)$, whence $f^*(E)$ %is compact. Hence $f^*$ takes compact objects to compact objects, so   $f:\thry{G}\to \thry{S}$ is in \alg{CohGpd}.
\end{proof}
%
%\begin{remark}
%Following Remark \ref{Remark: U Moerdijk descripiton} we can use the results of \cite{moerdijk:88} to give a sufficient condition for a groupoid morphism to be coherent.
%\end{remark}
%
%
The formal sheaves on a weakly coherent groupoid can be characterized directly:
\begin{lemma}\label{Lemma: DC, Theta(G)}
An equivariant sheaf $\cat{A}= \pair{A\rightarrow G_0, \alpha}$ on a weakly coherent groupoid \thry{G} is formal just in case:
\begin{enumerate}[(i)]
\item $\cat{A}$ is compact decidable;
\item
each fibre $A_x$ for $x\in G_0$ is an element of $\Sets_{\kappa}$;
\item
for each $a\in \mathds{S}$, the set $\bopen{\cat{A},a}=\cterm{x\in G_0}{a\in A_x}\subseteq G_0$ is open, and
the function $s_{\cat{A},a}:\cterm{x\in G_0}{a\in A_x}\rightarrow A$ defined by $s(x)=a$ is a continuous
section; and
\item
for any $a,b\in \mathds{S}$, the set
\[\bopen{\cat{A},a\mapsto b}=\cterm{g:x\rightarrow y}{a\in A_{x}\wedge \alpha(g,a)=b} \subseteq G_1\]
is open.
\end{enumerate}
\begin{proof}
Let a morphism
$f:\thry{G}\to\thry{S}$ in \alg{wcGpd} be given, inducing a geometric morphism $f:\Sh{\thry{G}}\to
\Sh{\thry{S}}$ such that the inverse image preserves compact objects. Then $f^*(\cat{U})$ is a compact decidable object with fibers in $\Sets_{\kappa}$; the
set $\bopen{f^*(\cat{U}),a}=f_0^{-1}(\bopen{a})\subseteq G_0$ is open; the continuous section $\bopen{a}\rightarrow U$ defined by $M\mapsto a$ pulls back along $f_0$ to yield the required section; and the set
$\bopen{f^*(\cat{U}),a\mapsto b}=f_1^{-1}(\bopen{a\mapsto b})\subseteq G_1$ is open. So $f^*(\cat{U})$ satisfies conditions (i)--(iv).

Conversely, suppose that $\cat{A}=\pair{A\rightarrow G_0, \alpha}$ satisfies conditions (i)--(iv). Define the function $f_0:G_0\rightarrow S_0$ by $x\mapsto A_x$, which is possible since
$A_x\in\Sets_{\kappa}$ by (ii). Then for a subbasic open set $\bopen{a}\subseteq S_0$, we have
\[ f_0^{-1}(\bopen{a})=\cterm{x\in G_0}{a\in A_x}=\bopen{\cat{A},a} \]
so $f_0$ is continuous by (iii). Next, define $f_1:G_1\rightarrow S_1$ by
\[g:x\rightarrow y\ \ \mapsto\ \
\alpha(g,-):A_x\rightarrow A_y.\]
Then for a subbasic open $\bopen{a\mapsto b}\subseteq S_1$, we have
\[ f_1^{-1}(\bopen{a\mapsto b})=\cterm{g\in G_1}{a\in A_{s(g)}\wedge \alpha(g,a)=b}=\bopen{\cat{A},a\mapsto b} \]
so $f_1$ is continuous by (iv). It remains to show that $f^*(\cat{U})=\cat{A}$. First, we must verify that what is a
pullback of sets:
\[\bfig
\square[A`U`G_0`S_0;```f_0]\place(100,400)[\pbangle]
\efig\]
is also a pullback of spaces. Let $a\in A$ with $V\subseteq A$ an open neighborhood. We must find an open box
around $a$ contained in $V$. Intersect $V$ with the image of the section $s_{\cat{A},a}(\bopen{\cat{A},a})$ to
obtain an open set $V'$ containing $a$ and homeomorphic to a subset $W\subseteq G_0$. Then we can write $V'$ as
the box $W\times_{S_0}\bopen{\syntob{x,y}{x=y},a}$ for the open set $\bopen{\syntob{x,y}{x=y},a}\subseteq U$. Conversely, let a basic open  $\bopen{\syntob{x,\vec{y}}{\phi},\vec{b}} \subseteq U$ be given, for $\phi$ a formula of $\theory_{\neq}$. We must show that it pulls back to an open subset of $A$. Let $a\in A_z$ be given and assume that $a$ (in the fiber over $f_0(z))$ is in
$\bopen{\syntob{x,\vec{y}}{\phi},\vec{b}}$.  Now, since \cat{A} is decidable, there is a canonical interpretation of
\syntob{x,\vec{y}}{\phi} in \Sh{\thry{G}} obtained by interpreting \cat{A} as the single sort, and using the
canonical coherent structure of \Sh{\thry{G}}. Thereby, we obtain an object
\[
\cat{B}:=\csem{x,\vec{y}}{\phi}^{\cat{A}}\embedd \cat{A}\times\ldots\times\cat{A}\to^{\pi_1}\cat{A}\]
in with an underlying open subset $B\subseteq A\times_{G_0}\ldots
\times_{G_0} A\to^{\pi_1}A$. One can verify that \cat{B} satisfies conditions (i)--(iv), se the proof of Lemma \ref{Lemma: Theta G is a coherent category} below. Let $W\subseteq B$ be the  image of the continuous section
$s_{\cat{B},a,\vec{b}}(\bopen{\cat{B},a,\vec{b}})$.  Then the pullback of $\bopen{\syntob{x,\vec{y}}{\phi},\vec{b}}$ along $f_0$ is the image of $W$ along the projection
$\pi_1:\cat{A}\times\ldots\times\cat{A}\to\cat{A}$, which is an open subset of $A$.
%
%We claim that it has $a$ as an
%element over $z$ and that it is contained in the pullback of $\bopen{\syntob{x,\vec{y}}{\phi},\vec{b}}$ along $f_0$.
%For an element $v\in G_0$ corresponds to a point
%
%\[ p_v:\Sets\to\Sets/G_0\epi\Sh{G_0}\epi\Sh{\thry{G}}\]
%
%such that value of $p_v^*$ at an object is the fiber over $v$. Now,  $p_v^*(\cat{A})=A_v$ is a set which we can
%consider as the underlying set of a $\theory_{\neq}$-model $\alg{M}_v = f_0(v)$ in the canonical way. And since
%$\cat{B}=\csem{x,\vec{y}}{\phi}^{\cat{A}}$ is computed using the fibre-wise set-induced canonical structure in
%\Sh{\thry{G}}, we have that $p_v(\cat{B})=\csem{x,\vec{y}}{\phi}^{\alg{M}_v}$. Therefore (index remembered for
%notational convenience),
%
%\begin{align*}
%
%     \pi_1(W)&= \cterm{\pair{v,c}\in A}{\fins{\vec{d}\in \Sets_{\kappa}}\pair{v,c,\vec{d}}
%           \in W\subseteq A\times_{G_0}\ldots\times_{G_0} A}\\
%           &= \cterm{\pair{v,a}\in A}{\pair{v,a,\vec{b}}\in B}\\
%           &= \cterm{\pair{v,a}\in A}{a,\vec{b}\in\csem{x,\vec{y}}{\phi}^{\alg{M}_v}}\\
%           &= \cterm{\pair{v,a}\in A}{a,\vec{b}\in \csem{x,\vec{y}}{\phi}^{f_0(v)}}
%
%\end{align*}
%
%so that $\pair{z,a}\in \pi(W)\subseteq G_0\times_{S_0} \bopen{\syntob{x,\vec{y}}{\phi},\vec{b}}$. It follows that
%$f^*(\cat{U})=\cat{A}$, and therefore, by Lemma \ref{Lemma: enough to check coh on U}, that  $f^*:\Sh{\thry{S}}\to\Sh{\thry{G}}$ preserves compact objects.
%
\end{proof}
\end{lemma}
The logically definable objects in the category of equivariant sheaves on the groupoid of models and isomorphisms of a
theory are readily seen to be a (guiding) example of objects satisfying conditions (i)--(iv) of Lemma \ref{Lemma:
DC, Theta(G)}, so we have:
\begin{lemma}\label{Lemma: DC, Definables are firm}
For any $\synt{C}{T}$ in \alg{dCoh_{\kappa}}, the functor $\cat{M}$ factors through $\Form(\thry{G}_{\theory})$,
\[\cat{M}:\synt{C}{T}\to\Form(\thry{G}_{\theory})\embedd \Sh{\thry{G}_{\theory}}\]
\end{lemma}
%
%Note that if \thry{G} is a coherent groupoid, then the terminal object in \Sh{\thry{G}} satisfies the conditions of Lemma \ref{Lemma: DC, Theta(G)}. Furthermore:
Next, we show that the formal sheaves on a weakly coherent groupoid form a decidable coherent category.
\begin{lemma}\label{Lemma: Theta G is a coherent category}
Let \thry{G} be an object of \alg{wcGpd}. Then $\Form(\thry{G})\embedd \Sh{\thry{G}}$ is a (positive) decidable
coherent category.
\begin{proof}
We verify that $\Form(\thry{G})$ is closed under the relevant operations using the characterization of Lemma \ref{Lemma: DC, Theta(G)}. By Lemma \ref{Lemma: Compact decidables in weakly compact topos}, it suffices to show that conditions (ii)--(iv) of Lemma \ref{Lemma: DC, Theta(G)} are closed under finite limits, images, and finite coproducts.

\textbf{Initial object.} Immediate.

\textbf{Terminal object.} The canonical terminal object, write \pair{X'\rightarrow X, \alpha}, is such that the
fiber over any $x\in G_0$ is $\{\star\}\in \Sets_{\kappa}$, whence the set \cterm{x\in G_0}{a\in X_x'} is $X$ if
$a=\star$ and empty otherwise. Similarly, the set $\cterm{g:x\rightarrow y}{a\in
X_x'\wedge\alpha(g,a)=b}\subseteq G_1$ is $G_1$ if $a=\star =b$ and empty otherwise.

\textbf{Finite products.} We do the binary product $\cat{A}\times\cat{B}$. The fiber over $x\in G_0$ is the
product $A_x\times B_x$, and so it is in $\Sets_{\kappa}$. Let a set $\bopen{\cat{A}\times\cat{B},c}$ be given. We
may assume that $c$ is a pair, $c=\pair{a,b}$, or $\bopen{\cat{A}\times\cat{B},c}$ is empty. Then,
\[
\bopen{\cat{A}\times\cat{B},\pair{a,b}}=\bopen{\cat{A},a}\cap \bopen{\cat{B},b}
\]
and the function $s_{\cat{A}\times\cat{B},\pair{a,b}}:\bopen{\cat{A}\times\cat{B},\pair{a,b}}\rightarrow
A\times_{G_0}B$ is continuous by the following commutative diagram:
\[\bfig
\morphism<0,-700>[\bopen{\cat{A},a}`A;s_{\cat{A},a}]
\morphism(1800,0)<0,-700>[\bopen{\cat{B},b}`B;s_{\cat{B},b}]
%
%\place(500,0)[\supseteq]
\morphism(0,0)/<-^{)}/<900,0>[\bopen{\cat{A},a}`\bopen{\cat{A}\times\cat{B},\pair{a,b}};\supseteq]
\morphism(900,0)<0,-700>[\bopen{\cat{A}\times\cat{B},\pair{a,b}}`A\times_{G_0} B;s_{\cat{A}\times \cat{B},
\pair{a,b}}]
%
%\place(1300,0)[\subseteq]
\morphism(900,0)/^{ (}->/<900,0>[\bopen{\cat{A}\times\cat{B},\pair{a,b}}`\bopen{\cat{B},b};\subseteq]
\morphism(0,-700)/<-/<900,0>[A`A\times_{G_0} B;\pi_1]
\morphism(900,-700)<900,0>[A\times_{G_0} B`B;\pi_2]
\efig\]
Similarly, the set $\bopen{\cat{A}\times B,c\mapsto d}$ is either empty or of the form \[\bopen{\cat{A}\times
B,\pair{a,b}\mapsto \pair{a',b'}}\] in which case
\[\bopen{\cat{A}\times
B,\pair{a,b}\mapsto \pair{a',b'}}= \bopen{\cat{A},a\mapsto a'}\cap \bopen{\cat{B},b\mapsto b'}.\]

\textbf{Equalizers and Images.} Let \cat{A} be a subobject of $\cat{B}=\pair{\pi_1:B\rightarrow G_0,\beta}$,
with $A\subseteq B$, and \cat{B} satisfying the properties (ii)--(iv) of Lemma \ref{Lemma: DC, Theta(G)}. Then
given a set $\bopen{\cat{A},a}$,
\[ \bopen{\cat{A},a}= \pi_1(A\cap s_{\cat{B},a}(\bopen{\cat{B},a})) \]
and we obtain $s_{\cat{A},a}$ as the restriction
\[\bfig
\dtriangle/<-`>`^{ (}->/<800,500>[B`\bopen{\cat{B},a}`G_0; s_{\cat{B},a}`\pi_1`]
\dtriangle(-800,0)/<-``^{ (}->/<800,500>[A`\bopen{\cat{A},a}` \bopen{\cat{B},a}; s_{\cat{A},a}``]
\morphism(0,500)/^{ (}->/<800,0>[A`B;] \place(-500,75)[\pbanglef]
%\place(-400,0)[\subseteq]
%
%\place(400,0)[\subseteq]
%
%\place(400,500)[\subseteq]
%
\efig\]
Similarly, given a set $\bopen{\cat{A},a\mapsto b}\subseteq G_1$,
\[ \bopen{\cat{A},a\mapsto b}= \bopen{\cat{B},a\mapsto b}\cap s^{-1}(\bopen{\cat{A},a}) \]
where $s$ is the source map $s:G_1\rightarrow G_0$. We conclude that \funksjon{\Form}{\thry{G}} is closed under
both equalizers and images.

\textbf{Binary coproducts.}  Write binary coproducts in $\Sets_{\kappa}$ as $X+Y=
\cterm{\pair{0,x},\pair{1,y}}{x\in X\wedge y\in Y}$. Then if $\bopen{\cat{A}+\cat{B},c}$ is non-empty, $c$ is a pair
$c=\pair{0,a}$ or $c=\pair{1,b}$. If the former, then $\bopen{\cat{A}+\cat{B}, \pair{0,a}}= \bopen{\cat{A},a}$, and the
section is given by composition:
\[\bfig
\square/>`<-`>`^{ (}->/<1000,500>[A`A+B`\bopen{\cat{A}+\cat{B}, \pair{0,a}}= \bopen{\cat{A},a}`G_0; p_1`s_{\cat{A},a}``]
\efig\]
The latter case is similar, and so is verifying that the set $\bopen{\cat{A}+\cat{B}, c\mapsto d}$ is open.
\end{proof}
\end{lemma}
\begin{lemma}
Let \thry{G} be a weakly coherent groupoid. Then $\Form(\thry{G})\embedd \Sh{\thry{G}}$ has enough
$\kappa$-small models.
\begin{proof}
This follows from the fact that the coherent inclusion
\[\Form(\thry{G})\embedd \Sh{\thry{G}}\]
reflects covers, since every formal sheaf is compact, and any point, given by an element $x\in
G_0$,
\[\Sets \to \Sets/G_0 \epi \Sh{G_0} \epi \Sh{\thry{G}} \epi \Sh{\Form(\thry{G})} \]
yields a coherent functor $\Form(\thry{G})\to \Sets_{\kappa}\embedd\Sets$, since the value of the point at an
equivariant sheaf is the fiber over $x$, and formal sheaves have fibers in $\Sets_{\kappa}$.
\end{proof}
\end{lemma}
\begin{lemma}
If $f:\thry{G}\to\thry{H}$ is a morphism of \alg{wcGpd}, then the induced coherent inverse image functor
$f^*:\Sh{\thry{H}}\to\Sh{\thry{G}}$ restricts to a coherent functor $\Form(f)=F:\Form(\thry{H})\to
\Form(\thry{G})$,
\[
\bfig
\square/>`^{ (}->`^{ (}->`>/<750,400>[\Form(\thry{H})`\Form(\thry{G})` \Sh{\thry{H}}` \Sh{\thry{G}}; F```f^*]
\efig\]
\begin{proof}
If \cat{A} is an object of $\Form(\thry{H})$ classified
by $h:\thry{H}\to\thry{S}$, then $f^*(\cat{A})=F(\cat{A})$ is classified by $h\circ f:\thry{G}\to\thry{S}$ in
\alg{wcGpd}.
%For an object $\cat{A}= \pair{A\rightarrow H_0, \alpha}$ in $\Form(\thry{H})$, the object
%$f^*(\cat{A})=\pair{A'\rightarrow G_0,\Form'}$ is stably compact decidable by assumption. For $x\in G_0$, the
%fiber $A'_x=A_{f_0(x)}$ is in $\Sets_{\kappa}$,
%
%\[\bfig
%
%\square[A'`A`G_0`H_0;```f_0]\place(100,400)[\pbangle]
%
%\efig\]
%
%the set $U_{f^*(\cat{A}),a}=f_0^{-1}(U_{\cat{A},a})$ is open, and the section
%$s_{f^*(\cat{A}),a}=f_0^*(s_{\cat{A},a})$ is continuous. Finally, the subset $U_{f^*(\cat{A}),a\mapsto b}=
%f_1^{-1}(U_{\cat{A},a\mapsto b})\subseteq G_1$ is also open.
\end{proof}
\end{lemma}
This completes  the construction of the `syntactical' functor:
\begin{definition}
The functor \[\Form:\alg{wcGpd}\to\alg{dCoh}_{\kappa}^{\mng{op}}\] is defined by sending a weakly coherent groupoid \thry{G} to
the decidable coherent category \[\Form(\thry{G})\embedd \Sh{\thry{G}}\] of formal sheaves, and a morphism
$f:\thry{G}\to\thry{H}$ to the restricted inverse image functor $f^*:\Form(\thry{H})\to\Form(\thry{G})$.
\end{definition}
\subsection{The Syntax-Semantics Adjunction}%%%%%%%%%%%%%%%%%%%%%%%%%%%%%%%%%%%%%%%%%%%%%%%%%%%%%%%%%%%%%%%%%%%%%%%%%%%%%%%%%%%%%%%%%%%%%%%%%%%%%%%%%%%%%%%%%%%%%%%%%%%%%%%%%%%%%%%%%%%%%%%%%%%%%%%%%%%%%%%%
\label{Subsection: DC, The Syntax-Semantics Adjunction}
We now show that the syntactical functor is left adjoint to the semantical functor:
\[
\bfig
\morphism|a|/{@{>}@/^7pt/}/<800,0>[\alg{dCoh_{\kappa}}^{\mng{op}}` \alg{wcGpd}; \Mod]
\morphism|b|/{@{<-}@/^-7pt/}/<800,0>[\alg{dCoh_{\kappa}}^{\mng{op}}` \alg{wcGpd}; \Form]
\place(400,0)[\top]
\efig
\]
First, we identify a counit candidate. Given $\cat{D}$ in $\alg{dCoh_{\kappa}}$, we have the `evaluation'  functor
\[\cat{Y}_{\cat{D}}:\cat{D}\to\Sh{\thry{G}_{\cat{D}}}\]
which sends an object $D$ to the `definable' equivariant sheaf which is such that the fiber of $\cat{Y}(D)$  over $F\in X_{\cat{D}}$ is the set $F(D)$, or more informatively, such that the diagram,
\[\bfig
\square<500,300>[\cat{D}`\Sh{\thry{G}_{\cat{D}}}`\cat{C}_{\theory_{\cat{D}}}`\Sh{\thry{G}_{\theory_{\cat{D}}}};\cat{Y}_{\cat{D}}`\zeta_{\cat{D}}`\cong`\cat{M}]
\efig\]
commutes, using the map $\zeta_{\cat{D}}$ and isomorphism $\thry{G}_{\cat{D}}\cong\thry{G}_{\theory_{\cat{D}}}$ from Section \ref{Subsection: Representation theorem for decidable coherent categories}.
%
%\[
%D\mapsto \cterm{\pair{M,a}}{M\in X_{\cat{D}}\wedge a\in M(D)}\to<250>^{\pi_1} X_{\cat{D}}
%\]
%
%over $X_{\cat{D}}$, and with action corresponding to the action of the  natural isomorphisms of $G_{\cat{D}}$ on
%functors of $X_{\cat{D}}$, as the composite of coherent functors
%
%\[
%
%\cat{N}_{\cat{D}}:\cat{D}\to^{\iota_{\cat{D}}} T\cat{D} \embedd^{\cat{M}^{\dag}}
%\Sh{\thry{G}_{\theory_{\cat{D}}}} \cong \Sh{\thry{G}_{\cat{D}}}.
%
%\]
%
$\cat{Y}_{\cat{D}}$ factors through $\Form(\thry{G}_{\cat{D}})$, by Lemma \ref{Lemma: DC, Definables
are firm}, to yield a coherent functor
\[\epsilon_{\cat{D}}:\cat{D}\to \Form(\thry{G}_{\cat{D}})=\Form\circ \Mod(\cat{D})\]
And if $F:\cat{A}\to\cat{D}$ is an arrow of \alg{dCoh_{\kappa}}, the square
\[\bfig
\square<750,500>[\cat{A}` \Form\circ \Mod(\cat{A})` \cat{D}` \Form\circ
\Mod(\cat{D});\epsilon_{\cat{A}}` F` \Form\circ \Mod(F)` \epsilon_{\cat{D}}]
\efig\]
commutes.

Next, we consider the unit. Let \thry{H} be a groupoid in \alg{wcGpd}. We construct a morphism
\[\eta_{\thry{H}}:\thry{H}\to \thry{G}_{\Form(\thry{H})}= \Mod(\Form(\thry{G})).\]
First, as previously noticed, each $x\in H_0$ induces a coherent functor $\alg{M}_x:\Form(\thry{H})\to
\Sets_{\kappa}$. This defines a function $\eta_0:H_0\rightarrow X_{\Form(\thry{H})}$. Similarly, any
$a:x\rightarrow y$ in $H_1$ induces an invertible natural transformation $\alg{f}_a:\alg{M}_x\rightarrow \alg{M}_y$. This defines
a function $\eta_1:H_1\rightarrow G_{\Form(\thry{H})}$, such that \pair{\eta_1,\eta_0} is a morphism of
discrete groupoids.  We argue that $\eta_0$ and $\eta_1$ are continuous. Let a subbasic open
$U=(\pair{g_1:\cat{A}\rightarrow \cat{B}_1,\ldots,g_n:\cat{A}\rightarrow
\cat{B}_n},\pair{a_1,\ldots,a_n})\subseteq X_{\Form(\thry{H})}$ be given, with $g_i:\cat{A}=\pair{A\rightarrow
H_0, \alpha}\to \cat{B}_i=\pair{B_i\rightarrow H_0, \beta_i}$ an arrow of $\Form(\thry{H})$ and $a_i\in
\mathds{S}$, for $1\leq i\leq n$. Form the canonical product $\cat{B}_1\times \ldots\times \cat{B}_n$ in
\Sh{\thry{H}}, so as to get an arrow $g=\pair{g_1,\ldots,g_n}:\cat{A}\to \cat{B}_1\times \ldots \times\cat{B}_n$
in $\Form(\thry{H})$. Denote by \cat{C} the canonical image of $g$ in \Sh{\thry{H}} (and thus in
$\Form(\thry{H})$), such that the underlying set $C$ (over $H_0$) of \cat{C} is a subset of
$B_1\times_{H_0}\ldots\times_{H_0}B_n$. Then
\begin{align*}
\eta_0^{-1}(U) &= \cterm{x\in H_0}{\fins{y\in \alg{M}_x(\cat{A})}\alg{M}_x(g_i)(y)=a_i\ \textnormal{for}\ 1\leq i\leq n}\\
               &= \cterm{x\in H_0}{\fins{y\in A_x}g_i(y)=a_i\ \textnormal{for}\ 1\leq i\leq n}\\
               &= \cterm{x\in H_0}{\pair{a_1,\ldots,a_n}\in \alg{M}_x(\cat{C})}\\
               &= \cterm{x\in H_0}{\pair{a_1,\ldots,a_n}\in C_x}
\end{align*}
which is an open subset of $H_0$ by Lemma \ref{Lemma: DC, Theta(G)} since \cat{C} is in $\Form(\thry{H})$. Thus $\eta_0$ is continuous. Next,
consider a subbasic open of $G_{\Form(\thry{H})}$ of the form $U=(\cat{A},a\mapsto b)\subseteq
G_{\Form(\thry{H})}$, for $\cat{A}=\pair{A\rightarrow H_0, \alpha}$ in $\Form(\thry{H})$. Then
\begin{align*}
\eta_1^{-1}(U) &= \cterm{g:x\rightarrow y}{a\in \alg{M}_x(\cat{A})\wedge (\alg{f}_g)_{\cat{A}}(a)=b}\subseteq H_1\\
               &=  \cterm{g:x\rightarrow y}{a\in A_x\wedge \alpha(g,a)=b}\subseteq H_1
\end{align*}
which is an open subset of $H_1$, since \cat{A} is in $\Form(\thry{H})$. Thus $\eta_1$ is also continuous, so
that \pair{\eta_1,\eta_0} is a morphism of continuous groupoids.
\begin{lemma}\label{Lemma: DC, unit well-behaved triangle}
The triangle
\begin{equation}\label{Equation: DC, unit well-behaved triangle}\bfig
\dtriangle/<-_{)}`<-`->/<900,400>[\Sh{\thry{H}}` \Form(\thry{H})` \Sh{\thry{G}_{\Form(\thry{H})}}; `
\eta^*_{\Form(\thry{H})}` \cat{Y}_{\Form(\thry{H})}]
\efig\end{equation}
commutes.
\begin{proof}
Let $\cat{A}=\pair{A\rightarrow H_0,\alpha}$ in $\Form(\thry{H})$ be given, and write $E_{\cat{A}}\rightarrow
X_{\Form(\thry{H})}$ for the underlying sheaf of $\cat{Y}_{\Form(\thry{H})}(\cat{A})$. Write
$a:\thry{H}\rightarrow\thry{S}$ and $a':\thry{G}_{\Form(\thry{H})}\rightarrow\thry{S}$, respectively, for the \alg{wcGpd}
morphisms classifying these objects. Then the triangle
\[\bfig
\Vtriangle[\thry{H}`\thry{G}_{\Form(\thry{H})}`\thry{S};\eta_{\Form(\thry{H})}`a`a']
\efig\]
in \alg{Gpd} can be seen to commute. Briefly, for $x\in H_0$, we have
$a(x)=A_x=\alg{M}_x(\cat{A})=(E_{\cat{A}})_{\alg{M}_x}= (E_{\cat{A}})_{\eta_0(x)}=a'(\eta_0(x))$ and similarly for elements
of $H_1$.
\end{proof}
\end{lemma}
It follows from Lemma \ref{Lemma: DC, unit well-behaved triangle} that the inverse image functor $\eta_{\Form(\thry{H})}^*$ preserves
compact objects, and so $\eta_{\Form(\thry{H})}:\thry{H}\to G_{\Form(\thry{H})}$ is indeed a
morphism of \alg{wcGpd}. It remains to verify that it is the component of a natural transformation. Given a
morphism $f:\thry{G}\to\thry{H}$ of \alg{wcGpd}, we must verify that the square
\[\bfig
\square<1000,500>[\thry{G}`\Mod\circ \Form(\thry{G})` \thry{H}` \Mod\circ \Form(\thry{H});
\eta_{\Form(\thry{G})}` f`  \Mod\circ \Form(f)` \eta_{\Form(\thry{H})}]
\efig\]
commutes. Let $x\in G_0$ be given. We chase it around the square. Applying $\eta_{\Form(\thry{G})}$, we obtain
the functor $\alg{M}_x:\Form(\thry{G})\to \Sets$ which sends an object $\cat{A}=\pair{A\rightarrow G_0,\alpha}$ to
$A_x$. Composing with $\Form(f):\Form(\thry{H})\to\Form(\thry{G})$, we obtain the functor
$\Form(\thry{H})\to \Sets$ which sends an object $\pair{B\rightarrow H_0,\beta}$ to the fiber over $x$
of the pullback
\[\bfig
\square[f_0^*(B)`B`G_0`H_0;```f_0]\place(100,400)[\pbangle] \efig\]
which is the same as the fiber $B_{f_0(x)}$. And this is the same functor that results from sending $x$ to
$f_0(x)$ and applying $\eta_{\Form(\thry{H})}$. For $a:x\rightarrow y$ in $G_1$, a similar check establishes
that $\eta_1\circ f_1(a):\alg{M}_{f_0(x)}\rightarrow \alg{M}_{f_0(y)}$ equals $\eta_1(a)\circ \Form(f):M_{x}\circ
\Form(f)\rightarrow \alg{M}_{y}\circ \Form(f)$. It remains to verify the triangle identities.
\begin{lemma}
%\label{Lemma: T_Multisort triangle identities}
The triangle identities hold:
\[
\bfig
\btriangle/<-`<-`<-/<1200,800>[\Form(\thry{H})` \Form \circ \Mod \circ \Form(\thry{H})` \Form(\thry{H});
\Form(\eta_{\thry{H}})` 1_{\Form(\thry{H})}` \epsilon_{\Form(\thry{H})}]
\place(300,300)[=]
\btriangle(0,-1200)<1200,800>[\Mod(\cat{D})` \Mod\circ \Form \circ \Mod(\cat{D})` \Mod(\cat{D});
\eta_{\Mod(\cat{D})}` 1_{\Mod(\cat{D})}` \Mod(\epsilon_{\cat{D}})]
\place(300,-900)[=] \efig
\]
\begin{proof}
We begin with the first triangle, which we write:
\[
\bfig
\btriangle/<-`<-`<-/<1000,800>[\Form(\thry{H})` \Form(\thry{G}_{\Form(\thry{H})})` \Form(\thry{H});
\Form(\eta_{\thry{H}})` 1_{\Form(\thry{H})}` \epsilon_{\Form(\thry{H})}]
\efig\]
This triangle commutes by the definition of $\epsilon_{\Form(\thry{H})}$ and Lemma \ref{Lemma: DC, unit
well-behaved triangle}, as can be seen by the following diagram:
\[\bfig
\square/>`^{ (}->`^{ (}->`<-/<1000,500>[\Form(\thry{H})` \Form(\thry{G}_{\Form(\thry{H})})` \Sh{\thry{H}}`
\Sh{\thry{G}_{\Form(\thry{H})}}; \epsilon_{\Form(\thry{H})}` ``\eta^*_{\thry{H}}]
\morphism(0,500)<1000,-500>[\Form(\thry{H})`\Sh{\thry{G}_{\Form(\thry{H})}};\cat{Y}_{\Form(\thry{H})} ]
\place(230,120)[=]\place(610,360)[=]
\efig\]
We pass to the second triangle, which can be written as:
\[\bfig
\btriangle(0,0)<1000,800>[\thry{G}_{\cat{D}}` \thry{G}_{\Form(\thry{G}_{\cat{D}})}` \thry{G}_{\cat{D}};
\eta_{\thry{G}_{\cat{D}}}` 1_{\thry{G}_{\cat{D}}}` \Mod(\epsilon_{\cat{D}})]
\efig
\]
Let $\alg{N}:\cat{D}\to\Sets$ in $X_{\cat{D}}$ be given. As an element in $X_{\cat{D}}$, it determines a coherent
functor $\alg{M}_{\alg{N}}:\Form(\thry{G}_{\cat{D}})\to \Sets$, the value of which at $\cat{A}=\pair{A\rightarrow
X_{\cat{D}}, \alpha}$ is the fiber $A_{\alg{N}}$. Applying $\Mod(\epsilon_{\cat{D}})$ is composing with the functor
$\epsilon_{\cat{D}}:\cat{D}\to \Form(\thry{G}_{\cat{D}})$, to yield the functor $\alg{M}_{\alg{N}}\circ
\epsilon_{\cat{D}}:\cat{D}\to\Sets$, the value of which at an object $B$ in \cat{D} is the fiber over $\alg{N}$ of
$\cat{Y}_{\cat{D}}(B)$, which of course is just $\alg{N}(B)$. For an invertible natural transformation $\alg{f}:\alg{M}\rightarrow \alg{N}$ in
$G_{\cat{D}}$, the chase is entirely similar, and we conclude the the triangle commutes.
\end{proof}
\end{lemma}
\begin{theorem}\label{Proposition: DC, The adjunction}
The contravariant functors \Mod\ and $\Form$ are adjoint,
\[
\bfig
\morphism|a|/{@{>}@/^7pt/}/<800,0>[\alg{dCoh_{\kappa}}^{\mng{op}}` \alg{wcGpd}; \Mod]
\morphism|b|/{@{<-}@/^-7pt/}/<800,0>[\alg{dCoh_{\kappa}}^{\mng{op}}` \alg{wcGpd}; \Form]
\place(400,0)[\top]
\efig
\]
where \Mod\ sends a decidable coherent category \cat{D} to the semantic groupoid
\homset{\alg{dCoh}}{\cat{D}}{\Sets_{\kappa}} equipped with the coherent topology, and $\Form$ sends a weakly
coherent groupoid \thry{G} to the full subcategory $\Form(\thry{G})\embedd \Sh{\thry{G}}$ of formal sheaves, i.e.\  those classified by the morphisms in \homset{\alg{wcGpd}}{\thry{G}}{\thry{S}}.
\end{theorem}
Notice that if \cat{D} is an object of \alg{dCoh_{\kappa}}, then the counit component
$\epsilon_{\cat{D}}:\cat{D}\to \Form\circ\Mod(\cat{D})$ is a Morita equivalence of categories, in the sense
that  it induces an equivalence $\sh{D}\simeq \Sh{\Form\circ\Mod(\cat{D})}$. In the case where \cat{D} is a
pretopos, the counit is, moreover,  also an equivalence of categories, since any decidable compact object in
\sh{D} is coherent and therefore isomorphic to a representable in that case. Now, say that a weakly coherent groupoid \thry{G} is \emph{semantic} if there exists a family $(f_i:\thry{G}\to\thry{S})$ of morphisms of topological groupoids such that for all $i\in I$ and some (equivalently, all) $a\in \mathds{S}$, the morphism $f_i$ is an $N$-fibration with respect to the open subgropoid $U=\bopen{a}$,
\[N=\left(\begin{array}{c}
 a  \\
  a \mapsto a   \\
 a
\end{array}\right)\]
and $((f_i)_0^{-1}(U),(f_i)_1^{-1}(N))_{i\in I}$ is a generating family of s-compact subgroupoids. Since by Lemma \ref{Lemma: Definables are GUNS} $\pair{\thry{S},U,N}\cong \cat{U}$, it follows from Lemma \ref{Lemma: GUN subobjects} that $\pair{\thry{G},(f_i)_0^{-1}(U),(f_i)_1^{-1}(N)}\cong f_i^*(\cat{U})$ and so that $f_i$ is compact and $f_i^*(\cat{U})\in\Form(\thry{G})$. Let \alg{SemGpd} be the full subcategory of \alg{wcGpd} consisting of semantic groupoids.
%
%
%recall the definition of an (open) decidable coherent groupoid in \ref{Proposition: Coherent decidable groupoid}. Say that a decidable coherent groupoid \thry{G} is $\kappa$-\emph{bounded} if there is a generating set of open %subgroupoids $(U_i,N_i)_{i\in I}$ satisfying the conditions of Proposition \ref{Proposition: Coherent decidable groupoid} such that for all $i\in I$ and all $x\in G_0$, there are at most $|\mathds{S}|$-many arrows with source in %$U_i$ and target $x$ up to $\sim_{N}$-equivalence. That is, the fibre over $x$ in $s^{-1}(U_i)/_{\sim_N}$ has size   $\leq|\mathds{S}|$. Let \alg{SemGpd} be the full subcategory of \alg{wcGpd} consisting of $\kappa$-bounded %decidable coherent groupoids.
%
\begin{corollary}
The adjunction of Theorem \ref{Proposition: DC, The adjunction} restricts to an adjunction
\[
\bfig
\morphism|a|/{@{>}@/^7pt/}/<800,0>[\alg{dCoh_{\kappa}}^{\mng{op}}` \alg{SemGpd}; \Mod]
\morphism|b|/{@{<-}@/^-7pt/}/<800,0>[\alg{dCoh_{\kappa}}^{\mng{op}}` \alg{SemGpd}; \Form]
\place(400,0)[\top]
\efig
\]
with the property that the unit and counit components are Morita equivalences of categories and topological
groupoids respectively.
\begin{proof}
\Mod\ factors through \alg{SemGpd} by Lemma \ref{Lemma: Definables are GUNS} (cf.\ \ref{Lemma: Mod factors through weakly coherent groupoids}) with the required family of morphisms, for a decidable coherent category \cat{D}, being the evaluation morphisms $f_D:\thry{G}_{\cat{D}}\to\thry{S}$ sending a functor to its value at $D\in \cat{D}$ and an isomorphism to its $D$-component. If \thry{G} is a semantic groupoid, then since the objects $f_i^*(\cat{U})$ are generating and in $\Form(\thry{G})$, the latter forms a site for \Sh{\thry{G}} when equipped with the coherent coverage, whence the unit of the adjunction is a Morita equivalence.
\end{proof}
\end{corollary}
%
%
%
%
%
%
%Furthermore, for any \cat{D} in
%\alg{dCoh_{\kappa}}, we have that the unit component $\eta_{\thry{G}_{\cat{D}}}:\thry{G}_{\cat{D}}\to
%\thry{G}_{\Form(\thry{G}_{\cat{D}})}$ is a Morita equivalence of topological groupoids, in the sense that it induces an
%equivalence $\Sh{\thry{G}_{\cat{D}}}\simeq \Sh{\thry{G}_{\Form(\thry{G}_{\cat{D}})}}$. We refer to the full image of
%\Mod\ in \alg{Gpd} as \alg{SemGpd}, the category of \emph{semantic groupoids}.
%
%\begin{corollary}
%The adjunction of Theorem \ref{Proposition: DC, The adjunction} restricts to an adjunction
%
%\[
%\bfig
%
%\morphism|a|/{@{>}@/^7pt/}/<800,0>[\alg{dCoh_{\kappa}}^{\mng{op}}` \alg{SemGpd}; \Mod]
%
%\morphism|b|/{@{<-}@/^-7pt/}/<800,0>[\alg{dCoh_{\kappa}}^{\mng{op}}` \alg{SemGpd}; \Form]
%
%\place(400,0)[\top]
%
%\efig
%\]
%
%with the property that the unit and counit components are Morita equivalences of categories and topological
%groupoids respectively.
%Moreover, counit components at pretoposes are actually equivalences of categories, so that the adjunction further restricts to the subcategory of (small, decidable) pretoposes
%
%\[\alg{dPTop_{\kappa}}\embedd\alg{dCoh_{\kappa}}\]
%
%and its image along \Mod\ to form a (bi-)equivalence,
%
%\[{\alg{dPTop_{\kappa}}^{\mathrm{op}}} \simeq \alg{SemGpd}.\]
%
%\end{corollary}
%
%
\subsection{Stone Duality for Classical First-Order Logic}%%%%%%%%%%%%%%%%%%%%%%%%%%%%%%%%%%%%%%%%%%%%%%%%%%%%%%%%%%%%%%%%%%%%%%%%%%%%%%%%%%%%%%%%%%%%%%%%%%%%%%%%%%%%%%%%%%%%%%%%%%%%%%%%%%%%%%%%%%%%%%%%%%%%%%%%%%%%%%%%

Returning to the classical first-order logical case, we can restrict the adjunction further to the full subcategory $\alg{BCoh_{\kappa}}\embedd \alg{dCoh_{\kappa}}$ of Boolean coherent categories. Unlike in the decidable coherent case,  the pretopos completion of a Boolean coherent category is again Boolean, so that \alg{BCoh_{\kappa}} is closed under pretopos completion. Since, as we mentioned in Section \ref{Subsection: Theories and Models}, completing a first-order theory so that its syntactic category is a pretopos involves only a conservative extension of the theory and does not change the category of models, it is natural to represent the classical first-order theories by the subcategory of Boolean pretoposes (see e.g.\ \cite{makkaireyes}, \cite{makkai:87b}).  We shall refer to the groupoids in the image of the semantic functor \Mod\ restricted to the full subcategory of Boolean pretoposes $\alg{BPTop_{\kappa}}\embedd\alg{dCoh_{\kappa}}$,  as \emph{Stone groupoids}.  Thus $\alg{StoneGpd}\embedd\alg{SemGpd}$ is the full subcategory of topological groupoids of models of theories in classical, first-order logic (the morphisms are still those continuous homomorphisms that preserve compact sheaves).
\begin{corollary}\label{Cor:adj for bpt}
The adjunction of Theorem \ref{Proposition: DC, The adjunction} restricts to an adjunction
\[
\bfig
\morphism|a|/{@{>}@/^7pt/}/<800,0>[\alg{BPTop_{\kappa}}^{\mng{op}}` \alg{StoneGpd}; \Mod]
\morphism|b|/{@{<-}@/^-7pt/}/<800,0>[\alg{BPTop_{\kappa}}^{\mng{op}}` \alg{StoneGpd}; \Form]
\place(400,0)[\top]
\efig
\]
with the property that the unit and counit components are Morita equivalences of topological groupoids and equivalences of pretoposes, respectively.
\end{corollary}

Moreover, given the obvious notion of `continuous natural transformation' of topological groupoid homomorphisms, the unit components of the foregoing adjunction can also be shown to be equivalences. Thus we have the following:

\begin{theorem}\label{thm:duality for bpt}
The adjunction of Corollary \ref{Cor:adj for bpt} is a (bi-)equivalence,
\begin{equation}\label{eq: Stone duality for bccs}
{\alg{BPTop_{\kappa}}^{\mathrm{op}}} \simeq \alg{StoneGpd}
\end{equation}
establishing a duality between the category of ($\kappa$-small) Boolean pretoposes and Stone topological groupoids.
\end{theorem}

Finally, a remark on the posetal case and classical Stone duality for Boolean algebras. By a \emph{coherent space} we mean a compact topological space such that the compact open sets are closed under intersection and form a basis for the topology. A \emph{coherent function} between coherent spaces is a continuous function such that the inverse image of a compact open is again compact. Stone duality can be obtained as a restriction of a contravariant adjunction between the category \alg{dLat} of distributive lattices and homomorphisms and the category \alg{CohSpace} of coherent spaces and coherent functions
\begin{equation}\label{eq: Stone duality for distlat}
\bfig
\morphism|a|/{@{>}@/^7pt/}/<800,0>[\alg{dLat}^{\mng{op}}` \alg{CohSpace}; ]
\morphism|b|/{@{<-}@/^-7pt/}/<800,0>[\alg{dLat}^{\mng{op}}` \alg{CohSpace}; ]
\place(400,0)[\top]
\efig
\end{equation}
where, as in Stone duality, the right adjoint is the `Spec' functor obtained by taking prime filters (or homming into the lattice 2), and the left adjoint is obtained by taking the distributive lattice of compact opens (or homming into the Sierpi\'{n}ski space, i.e.\ the set 2 with one open point). This adjunction restricts to a contravariant equivalence between distributive lattices and sober coherent spaces, and further to the full subcategory of Boolean algebras, $\alg{BA}\hookrightarrow\alg{dLat}$, and the full subcategory of Stone spaces and continuous functions, $\alg{Stone}\hookrightarrow \alg{CohSpace}$, so as to give the contravariant equivalence of classical Stone duality:
\begin{equation}\label{adj:stonedual}
\bfig
\morphism|a|/{@{>}@/^7pt/}/<800,0>[\alg{BA}^{\mng{op}}` \alg{Stone}; ]
\morphism|b|/{@{<-}@/^-7pt/}/<800,0>[\alg{BA}^{\mng{op}}` \alg{Stone}; ]
\place(400,0)[\simeq]
\efig
\end{equation}

The adjunction \eqref{eq: Stone duality for distlat} can be obtained from the adjunction of Theorem \ref{Proposition: DC, The adjunction} as follows. A poset is a distributive lattice if and only if it is a coherent category (necessarily decidable), and as we remarked after Definition \ref{Definition: DC, DCkappa}, such a poset always has enough $\kappa$-small models, so that $$\alg{dLat}\embedd\alg{dCoh_{\kappa}}$$ is the subcategory of posetal objects.  On the other side, any space can be considered as a trivial topological groupoid, with only identity arrows, and it is straightforward to verify that this yields a full embedding $$\alg{CohSpace}\embedd\alg{wcGpd}.$$ Since a coherent functor from a distributive lattice \cat{L} into \Sets\ sends the top object in \cat{L} to the terminal object 1 in \Sets, and everything else to a subobject of 1, restricting the semantic functor \Mod\ to \alg{dLat} gives us the right adjoint of (\ref{eq: Stone duality for distlat}). In the other direction, applying the syntactic functor \Form\ to the subcategory $\alg{CohSpace}\embedd\alg{wcGpd}$ does not immediately give us a functor into \alg{dLat}, simply because the formal sheaves do not form a poset (for instance, by Lemma \ref{Lemma: DC, Theta(G)}, the formal sheaves on a coherent groupoid include all finite coproducts of 1). However, if we compose  with the functor $\mathrm{Sub}(1):\alg{dCoh_{\kappa}}\to \alg{dLat}$ which sends a coherent category $\cat{C}$ to its distributive lattice $\mathrm{Sub}_{\cat{C}}(1)$ of subobjects of 1, then it is straightforward to verify that we have a restricted adjunction
\[
\bfig
\morphism|a|/{@{>}@/^7pt/}/<800,0>[\alg{dLat}^{\mng{op}}` \alg{CohSpace}; \Mod]
\morphism|b|/{@{<-}@/^-7pt/}/<800,0>[\alg{dLat}^{\mng{op}}` \alg{CohSpace}; \hspace{1ex}\mathrm{Form}_1]
\place(400,0)[\top]
\efig
\]
where $\mathrm{Form}_1(\cat{C}) = \mathrm{Sub}_{\mathrm{Form}(\cat{C})}(1)$. Moreover, this is easily seen to be precisely the adjunction (\ref{eq: Stone duality for distlat}), of which classical Stone duality for Boolean algebras is a special case.  Indeed, again up to the reflection into $\mathrm{Sub}(1)$, the duality \eqref{adj:stonedual} is precisely the poset case of the duality \eqref{eq: Stone duality for bccs} between ($\kappa$-small) Boolean pretoposes and Stone topological groupoids.

\subsection{Future work}%%%%%%%%%%%%%%%%%%%%%%%%%%%%%%%%%%%%%%%%%%%%%%%%%%%%%%%%%%%%%%%%%%%%%%%%%%%%%%%%%%%%%%%%%%%%%%%%%%%%%%%%%%%%%%%%%%%%%%%%%%%%%%%%%%%%%%%%%%%%%%%%%%%%%%%%%%%%%%%%%%%%%%%%%%%%%%%%%%%%%%%%%%%%%%%%%%%%%%%%
The underlying idea of this paper is to show how the representation theorem of Butz and Moerdijk, suitably adjusted and translated into logical terms, can be used to extend Stone duality to predicate logic. The result is, albeit in a very loose sense, a syntax-semantics `duality' between coherent decidable theories (subsuming classical first-order theories) and topological groupoids of models. In the context of the dualities mentioned in the introduction, this is another step towards a formulation and understanding of the dual nature of the syntax and semantics of theories.
The following is a selection of some the open problems, loose ends, and directions for ongoing and future work in the further pursuit of this goal.
\begin{enumerate}
\item There are several questions and open problems regarding topological groupoids and equivariant sheaf toposes with relevance for extending, sharpening, or finding applications of the duality theory, such as providing a characterization of coherent groupoids and coherent morphisms between them (the techniques of Lemmas \ref{Lemma: Generating set of GUNS} and \ref{Lemma: Compact bi-index} can be extended to yield a characterization of coherent groupoids, but a better one is desirable, as is a better characterization of semantic groupoids). For another example, \cite{preprint:subgpds} supplies a proof of the (known but apparently unpublished) fact that a subgroupoid inclusion induces an inclusion of subtoposes and uses this to derive a topological characterization of the definable subsets of spaces of models.
     %Since subtoposes of a classifying topos \classtop\ corresponds to quotient theories of \theory, this fact can be used in a duality setting to e.g.\ provide a topological characterization of the definable subsets of \theory-models.
     The role of the index set (and the groupoid \thry{S} `of small sets') and the possibility and mechanics of `re-indexing' is also to be further investigated (some steps in this direction are taken in \cite{MathQuart:repforgeothrs} where semantic groupoids are taken to be certain groupoids over \thry{S}). Yet another direction of research is to drop the restriction to theories with enough models and to topological, as opposed to localic, groupoids.

\item Instead of considering sheaves on the topological groupoid of models and isomorphisms of a theory, one can play the same game with the topological category of models and homomorphisms. For instance,  \cite{preprint:repforregthrs} shows that the classifying topos of a regular theory can be represented as equivariant sheaves on the topological category of models and homomorphisms, and relates this to Makkai's results in \cite{makkai:90}. This raises the same questions for topological categories as above for topological groupoids, such as what morphisms of topological categories induce inclusions of their respective equivariant sheaf toposes.

\item
 %Much of the purpose and point of the duality theory is of course to develop insights and tools relevant for the study of logical theories and their models. Makkai proves, for example,
 In \cite{makkai93} Makkai proves the descent theorem for Boolean pretoposes using the duality theory presented there. It is to be determined whether a shorter proof can be given using the constructions of this paper. It is also being explored to what extent notions and problems of classical model theory can be given a fruitful formulation in this setting. It is worth pointing out, in this context, that the `logical topology' of Definition \ref{Definition> Logical topology} has a long history (see \cite{grzemostryll:61} or the expository \cite{hjort:groupactionsandmodels}) (and perhaps  also to note the similarities between the construction in this paper of the generic topos model in equivariant sheaves on the groupoid of models and the definition of the Polish $S_{\infty}$-space of countable models in e.g.\ \cite{hjort:groupactionsandmodels}, where $S_{\infty}$ is the group of permutations on \thry{N} with pointwise convergence topology).
\item Regarding the groupoid of models $\thry{G}_{\theory}$ as the spectrum of the theory \theory, it is natural to look for a structure sheaf on $\thry{G}_{\theory}$  which represents \theory\ as global sections.  Indeed, when \theory\ is coherent, one can form such a structure sheaf $\tilde{\theory}$, roughly by taking $\tilde{\theory}(\bopen{A}) = \synt{C}{T}/A$, the slice category over the object $A$ in \synt{C}{T} determining the basic open \bopen{A} of $\thry{G}_{\theory}$. This sheaf is equivariant, and its equivariant global sections form a category equivalent to \synt{C}{T}.  The stalk of $\tilde{\theory}$ at a model \alg{M} is the complete diagram of \alg{M}.  This will be presented in \cite{BreinerAwodey:2012}.
\end{enumerate}

\section*{Acknowledgements}
 The second author was supported by the Eduard \v{C}ech Center for Algebra and Geometry (grant no.\ LC505) during parts of the research presented here. Both authors take the opportunity to thank Lars Birkedal and Dana Scott for their support.

%\newpage
%\section{Syntax-Semantics Duality}
%use coherent groupoids. State ind. characterization of object classifier
%
%25 pages i total

%%%%%%%%%%%%%%%%%%%%%%%%%%%%%%%%%%%%%%%%%%%%%%%%%%%%%%%%%%%%%%%%%%%%%%%%%%%%%%%%%%%%%%%%%%%%%%%%%%%%%%%%%%%%%%%%%%%

%% BIBLIOGRAPHY INFO
%delete the nocites? Except stone, perhaps
%\nocite{makkai:87}

%\nocite{butz:96}

%\nocite{butz:98}

%\nocite{johnstone:82}

%\nocite{maclane:98}

%\nocite{borceux:94}%delete?

%\nocite{lawvere:63}%delete?

%\nocite{lawvere:63b}%delete?

%\nocite{stone:36}

%% BIBLIOGRAPHY INFO
%%
\bibliographystyle{ieeetr}
\bibliography{bibliografi}
\end{document}